\documentclass[a4paper, 12pt]{amsart}

\usepackage[utf8]{inputenc}
\usepackage{amsmath,amsfonts,amssymb,amsopn,amscd,amsthm}
\usepackage{comment}
\usepackage{dsfont}
\usepackage{graphicx}
\usepackage{color}
\usepackage[colorlinks]{hyperref}
\usepackage{epigraph}
\usepackage{todonotes}
\usepackage[left=3.5cm,top=2.5cm,bottom=2cm,right=3cm]{geometry}
\usepackage{tikz-cd}

\title[Gaussian Multiplicative Chaos and Beta Ensembles match exactly]
      { On the circle, Gaussian Multiplicative Chaos and Beta Ensembles match exactly}

\author{Reda \textsc{Chhaibi}}
\address{R. C. : Universit\'e C\^ote d'Azur, CNRS, LJAD\\
Parc Valrose\\
06108 Nice Cedex 02, France}
\email{{\tt reda.chhaibi@univ-cotedazur.fr}}

\author{Joseph \textsc{Najnudel}} 
\address{J. N. :Bristol University -- University Walk, Clifton, Bristol, United Kingdom}
\email{{\tt joseph.najnudel@bristol.ac.uk}}

\date{}

\allowdisplaybreaks[4]

\DeclareRobustCommand{\SkipTocEntry}[5]{}

\DeclareMathOperator{\tr}{tr}

\DeclareMathOperator{\Hc}{ \textrm{ch} }
\DeclareMathOperator{\eqlaw}{\stackrel{\Lc}{=}}

\DeclareMathOperator{\Var}{Var}

\DeclareMathOperator{\id}{id}

\def\half{\frac{1}{2}}

\def\1{{\mathbf 1}}

\def\e{\textrm{\bf e}}

\def\N{{\mathbb N}}
\def\Z{{\mathbb Z}}
\def\Q{{\mathbb Q}}
\def\R{{\mathbb R}}
\def\C{{\mathbb C}}
\def\H{{\mathbb H}}

\def\D{\mathbb{D}}
\def\V{\mathbb{V}}

\def\P{{\mathbb P}}
\def\E{{\mathbb E}}
\def\F{{\mathbb F}}

\def\Cc{{\mathcal C}}

\def\Ec{{\mathcal E}}
\def\Fc{{\mathcal F}}
\def\Gc{{\mathcal G}}
\def\Hc{{\mathcal H}}

\def\Lc{{\mathcal L}}
\def\Mc{{\mathcal M}}
\def\Nc{{\mathcal N}}
\def\Oc{{\mathcal O}}

\newtheorem{thm}{Theorem}[section]
\newtheorem{proposition}[thm]{Proposition}
\newtheorem{corollary}[thm]{Corollary}
\newtheorem{question}[thm]{Question}

\newtheorem{lemma}[thm]{Lemma}

\newtheorem{rmk}[thm]{Remark}

\numberwithin{equation}{thm}

\begin{document}

\begin{abstract}
We identify an {\it equality} between two objects arising from different contexts of mathematical physics: Kahane's Gaussian Multiplicative Chaos ($GMC^\gamma$) on the circle, and the Circular Beta Ensemble $(C\beta E)$ from Random Matrix Theory. This is obtained via an analysis of related random orthogonal polynomials, making the approach spectral in nature. In order for the equality to hold, the simple relationship between coupling constants is $\gamma = \sqrt{\frac{2}{\beta}}$, which we establish when $\gamma \leq 1$ or equivalently $\beta \geq 2$. This corresponds to the sub-critical and critical phases for $GMC^{\gamma}$.

As a side product, we answer positively a question raised by Vir\'ag, on the fractal spectrum of a random measure constructed from the $C \beta E$.  We also give an alternative proof of the Fyodorov-Bouchaud formula concerning the total mass of $GMC^\gamma$ on the circle. This conjecture was recently settled by Remy using Liouville conformal field theory. We can go even further and give an explicit description of the Fourier coefficients of the random measure $GMC^\gamma$ in terms of independent Beta random variables. 
Furthermore, we notice that the ``spectral construction'' has a few advantages. For example, the Hausdorff dimension of the support is efficiently described for all $\beta>0$, thanks to existing spectral theory. Remarkably, the critical parameter for $GMC^\gamma$ corresponds to $\beta=2$, where the geometry and representation theory of unitary groups lie.
\end{abstract}

\maketitle

\hrule
\tableofcontents
\hrule
\newpage

\addtocontents{toc}{\SkipTocEntry}
\section*{Notation}
\begin{itemize}
 \item $\D := \left\{ z \in \C \ | \ |z| < 1 \right\}$ is the open unit disc, $\partial \D$ denotes its boundary, which is the unit circle. 
 \item A measure $\nu$ on $\partial \D$ induces a linear form on the space of continuous functions on the circle. Thus its evaluation against $f$ will be denoted $\nu(f)$.
 \item Equality in law between the random variables $X$ and $Y$ is denoted by $X \stackrel{\Lc}{=} Y$.
 \item The Vinogradov symbol $\ll$ is equivalent to the $\Oc$ notation: $f \ll g \Leftrightarrow f = \Oc(g)$. Moreover, 
 in all the computations of the article, we allow the implicit constant to depend on the parameter $\beta$.  If the implicit constant depends on other quantities, they will be 
indicated by subscripts: $f \ll_{x} g$ means that there exists $C$ depending only on $x$ and $\beta$ such that $|f| \leq C g$. 
\item All the random objects we consider in the paper are defined on a measurable space $(\Omega, \mathcal{B})$. When changes of probability measures are not involved, the underlying probability measure is  $\P$, and the symbol $\E$ denotes the expectation under $\P$.
\end{itemize}

\section{Introduction}
The relationship we are pointing out between Gaussian Multiplicative Chaos and Random Matrices Theory is best expressed in terms of the classical theory of orthogonal polynomials on the unit circle. As such, we start by recalling a few facts on the topic.

\subsection{Orthogonal Polynomials on the Unit Circle (OPUC)} Consider a probability measure $\mu$ on the unit circle $\partial \D$, $\D$ being the unit disc. By applying the Gram-Schmidt orthogonalization procedure to monomials $\{ 1, z, z^2, \dots \}$, one obtains a sequence $\left( \Phi_n \right)_{n \geq 0}$ of OPUC which satisfies the Szeg\"o recurrence:
\begin{align}
\label{def:matrix_szego}
\begin{pmatrix}
\Phi_{k+1}  (z) \\
\Phi_{k+1}^*(z)
\end{pmatrix}
= & 
\begin{pmatrix}
z            & - \overline{\alpha}_k \\
- \alpha_k z & 1
\end{pmatrix}
\begin{pmatrix}
\Phi_{k}  (z) \\
\Phi_{k}^*(z)
\end{pmatrix} \ ,
\end{align}
where 
$$\Phi^*_{n} (z) := z^n \overline{\Phi_n(1/\bar{z})}.$$
The Szegö recurrence is the analogue of the three term recurrence for orthogonal polynomials on the line $\R$. The coefficients $\alpha_j$ belong to the closed disc, $\overline{\D}$, and are called Verblunsky coefficients. If a measure $\mu$ determines the Verblunsky coefficients, the converse is also true (see \cite[Theorem 1.7.11 p.97]{Sim05-1}):

\begin{thm}[Verblunsky's theorem]
\label{thm:verblunsky}
Let $\Mc_1(\partial \D)$ be the simplex of probability measures on the circle, endowed with the weak topology, and let 
$$\mathcal{D} :=  \D^{\N} \sqcup \left( \sqcup_{n\in \Z_+} \D^{n} \times \partial \D \right)$$
 be endowed with the topology related to the following notion of convergence: a sequence $(A_p)_{p \geq 1}$ in $\mathcal{D}$ converges to an element $A_{\infty} = (\alpha_j)_{0 \leq j < K}$ with finitely many or infinitely many components ($K$ finite or infinite) if and only if for all $j < K$, 
the coefficient of order $j$ of $A_p$ is well-defined for $p$ large enough and converges to $\alpha_j$.
Then, the  map
$$
\begin{array}{cccc}
 \V: & \Mc_1(\partial \D) & \rightarrow & \mathcal{D} \\
\end{array}
$$
given by the sequence of Verblunsky coefficients 
is a homeomorphism. Atomic measures with $n$ atoms have $n$ Verblunsky coefficients, the last one being of modulus one, other measures have infinitely many Verblunsky coefficients. 
\end{thm}

If $Leb$ is the Lebesgue measure on $\partial \D$, then $\V(Leb) = \left( 0, 0, \dots \right)$. In fact, the tangent map of the Verblunsky map, at the point $Leb$, gives exactly the Fourier coefficients of the perturbation. Hence the Verblunsky map is inherently spectral in nature and Verblunsky coefficients can be seen as non-linear Fourier coefficients.

Now let us introduce the random matrix model of interest.

\subsection{The Circular Beta Ensemble (\texorpdfstring{$C \beta E$}{CBE})} For this paragraph, $\beta>0$ plays the role of a coupling constant. Consider $n$ points on the unit circle whose probability distribution is:

\begin{align}
\label{eq:def_CBE}
(C\beta E_n)
\quad &
\frac{1}{Z_{n, \beta}} \prod_{1 \leq k < l \leq n}\left| e^{i \theta_k} - e^{i \theta_l} \right|^\beta d\theta.
\end{align}

For $\beta = 2$, we recognize Weyl's integration formula for central functions on the unitary group $U(n)$. In this case the $C\beta E_n$ reduces to an ensemble known as the Circular Unitary Ensemble ($C U E_n$). It is nothing but the distribution of the eigenvalues of a Haar distributed random matrix. Naturally, the study of this case is very rich in the representation theory of unitary groups (see for example Diaconis-Shahshahani \cite{DS94} and Bump-Gamburd \cite{BG06}).

For general $\beta>0$, the representation-theoretic picture is more complicated: $C \beta E_n$ is the orthogonality measure for Jack polynomials in $n$ variables (see Appendix \ref{section:appendix_CBE} and the references therein). In turn, Jack polynomials are also intimately related to representation theory via rational Cherednik algebras \cite{DG10}. Our point of view will be more direct. From the work of Killip and Nenciu \cite{KN04}, the characteristic polynomial
of $C \beta E_n$ 
can be realized as the last term of the Szeg\"o recurrence, whose distribution of the  Verblunsky coefficients is explicitly given.
More precisely, let $(\alpha_j)_{0 \leq j \leq n-2}$ and $\eta$ be independent complex random variables, 
such that  $\eta$ is uniform on the unit circle, and for 
$0 \leq j \leq n-2$, 
  $\alpha_j$ is rotationally invariant and $|\alpha_j|^2$ is a Beta random variable with parameters $1$ and $\beta_j := \frac{\beta (j+1)}{2}$:
\begin{align}
\label{eq:beta_density} 
\P\left( |\alpha_j|^2 \in dx \right) = \beta_j \left( 1-x \right)^{ \beta_j - 1} \mathds{1}_{\left\{ 0 < x < 1 \right\} } dx \ .
\end{align}

In passing, let us record the following equalities: 

\begin{align}
\label{eq:beta_variance} 
\E\left( |\alpha_j|^2 \right) = \frac{1}{1+\beta_j} \ 
\end{align}
and 
\begin{align}
\label{eq:log_beta} 
\E\left[ -\log\left( 1 - |\alpha_j|^2 \right) \right] = \frac{1}{\beta_j}. \ 
\end{align}
 Killip and Nenciu prove in \cite[Theorem 1]{KN04} and \cite[Proposition 4.2]{KN04} that 
$$
  \V^{-1}\left( \alpha_{n-2}, \dots, \alpha_{1}, \alpha_0, \eta \right)
= \sum_{j=1}^n \widetilde{\pi}_j \delta_{\tilde{\theta}_j^{(n)}}(d\theta) \ ,
$$
where the weights $\left( \widetilde{\pi}_j \right)_{1 \leq j \leq n}$ have a $\beta$-Dirichlet distribution and the support $(\tilde{\theta}_j^{(n)})_{1 \leq j \leq n}$ is independently distributed according to the $C \beta E_n$ given in \eqref{eq:def_CBE}. 

Moreover, thanks to \cite[Proposition B.2]{KN04}, reversing the order of Verblunsky coefficients, except the last one $\eta$, changes the weights but preserves the distribution of the support.

From this property, a fruitful idea consists in using the reversed order of Verblunksy coefficients and incorporating the weights in the definition of the $C \beta E_n$. Therefore, we {\it redefine} the Circular $\beta$ Ensemble with $n$ points as the random probability measure:
\begin{align}
\label{eq:def_CBE2}
C\beta E_n := &
  \V^{-1}\left( \alpha_0, \alpha_1, \dots, \alpha_{n-2}, \eta \right)
= \sum_{j=1}^n \pi_j \delta_{\theta_j^{(n)}}(d\theta) \ .
\end{align}
The support points $\left( \theta_j^{(n)} \right)_{1 \leq j \leq n}$ are the arguments of the zeroes of the last orthogonal polynomial 
$X_n := \Phi^*_n (\alpha_0, \alpha_1, \dots, \alpha_{n-2}, \eta)$ associated to the finite sequence 
of Verblunsky coefficients $\left( \alpha_0, \alpha_1, \dots, \alpha_{n-2}, \eta \right)$:  we have 
$$X_n (z) = \prod_{j=1}^n ( 1 - z e^{i  \theta_j^{(n)}})$$
and $(e^{i  \theta_j^{(n)}})_{1 \leq j \leq n}$  are distributed as in \eqref{eq:def_CBE}. Nevertheless, the distribution of the weights $\left( \pi_j \right)_{1 \leq j \leq n}$ is not known explicitly, with a tractable form.

With the definition \eqref{eq:def_CBE2}, a remarkable fact is that the sequence of Verblunsky coefficients is \textit{consistent}. 
Indeed, even if $C\beta E_n$ and $C\beta E_{n+1}$ have a priori no reason for living on the same probability space, it is possible to couple them in such a way that the $n-1$ first Verblunsky coefficients are exactly the same. This provides a way to couple the characteristic polynomial of  $C \beta E_n$ for all values of $n \geq 1$: if $(\alpha_j)_{j \geq 0}$ is an {\it infinite} sequence of independent variables whose distribution is given as above, and if $\eta$ is an independent variable, uniform on the unit circle, then the last orthogonal polynomial given by the sequence of Verblunsky coefficients $(\alpha_0, \dots, \alpha_{n-2}, \eta)$ has the same law as the $C \beta E_n$ for all $n \geq 1$. With this particular coupling, the Verblunsky coefficients provide a sequence of random measures indexed by $n$, supported by the points of the $C \beta E_n$, and tending to a limiting random measure $\mu^{ \beta}$, whose Verblunsky coefficients are $(\alpha_j)_{j \geq 0}$. 

In light of Verblunsky's Theorem \ref{thm:verblunsky}, this remark begs the question:

\begin{question}
\label{question:whatIsMu}
Is there anything remarkable or canonical about the projective limit
 $$ \varprojlim C\beta E_n
 := \V^{-1}\left( \alpha_0, \alpha_1, \alpha_2, \dots \right)
  = \mu^\beta \ ,$$
 obtained from using all Verblunsky coefficients? Does this measure arise in other circumstances?
\end{question}

 Before discussing this question, it is worth explaining why the points $C \beta E_n$ can be seen as quadrature points of the infinite random measure $\varprojlim C \beta E_n = \mu^\beta$. Any sequence of measures, indexed by $n$, whose $n-1$ first Verblunsky coefficients match the $n-1$ first elements of the sequence $\left( \alpha_0, \alpha_1, \dots \right)$ will converge to $\mu^\beta$, in the topology of weak convergence. Moreover, if we assume that the Verblunsky coefficients are $(\alpha_0, \dots, \alpha_{n-2}, \eta)$  with $|\eta|=1$, then the approximating measure is atomic, supported by $n$ points.  The general theory of orthogonal polynomials dictates that for all polynomials $P$ of degree $\deg P \leq n-1$:
$$ \int_{\partial \D} P \ \mu^\beta = \sum_{j=1}^n \pi_j P( e^{i \theta_j^{(n)}} ) \ .$$
In the language of approximation theory, that is exactly to say that $\left( \pi_j \right)_{1 \leq j \leq n}$ are (random) quadrature weights and that the $n$ of points $C \beta E_n$ can be seen as the $n$ (random) quadrature points for the (random) measure $\mu^\beta = \varprojlim C \beta E_n$.

\medskip

\subsection{The Gaussian Multiplicative Chaos (\texorpdfstring{$GMC^\gamma$}{GMC})}
\label{subsection:GMC}
In this paragraph, $\gamma>0$ plays the role of coupling constant in an a priori different context. Define the Gaussian field on the unit disc:
\begin{align}
\label{eq:def_G}
G(z) := & \ 2 \Re \sum_{k=1}^\infty \frac{z^k}{\sqrt{k}} \Nc_k^\C  
\end{align}
where $(\Nc_k^\C)_{k \geq 0}$ denote i.i.d. complex Gaussian variables, such that 
$$\E \left[ (\Nc_k^\C)^2 \right] = \E \left[ \Nc_k^\C \right] = 0, \;
  \E \left[ |\Nc_k^{\C}|^2 \right] = 1 \ .$$
One can establish that:
\begin{itemize}
 \item $\textrm{Cov}(G(w), G(z) ) = -2 \log \left| 1 - w \bar{z} \right| \ .$ 
 \item The field can be extended to the closed unit disc $\overline{\D}$ but its restriction to the circle is not a function. In fact, $G_{|\partial \D}$ is almost surely a random Schwartz distribution in $\cap_{\varepsilon > 0} H^{-\varepsilon}(\partial \D)$ where the Sobolev spaces are given for all $s \in \R$ by:
 $$ H^{s}(\partial \D) := \left\{ f \ \Big | \ \sum_{n\in \Z} (1+ |n|)^s |\widehat{f}(n)|^2 < \infty \right\} \ .$$
 \item  Because $G$ is harmonic, $G(re^{i\theta}) = (G_{|\partial \D} * P_r) \left( e^{i \theta} \right)$ where $*$ denotes convolution and $P_r$ is the Poisson kernel. 
\end{itemize}

We can define the measure 
\begin{align}
\label{eq:gmc_r_def}
GMC_r^\gamma(f) := & 
   \int_{\partial \D}
   \frac{d\theta}{2\pi}
   f( e^{i\theta} )
   \exp\left( \gamma G(r e^{i \theta}) - \half \gamma^2 \Var(G(r e^{i \theta})) \right) \\
= &
   \nonumber
   \int_{\partial \D}
   \frac{d\theta}{2\pi}
   f( e^{i\theta} )
   e^{ \gamma G(r e^{i \theta}) } \left( 1-r^2 \right)^{\gamma^2}
   \ . 
\end{align}

The Gaussian Multiplicative Chaos with coupling constant $\gamma>0$ is the weak limit:
\begin{align}
\label{eq:gmc_limit}
GMC^\gamma := & \lim_{r \rightarrow 1} GMC_r^\gamma \ .
\end{align}
To be exact, the above limit holds in probability, upon integrating against continuous functions. The existence of such a limit for all $\gamma>0$ is well-established via standard regularization techniques such as convolution or Karhunen-Lo\`eve expansions of Gaussian processes \cite{RV13, B17}. The literature treats higher dimensions and different geometries as well. Of course, this includes our particular case of convolution by the Poisson kernel. However, there are different regimes regarding the limit \eqref{eq:gmc_limit}:
\begin{itemize}
 \item $\gamma<1$, Sub-critical phase. $GMC^\gamma$ is a non-degenerate random measure, which can be seen from the following $L^1$ convergence.

 \begin{thm}[Theorem 1.2 in \cite{B17}]
\label{thm:GMC_CV}
For all nonnegative, smooth functions $f$, and for $\gamma<1$, i.e. in the sub-critical regime:
$$ GMC^\gamma_r(f) \stackrel{r \rightarrow 1}{\longrightarrow} GMC^\gamma(f) \ , $$
the convergence being in probability and in $L^{1}\left( \Omega, \mathcal{B}, \P \right)$. 
\end{thm}

 \item $\gamma=1$, Critical phase. The limit in \eqref{eq:gmc_limit} is the trivial zero measure, however one can perform different normalizations in order to obtain the so-called critical GMC. A random renormalization via the so-called derivative martingale has been implemented in \cite{DRSV14}, while the Seneta-Heyde renormalization has been implemented in \cite{JS17}. Both constructions agree \cite{P18}. Moreover,  Aru, Powell and Sepulveda \cite[Section 4.1]{APS18} have proven that the critical GMC can be written as the limit
of the subcritical GMC when the parameter tends to $1$ from below. This allows us to bootstrap the construction of the sub-critical GMC and obtain the critical GMC  via the limit in probability: 
\begin{align}
\label{eq:RenormalizeGMC}
GMC^{\gamma=1} = \lim_{\gamma \rightarrow 1^-} \frac{GMC^\gamma}{1-\gamma} \ ,
\end{align}
when the random measures $ GMC^{\gamma}$ are constructed from the same field $G$ for all values of $\gamma \in (0,1)$. The critical GMC is known to be non-atomic and it has been conjectured to assign full measure to a random set of Hausdorff dimension zero (see the overview section of \cite{DRSV14}). 
 
 As a corollary of our main result (Corollary \ref{corollaryxxx}) we shall see that this latter conjecture holds, in the context of the circle.
 
 \item $\gamma>1$, Supercritical phase. In this case, there are two constructions resulting in different measures.
 
 A first point of view consists in noticing that the renormalization of Eq. \eqref{eq:gmc_r_def} by a factor $(1-r^2)^{\gamma^2}$ is too strong, and the limit \eqref{eq:gmc_limit} is the zero measure. One needs a different renormalization procedure so that a non-trivial limit holds. The correct normalization at the exponential scale is given by the precise asymptotic behavior of the maximum $\max_{\theta \in \R} G(r e^{i \theta})$ as $r \rightarrow 1^-$. As such, one naturally expects the limit to be atomic, giving mass to the Gaussian field's maxima.
 This was done in \cite{MRV16}. With such a construction, the $\gamma>1$ regime is called the glassy phase and the transition is referred to as a freezing transition. The term "freezing" comes from the fact that the logarithm of the total mass of the measure behaves linearly in $\gamma$ because of the new renormalization. All in all, the result is that the limiting measure can be described as follows: one starts with the critical GMC, and conditionally on the corresponding random measure $GMC^{\gamma=1}$, one takes a strictly positive stable noise of scaling exponent $\frac{1}{\gamma}$ and intensity $GMC^{\gamma=1}$. In loose terms, in the supercritical regime, one only sees Dirac masses corresponding to the extrema of the underlying Gaussian field, and which are ``sprinkled'' on the circle with an intensity depending on the critical measure. 
 
Another version of the supercritical Gaussian multiplicative chaos has been previously constructed in \cite{BJRV13} by taking a {\it subcritical} GMC with coupling constant $\gamma' = \frac{1}{\gamma}$, as the intensity of a stable noise of scaling exponent $\frac{1}{\gamma^2}$.  We use a different normalization, hence extra factors $2$ in \cite{BJRV13}. The constructed measure is named the KPZ dual measure. As explained in that paper, the name stems from the relationship to the KPZ formula and its symmetry with respect to the transform $\gamma \mapsto \frac{1}{\gamma}$. This last construction cannot be naturally recovered from a logarithmically correlated Gaussian field on the circle without adding some extra randomness, contrarily to the construction of \cite{MRV16} with a freezing transition. Nevertheless, the KPZ dual measure seems to have better analyticity properties than the construction with a freezing transition. We will make further remarks on the topic at the end of the next section.
  
%
\end{itemize}
 
\section{Main result and consequences}

The Main Theorem of the present article provides a direct link between the a priori unrelated objects introduced in the previous section: namely, it shows that up to a suitable normalization, the random measure  $\varprojlim C\beta E_n$ and the Gaussian multiplicative chaos of parameter $\gamma := \sqrt{ \frac{2}{\beta} }$ have the same distribution in the sub-critical and the critical cases, i.e. for $\beta \geq 2$.

Notice that the construction of $\varprojlim C\beta E_n$ bypasses the phase transition involved in the definition of the GMC, since the description in terms of Verblunsky coefficients is uniform for all values of $\beta > 0$. However, we do not exactly know how the two random measures $C \beta E_{\infty}$ and $GMC^{\gamma}$ are related in the supercritical case. 

The precise statement of the main result of the article is the following: 

\begin{thm}[Main Theorem - $GMC^\gamma = \varprojlim C\beta E_n$]
\label{thm:main}
For $\beta \geq 2$, let $(\alpha_j)_{j \geq 0}$ be a sequence of independent, rotationally invariant complex-valued random variables, 
such that $|\alpha_j|^2$ is Beta-distributed with parameters $1$ and $\beta_j = \frac{\beta}{2}(j+1)$. Let $\mu^{\beta}$ be the random probability measure whose Verblunsky coefficients are given by the sequence $(\alpha_j)_{j \geq 0}$, and 
let
$$ C_0 
 := \left\{\begin{array}{cc}
          \prod_{j=0}^{\infty} \left( 1 - \left|\alpha_j\right|^2 \right)^{-1} \left( 1 - \frac{2}{\beta(j+1)} \right) & \textrm{if $\beta>2$}\\
        2  \left( 1 - \left|\alpha_0\right|^2 \right)^{-1}
         \prod_{j=1}^{\infty} \left( 1 - \left|\alpha_j\right|^2 \right)^{-1} \left( 1 - \frac{2}{\beta(j+1)} \right) & \textrm{if $\beta=2$} \ .
          \end{array}
   \right.$$
Then, the product of $C_0$ by the measure $\mu^{\beta}$ has the same law as the measure corresponding to the Gaussian multiplicative 
chaos $GMC^\gamma$, with parameter $\gamma = \sqrt{\frac{2}{\beta}} \leq 1$.
In particular, $\mu^\beta$ has the same law as $GMC^\gamma$, renormalized into a probability measure, and the total mass of 
$GMC^\gamma$ has the same law as $C_0$.
\end{thm}

\begin{proof}[General structure of proof]
First, the result can be bootstrapped quite easily from the sub-critical phase to the critical phase by using  \eqref{eq:RenormalizeGMC}. Therefore, it is enough to deal only with the sub-critical phase $\beta >2$ ($\gamma < 1$). Let us sketch the general ideas of the proof. A finer description of the structure of the paper is given at the end of the section.

In the sub-critical phase, the main idea of the proof is the following. We consider the sequence of orthogonal polynomials $(\Phi^*_n)_{n \geq 0}$ associated to the Verblunsky coefficients $(\alpha_j)_{j \geq 0}$. A general theorem in OPUC theory, due to Bernstein and Szeg\"o, implies that $\mu^{\beta}$ is the limit, when $n$ goes to infinity, of the probability measure $\mu_n^\beta$ whose density at $e^{i \theta}$ is proportional to $|\Phi^*_n(e^{i \theta})|^{-2}$. On the other hand, one can show that $(\Phi^*_n)_{n \geq 0}$ almost surely converges, uniformly on compact sets of the unit disc $\mathbb{D}$, to a limiting random holomorphic function $\Phi^*_{\infty}$, which is the exponential of a logarithmically correlated Gaussian field: see Proposition \ref{proposition31}. From the precise form of the correlation of this field, we deduce that the regularization of the Gaussian multiplicative chaos can be written as:
$$ GMC_r^{\gamma=\sqrt{\frac{2}{\beta}}}(d\theta)
   = (1-r^2)^{\frac{2}{\beta}} |\Phi^*_{\infty}(re^{ i \theta})|^{-2} d\theta \ .$$
To prove the Main Theorem, it is then enough to show that up to a delicate issue of renormalization, the limit of the measure $|\Phi^*_n(e^{i \theta})|^{-2} d \theta$ when $n$ goes to infinity (the measure $\mu^{\beta}$) is the same as the limit of the measure $|\Phi^*_{\infty} (r e^{i \theta})|^{-2} d \theta$ when $r$ goes to $1$ from below (the Gaussian multiplicative chaos). In other words, up to a suitable normalization, the two limits $n \rightarrow \infty$ and $r \rightarrow 1^-$ commute when we  start with the measure  $|\Phi^*_n(r e^{i \theta})|^{-2} d \theta$. One can sketch the following diagram:

$$
\begin{tikzcd}
\mu_{n,r}^{\beta}(d\theta) \propto \frac{1}{|\Phi^*_n(r e^{i \theta})|^{2}} d\theta \arrow[d, "r \rightarrow 1"] \arrow[r, "n \rightarrow \infty"] &  \mu_{r}^{\beta}(d\theta) \propto GMC_r^{\gamma=\sqrt{\frac{2}{\beta}}}(d\theta) \arrow[d, "r \rightarrow 1"] \\
\mu^\beta_n(d\theta) \propto \frac{1}{|\Phi^*_n(e^{i \theta})|^{2}} d\theta \arrow[r, "n \rightarrow \infty"]
& \mu^\beta \Big\backslash \frac{1}{C_0} GMC^{\gamma=\sqrt{\frac{2}{\beta}}}(d\theta) 
\end{tikzcd} \ ,
$$

where the symbol $\propto$ stands for "proportional to", the multiplicative factor being a random variable. In the end, the proof boils down to tracking the exact behavior of these factors.
\end{proof}

The diagram above shows in particular that the subcritical GMC is the limit of a suitable normalization of the measure $|\Phi^*_n(e^{i \theta})|^{-2} d \theta$ when $n$ goes to infinity. 
It is reasonable to expect a similar convergence for powers of $|\Phi^*_n(e^{i \theta})|$ with more general exponents. However, the techniques of the present paper do not directly apply to this case since
the result by Bernstein and Szeg\"o crucially depends on the fact that we consider a power of exponent $-2$. 
It is also natural to conjecture a convergence to the GMC when $\Phi^*_n(e^{i \theta})$ is replaced by the characteristic polynomial $X_n$ of the $C \beta E_n$, since these 
two polynomials are very strongly related: the only difference between the two settings is that the squared modulus of the  last Verblunsky coefficient $\alpha_{n-1}$ associated
to $\Phi^*_n$ is  distributed as a 
Beta random variable with parameters $1$ and $\beta n/2$, whereas the last Verblunsky coefficient used in the construction of $X_n$ is uniform on the unit circle. 
This change of the last coefficient has non-trivial consequences, in particular for the behavior of the polynomial on the unit circle: for example, 
$X_n$ has zeros on the unit circle, whereas all zeros of 
$\Phi^*_n$ are outside the unit disc. Since $X_n$ has zeros on the unit circle, the measure with density $|X_n|^{\alpha}$ with respect to the Lebesgue measure
is infinite as soon as $\alpha \leq -1$, and in particular for the exponent  $\alpha = -2$ considered in the present article. 

In \cite{W15}, Webb proves that if $X_n$ is the characteristic polynomial of the CUE (i.e. $\beta = 2$), and if $\alpha \in (-1/2, \sqrt{2})$, then the random measure on the unit circle with density  $|X_n(e^{i \theta})|^{\alpha}/ \mathbb{E} [ |X_n(e^{i \theta})|^{\alpha}]$ converges in law to $GMC^{|\alpha|/2}$ 
when $n \rightarrow \infty$. This result has been extended to $\alpha \in [-1/2, 2)$ by 
 Nikula, Saksman and Webb in \cite{NSW18}. For the $C \beta E$ with general $\beta > 0$, convergence to the GMC has been proven 
in the subcritical phase by Lambert in  \cite{L19}, in the case where we take the polynomial inside the unit disc, at a small mesoscopic distance ($n^{-1} (\log n)^6$) from the unit circle.
In these results, the convergence which occurs is a convergence in law: the matrix ensembles for different dimensions $n$ are a priori not defined on the same probability space, contrarily to the setting of the present article, where the polynomials $(\Phi^*_n)_{n \geq 0}$ are constructed from a single infinite sequence of Verblunsky coefficients. 

So far, we believe that the equality between $GMC^{\gamma}$ and $C \beta E_{\infty}$ can be extended to the supercritical regime, possibly after suitable adjustements. We 
expect that the measure $\mu_{\beta}$ for $\beta < 2$ is closely related to the KPZ dual measure constructed in \cite{BJRV13} for $\gamma = \sqrt{2/ \beta}$, which
seems to have good properties of analyticity with respect to the parameter $\gamma$. 
However, proving this link does not seem to be straightforward and is beyond the scope of this paper. 

\begin{rmk}[The splitting phenomenon]
Just like the characteristic polynomial from $C \beta E_n$ evaluated at one point is a product of independent random variables (see \cite{BHNY}) so is the total mass $C_0$. 

As explained in \cite{BHNY}, the splitting phenomenon for the characteristic polynomial is the probabilistic manifestation of the product formula for the (circular) Selberg integrals. It will be apparent in the proof of the Fyodorov-Bouchaud formula (Corollary \ref{corollary:fyodorovBouchaud}), that the splitting of the total mass is the probabilistic manifestation of another product formula,   related to the $\Gamma$ function.
\end{rmk}
Before diving into technical considerations, let us provide a few corollaries.

\addtocontents{toc}{\SkipTocEntry}

\subsection{The law of the Verblunsky coefficients of the GMC}
The Main Theorem gives a way to construct the Gaussian multiplicative chaos from a sequence of Verblunsky coefficients. We can think of it in the reverse way: 
\begin{corollary}
For $\gamma \leq 1$, let $(\alpha_j)_{j \geq 0}$ be the Verblunsky coefficients associated to the random probability measure 
obtained by dividing $GMC^{\gamma}$ by its total mass. Then, the random variables $(\alpha_j)_{j \geq 0}$ are independent, rotationally invariant in distribution, and $|\alpha_j|^2$ is distributed like a Beta variable of parameters $1$ and $\frac{\beta (j+1)}{2}$, for $\beta = \frac{2}{\gamma^2}$. 
Moreover, the total mass of $GMC^{\gamma}$ is given by the formula defining $C_0$ in the Main Theorem. 
\end{corollary}
\begin{proof}
The joint law of total mass and the Verblunsky coefficients associated to $GMC^{\gamma}$ is uniquely determined by the law of this measure, and then it is the same for any other random measure with the same distribution.

In particular, it is the same for the measure $C_0 \mu^{\beta}$ considered in the Main Theorem. Now, by construction, $C_0 \mu^{\beta}$ has Verblunsky coefficients with the desired distribution and its total mass is $C_0$.
\end{proof}

\addtocontents{toc}{\SkipTocEntry}
\subsection{Coupling the \texorpdfstring{$C \beta E$}{CBE} for different \texorpdfstring{$\beta$}{B}}

The different measures $GMC^\gamma$ for different $\gamma>0$ naturally live {\it on the same probability space}, as limits of measures built from the Gaussian field $G$.
From the previous corollary, dividing these measures by their total mass gives a coupling of the measures $C \beta E_{\infty}$ for $\beta \geq 2$ - and therefore a coupling of their Verblunsky coefficients. 
Introducing an independent uniform variable $\eta$  on the unit circle gives a way to deduce a coupling of $C \beta E_n$ for all $n \geq 1$ and all $\beta \geq 2$.

\addtocontents{toc}{\SkipTocEntry}
\subsection{Hausdorff dimension of the support}

At first, the description of Hausdorff dimension of spectral measures for random Schr\"odinger operators was investigated by Kiselev, Last and Simon in \cite{KLS98}. The adaptation to OPUC was made in the book by Simon \cite[Chapter 12]{Sim05-2}.

\begin{corollary} \label{corollaryxxx} 
The Hausdorff dimension of the support of $\mu^\beta = C \beta E_{\infty} $ is given as 
follows: 
\begin{itemize}
 \item If $\beta>2$ (sub-critical), then $\dim_\H supp(\mu^\beta) = 1 - \frac{2}{\beta} $.
 \item If $\beta=2$ (critical), then $\dim_\H supp(\mu^\beta) = 0$ and $\mu^\beta$ is  non-atomic.
 \item If $\beta<2$ (super-critical), then $\mu^\beta$ is atomic. 
\end{itemize}
Since for $\beta \geq 2$,  $\gamma = \sqrt{2/ \beta}$, $\mu^{\beta}$ has the same law as $GMC^{\gamma}$  up to normalization, 
one deduces: 
\begin{itemize}
 \item If $\gamma \in (0,1) $ (sub-critical), then $\dim_\H supp(GMC^{\gamma}) = 1 - \gamma^2$.
 \item If $\gamma = 1$ (critical), then $\dim_\H supp(GMC^{\gamma}) = 0$ and $GMC^{\gamma}$ is non-atomic. 
\end{itemize} 
\end{corollary}
\begin{proof}
Apply \cite[Theorem 12.7.7]{Sim05-2}. Notice that because our Verblunsky coefficients $\left( \alpha_j(\mu^\beta); j \in \N \right)$ are rotation invariant, the Alexandrov measures $\left( \mu^\beta_\lambda; \lambda \in \partial \D \right)$ are in fact all the same in law. Therefore the conclusion of that theorem, holding for almost every $\lambda$, is true for the measure $\varprojlim C \beta E_n$.	

Let us mention that \cite[Theorem 12.7.7]{Sim05-2} depends on the previous \cite[Theorem 12.7.2]{Sim05-2}, whose hypotheses do not exactly match ours. Nevertheless the proofs carry verbatim.
\end{proof}

In Corollary  \ref{corollaryxxx}, the statement on the subcritical GMC has already been proven by different methods (see \cite{RV13}), and the
statement on the critical GMC has been conjectured in \cite{DRSV14}. In the supercritical phase $\gamma > 1$, the "freezing" construction of $GMC^{\gamma}$ given in \cite{MRV16}, and the 
"KPZ dual" construction given in \cite{BJRV13} both provide atomic measures, in agreement with the behavior of $\mu^{\beta}$ for $\gamma = \sqrt{2/\beta}$. 
 It is also worth mentioning that a similar analysis for Gaussian ensembles has been tried in \cite{BFS07}.

\addtocontents{toc}{\SkipTocEntry}
\subsection{The Fyodorov-Bouchaud formula and beyond}
The Main Theorem allows us to easily deduce the distribution of the total mass of the chaos, which has been conjectured by Fyodorov and Bouchaud \cite{FB08}, and 
 recently proven by Remy \cite{R17}, using the partial differential equations satisfied by correlation functions in Liouville conformal field theory: the conformal field theory on the hyperbolic disc uses the GMC as an ingredient.
 
 Moreover, our approach allows to explicitly write the Fourier coefficients of the random  measure $GMC^{\gamma}$, in terms of the Verblunsky coefficients $(\alpha_j)_{j \geq 0}$, 
 whose joint distribution is explicitly known. The following result holds:
 
\begin{corollary}
\label{corollary:fyodorovBouchaud}
In the sub-critical phase and critical phases ($\gamma \leq 1$):
$$ GMC^\gamma\left( \partial \D \right) \ \eqlaw \ K_\gamma \e^{-\gamma^2} \ ,$$
where $GMC^\gamma\left( \partial \D \right)$ denotes the total mass of the $GMC^\gamma$, $\e$ is a standard exponential random variable, and $K_\gamma$ is an explicit constant:
$$ K_\gamma
:= \left\{ \begin{array}{cc}
           \Gamma(1-\gamma^2)^{-1} & \textrm{if $\gamma<1$,}\\
           2                       & \textrm{if $\gamma=1$} \ .
           \end{array} \right.
$$
Moreover, if for $n \geq 1$,  
$$C_n := \int_0^{2\pi} GMC^{\gamma} (e^{i \theta}) e^{-n i \theta} \frac{d \theta}{2 \pi} 
$$
and 
$$c_n := \frac{C_n} {GMC^{\gamma} (\partial \mathbb{D}) },$$
then 
$$c_{n} = \alpha_{n-1} \prod_{j=0}^{n-2} (1 - |\alpha_j|^2) + V^{(n-1)}(\alpha_0, \dots, \alpha_{n-2}, \overline{\alpha_0}, \dots, \overline{\alpha_{n-2}})$$
where $(\alpha_j)_{j \geq 0}$ are the Verblunsky coefficients associated to the measure $GMC^{\gamma}$ and  $V^{(n-1)}$ is an explicitly computable polynomial with integer coefficients. In particular: 
$$ \left\{
   \begin{array}{ccc}
   c_1 & = & \alpha_0,\\
   c_2 & = & \alpha_0^2 + \alpha_1 (1 - |\alpha_0|^2),\\
   c_3 & = & (\alpha_0 - \alpha_1 \overline{\alpha_0}) [ \alpha_0^2 + \alpha_1 (1 - |\alpha_0|^2)] \\
       &   & \quad \quad + \alpha_1 \alpha_0 + \alpha_2 (1 - |\alpha_0|^2) (1 - |\alpha_1|^2).
   \end{array}
   \right.
$$
Since the joint law of $(\alpha_j)_{j \geq 0}$ is known,  these formulas, together with  the formula of Theorem \ref{thm:main} giving the total mass of $GMC^{\gamma}$ in function of $(\alpha_j)_{j \geq 0}$, uniquely determine the joint law of the Fourier coefficients $(c_n)_{n \geq 1}$ and $(C_n)_{n \geq 1}$.
\end{corollary}
\begin{proof}
In the case $\gamma<1$, pick a $z \in \C$, with $\Re(z) \leq 0$. From Theorem \ref{thm:main}:
\begin{align*}
    \E\left( GMC^\gamma\left( \partial \D \right)^{z} \right)
= & \prod_{j=0}^\infty \left( 1 - \frac{2}{\beta(j+1)} \right)^{z} \E\left( \left( 1 - |\alpha_j|^2 \right)^{-z} \right)\\
= & \prod_{j=0}^\infty \left( 1 - \frac{2}{\beta(j+1)} \right)^{z}
                       \frac{\beta_j}
                            {\beta_j - z}\\
= & \prod_{j=0}^\infty \frac{ \left( 1 + \frac{1}{(j+1)} \right)^{ -\frac{2z}{\beta} }
                              \left( 1 - \frac{2z}{\beta(j+1)} \right)^{-1} }
                            { \left[
                              \left( 1 + \frac{1}{(j+1)} \right)^{ -\frac{2}{\beta} }
                              \left( 1 - \frac{2}{\beta(j+1)} \right)^{-1}
                              \right]^z
                            }
                       \\
= & \frac{ \Gamma\left( 1 - \frac{2 z}{\beta} \right)   }
         { \Gamma\left( 1 - \frac{2}{\beta}   \right)^{z} }\ ,
\end{align*}
where on the last line, we used the Weierstrass product formula for the Gamma function. One recognizes that the $\Gamma$ function is the Mellin transform of an exponential, which gives the desired result for $\gamma < 1$. 

The case $\gamma=1$ is handled by taking the limit $\gamma \rightarrow 1^-$ as in \eqref{eq:RenormalizeGMC}:
$$ GMC^{\gamma=1}\left( \partial \D \right) \eqlaw \lim_{\gamma \rightarrow 1^-} \frac{K_\gamma}{1-\gamma} \e^{-\gamma^2} \ .$$
The constant $K_\gamma$ vanishes at $\gamma=1$ and absorbs the renormalization:
\begin{align*}
    \frac{K_\gamma}{1-\gamma}
= & \frac{1}{(1-\gamma) \Gamma(1-\gamma^2)}
=   \frac{1+\gamma}{\Gamma(2-\gamma^2)}
\stackrel{\gamma \rightarrow 1^-}{\longrightarrow} 2 \ .
\end{align*}
This provides the distribution of $GMC^\gamma\left( \partial \D \right)$. 
The expression giving $c_n$ is a consequence of general theory of OPUC: see  \cite[Theorem 1.5.5 p.60]{Sim05-1}), and  \cite{Sim05-1}, formulas (1.3.51), (1.3.52), (1.3.53). 
\end{proof}

 It is possible to compute, when they exist, moments of the Fourier coefficients, i.e. the expectation of products of some powers of the $c_n$'s, their conjugates, and a power of the total mass of the chaos. For example, it is not difficult (but not obvious) to deduce Conjecture 1 of \cite{R17} from our Main Theorem. Similarly, we get immediately
  $$\mathbb{E}[ |c_1|^2] = \frac{1}{1 + \half \beta} = \frac{\gamma^2}{1 + \gamma^2}$$
  where $\gamma = \sqrt{\frac{2}{\beta}}$, which is consistent to the computation of the Edwards-Anderson's order parameter in the circular model of the $\frac{1}{f}$ noise: see formula (7) of \cite{CLD}. One can also recover formula  (42) of \cite{CLD} when $\gamma \leq 1$.

\addtocontents{toc}{\SkipTocEntry}
\subsection{Further remarks}

\medskip

\paragraph{\it On the supercritical phase:} For all $\beta > 0$, with the notation of the Main Theorem, one can define the random measure $C'_0 \mu^{\beta}$, where 
   $$C'_0 = \prod_{j=0}^{\infty} (1- |\alpha_j|^2)^{-1} e^{-\frac{2}{\beta(j+1)}}.$$
  The paper shows that in the subcritical and the critical cases ($\beta \geq 2$), this measure has the same law as the Gaussian multiplicative chaos, times a deterministic constant depending only on $\beta$.  Nevertheless, in the supercritical phase ($\beta <2$), the measure $C'_0 \mu^{\beta}$ is still well-defined, and gives a new way to construct a supercritical Gaussian multiplicative chaos. 
  
  It is natural to ask how this construction can be compared to the very different constructions given in \cite{BJRV13} and \cite{MRV16}, which we described in the introduction (Subsection \ref{subsection:GMC}). Since the quantities involved in the definition of $C'_0 \mu^{\beta}$ are analytic in $\beta$, it is very unlikely that our construction gives the freezing transition appearing in \cite{MRV16}. Therefore, we conjecture that $C'_0 \mu^{\beta}$ is strongly related to the KPZ dual measure of \cite{BJRV13}: the two random measures may have the same law, up to a multiplicative constant. As a corroborating evidence is the fact that the laws of the total masses agree. This is done as follows.

  By Proposition 6 of \cite{BJRV13}, we have, for $\gamma >1$ and $0 \leq \rho < 1/\gamma^2$, with obvious notation: 
 $$ \E\left[  GMC^{\gamma, BJRV} (\partial \mathbb{D})^{\rho} \right]
  = \frac{\Gamma(1- \rho \gamma^2) \Gamma(1- \gamma^{-2})^{\rho \gamma^2}}
         {\Gamma(1-\rho) \gamma^{-2 \rho \gamma^2} }
    \E\left[  GMC^{1/\gamma} (\partial \mathbb{D})^{\rho \gamma^2} \right] \ ,$$
  and, using the Fyodorov-Bouchaud conjecture proved by Remy in \cite{R17} and in a different way in the present article, we obtain:
 $$ \E\left[  GMC^{\gamma, BJRV} (\partial \mathbb{D})^{\rho} \right]
  = \kappa_{\gamma}^{\rho} \Gamma(1- \rho \gamma^2)$$
  for some $\kappa_{\gamma} > 0$ depending only on $\gamma$. 
  Hence, the total mass $GMC^{\gamma, BJRV} (\partial \mathbb{D})$ is the power $-\gamma^2$ of an exponential variable, as the total mass $C'_0$ of the measure 
  $C'_0 \mu^{\beta}$. 
  
  If one considers the construction of \cite{MRV16}, it is seen from the freezing transition that the total mass of the supercritical chaos is not anymore the power $-\gamma^2$ of an exponential random variable.

\medskip

\paragraph{\it On relating the GMC and Random Matrix Theory:}  To the best of the authors' knowledge, the first hints that there should be a relationship between Gaussian multiplicative cascades and Random Matrix Theory have first appeared in the paper \cite{BFS07} and then Vir\'ag's ICM Proceeding of Seoul 2014 \cite{V14}. In both of these references, the focus is on multiplicative cascades on the real line and on tridiagonal models for the $G \beta E$ (Gaussian $\beta$ Ensembles). The point of view is very similar to ours since the relationship is probed through orthogonal polynomials, and of course, we expect similar results to hold in the context of the real line. A key ingredient would be the Bernstein-Szeg\"o type measure.

Nevertheless, the Main Theorem \ref{thm:main} is remarkable, as it says that in the case of the circle, the relationship is much stronger than initially expected. For example, the questions 1 and 2 raised by Vir\'ag \cite{V14}, in our context, would ask whether $GMC^\gamma$ and $\varprojlim C \beta E_n$ are similar at the level of fractal spectra. The fractal spectrum of a random measure on the circle is determined by the behavior of the number of arcs of the form $e^{2 i \pi [k, k+1]/m}$, $k \in \{0,1,\dots, m-1\}$, whose measure is larger than $m^{-\alpha}$ for a given $\alpha \in \mathbb{R}$, and by the  behavior of the sum of 
the $q$-th power of the measure of these arcs, for a given $q \in \mathbb{R}$, when $m \rightarrow \infty$: see \cite[Chapter 17]{F04} for a more detailed 
discussion. The answer to Vir\'ag's question, and more generally, to any question asking if  $GMC^\gamma$ and $\varprojlim C \beta E_n$ enjoy similar properties, 
is positive, since we prove that  $GMC^\gamma$ and $\varprojlim C \beta E_n$ are in fact the same object.

In fact, one could wonder in which sense the $C \beta E_n$ is a regularization of $GMC^\gamma$. As explained in Appendix \ref{section:appendix_CBE}, from the works of Macdonald, Random Matrix Theory is a very peculiar regularization of a Gaussian space {\it at the level of symmetric functions}. This regularization is of course very different from convolution, which is the standard process in order to construct the GMC (Theorem \ref{thm:GMC_CV}), hence the difficulty of proving $\varprojlim C \beta E_n = GMC^{\gamma}$. This difficulty is further exemplified by the following. Relating an approximation via finitely many Verblunsky coefficients and an approximation via convolution is present in the literature in the form of Golinskii-Ibragimov (GI) measures (See \cite[Section 6.1]{Sim05-1}), however one cannot apply any of the general approximation theorems that are available. Most of the results in \cite{Sim05-1} treat only regular measures by assuming the existence of densities or via the Szegö condition, which is a finite entropy condition for the Lebesgue measure relative to the measure of interest. And the GMC is very far from that.

Finally, in the same way that the GMC plays an important role in understanding the extrema of log-correlated fields, it is certainly desirable to relate the current paper to our previous work \cite{CMN18} investigating the extrema of the characteristic polynomial of the $C \beta E$ field.

\medskip

{\it Relation to conformal field theories:} The GMC plays a key role in the recent mathematical progress \cite{kupiainen2020integrability,
guillarmou2020conformal, guillarmou2021segal} for understanding Liouville Conformal Field Theory (CFT). In particular, one could ponder about the relationship of our Theorem \ref{thm:main} to these aforementioned results. On the one hand, since we establish an exact relationship between GMC and Beta ensembles, we argue this is an instance of the integrability of Liouville CFT, in the same fashion that the celebrated DOZZ formula \cite{kupiainen2020integrability} is exact. More interestingly, the GMC on the circle is the central object for the conformal bootstrap in \cite{guillarmou2020conformal, guillarmou2021segal}.

Notice that our log-correlated field $G$ in Eq. \eqref{eq:def_G} appears in \cite[Eq. (3.1-3.4)]{guillarmou2020conformal}. And $GMC^\gamma$ is the steady-state of the dynamic on annuli which allows to glue Liouville amplitudes according to Segal's axioms \cite[Section 6]{guillarmou2021segal}.

\addtocontents{toc}{\SkipTocEntry}
\subsection{Structure of the paper}

In Section \ref{section:GMC_in_D}, we show the convergence of $\Phi_n^*$ towards the exponential of a logarithmically correlated field inside the unit disc, and we provide a bound on the moments of $|\Phi_n^*(z)|$ when $z \in \D$. This is a consequence of a general result on OPUC, which can be of interest beyond our study of $C \beta E$. The setting is that of rotationally invariant Verblunsky coefficients with mild decay.
 
In Section \ref{section:proofmain}, we begin the proof of the Main Theorem \ref{thm:main}. In fact, we made the choice of factoring the proof of Theorem \ref{thm:main}, so that a first part can be presented as quickly as possible. This section gives all the required arguments for a complete proof modulo two Lemmas whose proofs are postponed for later. These are Lemma \ref{lemma:RenormalizedGMC_L1Limit} and Lemma \ref{lemma:OmegaRhoControl}, which motivate the next two sections.

In Section \ref{section:approximations}, we define and estimate some quantities, and consider a new probability distribution, in order to study the behavior of the polynomials $\Phi^*_n$ near the unit circle. 

In Section \ref{section:convergenceSDE}, we prove the convergence of some discrete stochastic process towards the solution of a suitable stochastic differential equation. This inhomogeneous SDE is rather ill-behaved and its analysis is a key ingredient in proving Lemma \ref{lemma:OmegaRhoControl}.

Finally, Section \ref{section:proofmain2} gives the missing proofs of Lemmas \ref{lemma:RenormalizedGMC_L1Limit} and \ref{lemma:OmegaRhoControl} and thus concludes the proof of the Main Theorem \ref{thm:main}.

\section{Orthogonal polynomials and Gaussian field inside the disc}
\label{section:GMC_in_D}
We consider the following Gaussian random holomorphic function $G^\C$, defined on the unit disc by 
$$G^\C(z) :=  \sum_{k=1}^\infty \frac{z^k}{\sqrt{k}} \Nc_k^\C$$
where $(\Nc_k^\C)_{k \geq 1}$ are i.i.d. complex Gaussian variables, such that 
$$\E [ \Nc_k^\C] = \E [ (\Nc_k^\C)^2] = 0, \; \E [ |\Nc_k^\C|^2]  = 1.$$
The function $G^\C$ itself is complex Gaussian and centered, the covariance structure being given by 
$$\E [ G^\C(w) G^\C(z) ] = 0, \; \E [  G^\C(w) \overline{G^\C(z)}]= - \log (1 - w \bar{z}).$$
From this covariance structure, we deduce that the field 
$$G(z) := 2 \Re (G^\C(z)) = G^\C(z) + \overline{G^\C(z)}$$
is real-valued, centered and Gaussian, with covariance 
$$\E [ G(w) G(z) ] = - \log ( 1 -  w \bar{z}) - \log ( 1 - \bar{w} z)  = - 2 \log | 1 - w \bar{z}|.$$
Recall that the Gaussian multiplicative chaos of parameter $\gamma < 1$ can then be constructed by considering the measure defined in \eqref{eq:gmc_r_def} and letting $r \rightarrow 1^-$. We have the following result: 

\begin{proposition}
\label{proposition31}
For $\beta >2$, let  $(\alpha_j)_{j \geq 0}$ be distributed as in Theorem \ref{thm:main}, and let $(\Phi_n^*)_{n \geq 0}$ be the corresponding sequence of OPUC. By general theory (see  \cite[Theorem 1.7.1 p.90]{Sim05-1}), these polynomials are equal to $1$ at $0$ and do not vanish on the unit disc. 
Hence, for any $n \geq 0$, there exists a unique continuous function from the unit disc to the complex plane, vanishing at zero, and 
 whose exponential is equal to $\Phi^*_n$: let us denote this function by  $\log \Phi_n^*$. Then, almost surely,  $\left( \log \Phi_{n}^* \right)_{n \geq 1}$ converges to a limit $\log \Phi_{\infty}^*$, uniformly on compact sets of the unit disc. Moreover, this limit has the same distribution as the Gaussian field $\gamma G^\C$ for $\gamma = \sqrt{\frac{2}{\beta}}$. Consequently, the random measure 
$$\lim_{r \rightarrow 1-} (1- r^2)^{ \frac{2}{\beta}} | \Phi_{\infty}^*(re^{i \theta})|^{-2}  \frac{d \theta}{2 \pi}$$
exists and has the same distribution as $GMC^{\gamma}$, and then Theorem \ref{thm:main} is proven if we show that $C_0 \mu^{\beta}$ coincides with this random measure. 

Moreover, we have the following bound on the moments of  $\Phi_{n}^*$: 
$$ 
\forall z \in \D, \forall p \in \R, \forall n \in \Z_+, \ 
\E\left(
\left| \Phi_n^*(z) \right|^{p}
\right)
\leq \left( 1 - |z|^2 \right)^{-\frac{p^2}{2 \beta}}.
$$
\end{proposition}
 
 Let us introduce the filtration: 
\begin{align}
\label{def:filtration}
\F := & \left( \Fc_n := \sigma\left( \alpha_0, \alpha_1, \dots, \alpha_{n-1} \right) ; n \in \Z_+ \right) \ .
\end{align}
Throughout the paper, the following martingale structure is crucial. We have, by the Szeg\"o recursion: 
$$\Phi_{n+1}^*(z) = \Phi_n^*(z) \left( 1 - \alpha_n Q_n(z) \right)$$
where 
$$Q_n(z) := z \frac{\Phi_n(z)}{\Phi_n^*(z)}.$$
From \cite[Corollary 1.7.2 p.90]{Sim05-1}, $Q_n <1$ inside the unit disc, and then we can write on $\mathbb{D}$:
$$\log \Phi_{n+1}^* (z) = \log \Phi_{n}^* (z)  + \log ( 1 - \alpha_n Q_n(z) ),$$
where we take the principal branch of the logarithm in the last term of the equality: we then have, for $n \geq 0$, 
$$ \log \Phi_{n}^* (z)  = \sum_{j=0}^{n-1} \log (1 - \alpha_j Q_j(n)).$$
Notice that this formula could have been used as a definition of $\log \Phi^*_n$.

From the fact that $Q_n$ is $\Fc_n$-measurable, $|\log ( 1- |\alpha_n|)|$ is integrable, $\alpha_n$ is rotationally invariant
and independent of $\Fc_n$, we deduce that $(\log \Phi_{n}^*(z))_{n \geq 0}$ is a $\F$-martingale, for all $z \in \mathbb{D}$. 

\begin{proof}[Proof of Proposition \ref{proposition31}]
We claim that, almost surely, $\left( \log \Phi_n^* \right)_{n \in \N}$ converges uniformly on compact sets of $\D$. Moreover, $\Re \log \Phi_{n}^*(z)$ and $\Im \log \Phi_{n}^*(z)$ have exponential moments of all orders (positive and negative), uniformly bounded in $n$ when the order and $z \in \D$ are fixed. A fortiori we have also uniform bounds for usual moments ($|x|^p \leq p!(e^x + e^{-x})$ for all $x \in \mathbb{R}$).
This is true by virtue of a general convergence result, Proposition \ref{proposition:cv_opuc}, which we shall prove in the next subsections. The only required hypothesis is \eqref{eq:Sigma2_hypothesis} and it is implied by:

\begin{lemma}
 For all $k>2$ and $\sigma \geq 0$:
 $$ \E\left( e^{ \sigma \sum_{j \geq 0} |\alpha_j|^k } \right) < \infty \ .$$
\end{lemma}
\begin{proof}
Using the independence of the $\alpha_j$'s, their explicit density and then an integration by parts:
\begin{align*}
    \E\left( e^{ \sigma \sum_{j \geq 0} |\alpha_j|^k } \right)
= & \prod_{j=0}^\infty \E\left( e^{ \sigma |\alpha_j|^k } \right) \\
= & \prod_{j=0}^\infty \left( \beta_j \int_0^1 dx \ e^{\sigma x^{k/2}} (1-x)^{\beta_j-1} \right) \\
= & \prod_{j=0}^\infty \left( 1 + \half k \sigma \int_0^1 dx \ e^{\sigma x^{k/2}} x^{k/2-1} (1-x)^{\beta_j} \right) \\
\leq & \prod_{j=0}^\infty \left( 1 + \half k \sigma e^\sigma \int_0^1 dx \ x^{k/2-1} (1-x)^{\beta_j} \right) \\
=    & \prod_{j=0}^\infty
       \left( 1 + \half k \sigma e^\sigma
       \frac{ \Gamma( \half k ) \Gamma(\beta_j + 1) }{\Gamma(\beta_j + \half k + 1)}
       \right) \ .
\end{align*}
This product is finite because of the asymptotics:
$$  \frac{\Gamma(\beta_j + 1) }{\Gamma(\beta_j + \half k + 1)}
\ll \beta_j^{-\half k} \ll j^{-\half k} \ .
$$
\end{proof}

Let us now identify the law of the limit of  $\log \Phi_n^*$. By the ratio asymptotics \cite[Theorem 1.7.4 p.91]{Sim05-1}, as $\alpha_n \stackrel{n \rightarrow \infty}{\longrightarrow} 0$, we have a.s.
$$ \lim_{n \rightarrow \infty} \frac{\Phi_{n-1}(z)}{\Phi_{n-1}^*(z)} = 0,$$
uniformly in the interior of the unit disc $\D$. On the other hand, from results by Killip and Nenciu \cite{KN04}, we deduce that if
$\log X_n$ is the logarithm of the  characteristic polynomial corresponding to the $C \beta E$, defined as follows: 
$$\log X_n (z) := \sum_{\lambda \in C \beta E_n}  \log (1 - \lambda z),$$
then we have the equality in law: 
$$ \left( \log X_n(z) \ ; \ z \in \D \right)
 = \left( \log \Phi_{n-1}^*(z)  + \log \left( 1 - z \eta \frac{\Phi_{n-1}(z)}{\Phi_{n-1}^*(z)} \right) \ ; \ z \in \D \right)$$
where $\eta$ is an independent uniform random variable on the unit circle.
We deduce that $\log X_n$ converges in law to $ \log \Phi^*_{\infty}$ when $n$ goes to infinity, for the topology of uniform convergence 
on compact sets in $\D$. In particular, since the Taylor coefficients at zero of $\log X_n$ can be written as contour integrals involving $\log X_n$ on a 
small circle centered at $0$, their finite dimensional joint distributions tend to the corresponding distributions for $\log \Phi^*_{\infty}$. 

Now, thanks to the result of Jiang and Matsumoto \cite{JM15}, the joint distribution, for finitely many given values of $k \geq 1$,  of the sum of the $k$-th powers of the zeros of $X_n$, tends to the joint distribution of the corresponding independent complex Gaussian variables $\sqrt{\frac{2k}{\beta}} \Nc_k^\C$.  We provide an independent proof of this fact in Appendix \ref{section:appendix_CBE}. Although not quantitative, this independent proof has the advantage of showing why $C \beta E$ is inherently a regularization of a Gaussian space
 at the level of symmetric functions \footnote{This proof was in fact known to some specialists, like Philippe Biane}.
Using the standard expansion of the logarithm, the Taylor coefficients of $\log X_n$ tend, in the sense of finite dimensional marginals, to independent Gaussians $\sqrt{\frac{2}{\beta k}}\Nc_k^\C$. Hence, the Taylor coefficients of $\log \Phi^*_{\infty}$ have the same joint
distribution as these variables, which shows that $\log \Phi^*_{\infty}$ has the same law as $\gamma G^\C$. 

It remains to check the bounds on moments. For that, we observe that for $z \in \D$, $p \in \mathbb{R}$, $(|\Phi^*_n(z)|^p)_{n \geq 0}$ is a submartingale, since it is the image of the martingale $(\Re \log \Phi^*_n(z))_{n \geq 1}$ by the convex function $x \mapsto e^{px}$. From the bound on the exponential moments of $\Re \log  \Phi^*_n(z)$, we deduce that $(|\Phi^*_n(z)|^p)_{n \geq 0}$ is bounded in $L^2$, and
then, by Doob's submartingale inequality, $\sup_{n \geq 0} |\Phi^*_n(z)|^p$ is in $L^2$, and a fortiori in $L^1$.  By dominated convergence, we deduce, since $\Phi^*_n (z)$ converges a.s. to $\Phi^*_{\infty}(z)$, that
$$\E [ |\Phi^*_{\infty}(z)|^p] =\lim_{n \rightarrow \infty} \E [ |\Phi^*_{n}(z)|^p].$$

Since $(|\Phi^*_n(z)|^p)_{n \geq 0}$  is a submartingale, the last limit is increasing, and then for any $n$, 
$$\E [ |\Phi^*_{n}(z)|^p] \leq \E [ |\Phi^*_{\infty}(z)|^p] = \E [ e^{ p \gamma \Re(G^\C(z))}] = \E [ e^{\half p \gamma G(z)} ] = (1-|z|^2)^{- \frac{p^2 \gamma^2}{4}} = (1-|z|^2)^{- \frac{p^2}{2 \beta} }.$$

The proof of the proposition is complete, modulo the convergence result stated in Proposition \ref{proposition:cv_opuc}. We chose to treat it separately because the technology developed in the next two subsections actually holds beyond the $C \beta E$, for large classes of orthogonal polynomials.
\end{proof}

The identity in law between $\log \Phi^*_{\infty}$ and $\gamma G^\C$, and then between $\Phi^*_{\infty}$ and $e^{\gamma G^\C}$, implies identities in law between functions of Gaussian variables and functions of the Verblunsky coefficients. Indeed, using Szeg\"o recursion (written in a matricial way), we can write the coefficients of the polynomial $\Phi^*_n$ as polynomials in the Verblunsky coefficients $\alpha_0, \dots, \alpha_{n-1}$ and their conjugates. Passing to the limit, we deduce an expression of the Taylor coefficients of $\Phi^*_{\infty}$ as limits of infinite series involving all the Verblunsky coefficients and their conjugates. On the other hand, the Taylor coefficients of $e^{\gamma G^\C}$ can be written as polynomials in the Gaussian variables $\Nc_k^\C$. Identifying the two expressions gives the following result ($SC_n$ representing the coefficient in $z^n$): 

\begin{corollary}
For $N, n \geq 0$, let $\Pi_{n,N}$ be the set of sequences $(\pi_j)_{j \geq 0}$ such that $\pi_j = 1$ for $n$ values of $j$, all smaller than or equal to $N$, 
and $\pi_j = 2$ for all the other values of $j$, and let 
$$ SC_{n,N} := \sum_{\pi \in \Pi_{n,N} } \prod_{j,  \pi_{j} \pi_{j+1} = 12} \left[ -\alpha_j \right]
 \prod_{j, \pi_{j} \pi_{j+1} = 21}    \left[ -\overline{\alpha_j} \right]$$
      Then, $SC_{n,N}$ a.s. tends to a limit $SC_n$ when $N$ goes to infinity, 
      and this limit can be expressed in terms of i.i.d Gaussians $\Nc^\C_k$ as:
$$ SC_n = \sum_{(m_k)_{k \geq 1}, \sum_{k \geq 1} k m_k = n} \prod_{k \geq 1} \frac{\left( \Nc_k^\C \right)^{m_k} }{m_k!}
                                          \left( \frac{2}{\beta k} \right)^{\half m_k} \ .
$$
\end{corollary}
In particular, for $n = 1$, we deduce the following non-trivial identity in law
\begin{align}
\label{eq:firstGaussian}
\sqrt{ \frac{2}{ \beta}} \Nc_1^\C = \sum_{j=-1}^\infty \overline{ \alpha_j } \alpha_{j+1} \ ,
\end{align}
the series being a.s. (not absolutely) convergent. In the above sum, by convention $\alpha_{-1}=-1$. This convention is consistent with the book \cite{Sim05-1} and shall be used throughout the paper.

From the corollary, we can deduce an expression of the Gaussian variables $\Nc_k^\C$ themselves in terms of the $\alpha_j$. A more direct way to get this expression is to use the CMV matrices (from Cantero, Moral, and Velazquez \cite{CMV03}, see also \cite[Section 4.2.]{Sim05-1}). One can show  that the characteristic polynomial corresponding to the  $C \beta E$ of order $n$ has the same distribution as the characteristic polynomial of the matrix 
$\mathcal{C}_n := \mathcal{L}_n \mathcal{M}_n$, 
$\mathcal{L}_n$ and $\mathcal{M}_n$ being the $n \times n$ top-left minors of $\mathcal{L}_{n,0}$ and $\mathcal{M}_{n,0}$, with 
$$ \mathcal{L}_{n,0}= \operatorname{Diag} (\Theta_{n,0}, \Theta_{n,2}, \dots),
\; \mathcal{M}_{n,0} = \operatorname{Diag} (1, \Theta_{n,1}, \Theta_{n,3}, \dots),$$
$$\Theta_{n,j} = \left( \begin{matrix} 
\alpha_{n,j} & \rho_{n,j} \\ \rho_{n,j} & -\overline{\alpha_{n,j}}
\end{matrix} \right), \; \; \rho_{n,j} = (1 - |\alpha_{n,j}|^2)^{1/2}.
$$
Here, $\alpha_{n,j} = \alpha_j$ for $j \leq n-2$, $\alpha_{n,n-1}$ is independent of the $\alpha_j$'s and uniform on the unit circle,
$\alpha_{n,j} = 0$ for $j \geq n$. Notice that since the matrices $\Theta_{n,j}$ have size $2 \times 2$ and 
the coefficient $1$ is a single entry, we have for $n$ even, 
$$ \mathcal{L}_{n}= \operatorname{Diag} (\Theta_{n,0}, \Theta_{n,2}, \dots, \Theta_{n, n-2}),
\; \mathcal{M}_{n} = \operatorname{Diag} (1, \Theta_{n,1}, \Theta_{n,3}, \dots, \Theta_{n,n-3 } , \alpha_{n, n-1} ),$$
and for $n$ odd, 
$$ \mathcal{L}_{n}= \operatorname{Diag} (\Theta_{n,0}, \Theta_{n,2}, \dots, \Theta_{n, n-3}, \alpha_{n, n-1} ),
\; \mathcal{M}_{n} = \operatorname{Diag} (1, \Theta_{n,1}, \Theta_{n,3}, \dots, \Theta_{n,n-2 }),$$

The coefficient in $z^k$ of $- \log X_n(z)$ is given by $\frac{1}{k}$ times the sum of the $k$-th power of the points in the $C \beta E$. 
Hence, the joint law of the coefficients of $- \log X_n(z)$ is the same as the joint law of $(\operatorname{tr} (\mathcal{C}_n^k)/k)_{k \geq 1}$, 
and we deduce, from the convergence in law of $\log X_n$ towards $\log \Phi_{\infty}^*$, 
that the finite-dimensional marginals of $(- \sqrt{ \frac{\beta}{2k} } \operatorname{tr} (\mathcal{C}_n^k))_{k \geq 1}$ tend in law to i.i.d. complex Gaussian variables. 

In the expansion of $\operatorname{tr}( \mathcal{C}_n^k)$, each term involves a product of $\Oc(k)$ factors equal to  $\alpha_{n,j}$, $\overline{\alpha_{n,j}}$ or
$\rho_{n,j}$, and such that all the indices $j$ involved are in an interval of length $\Oc(k)$. 
For fixed $k$, we then have a bounded number of terms involving $\alpha_{n,n-1}$ or its conjugate, and all these terms tend to zero in probability 
when $n \rightarrow \infty$, since a careful look of the matrix product shows  that they necessarily also involve a factor $\alpha_{n,j}$ for $j = n + \Oc(k)$, $j \neq n-1$. 
Hence, we can replace  $\alpha_{n,n-1}$ by $0$ in the matrices  $\mathcal{L}_n$ and $\mathcal{M}_n$ without changing the limiting distribution of $\operatorname{tr} \mathcal{C}_n^k$. It is not difficult to deduce the following result: 
\begin{corollary}
\label{corollary:Verblunsky2Gaussians}
Let $\mathcal{C} = \mathcal{L} \mathcal{M}$, for $$ \mathcal{L}= \operatorname{Diag} (\Theta_{0}, \Theta_{2}, \dots),
\; \mathcal{M} = \operatorname{Diag} (1, \Theta_{1}, \Theta_{3}, \dots),$$
where
$$\Theta_{j} = \left( \begin{matrix} 
\alpha_{j} & \rho_{j} \\ \rho_{j} & -\overline{\alpha_{j}}
\end{matrix} \right), \; \; \rho_{j} = (1 - |\alpha_{j}|^2)^{1/2},
$$
$(\alpha_j)_{j \geq 0}$ being distributed as in Theorem \ref{thm:main}. 
Then, $(- \sqrt{ \frac{\beta}{2k} } \operatorname{tr} (\mathcal{C}^k))_{k \geq 1}$ has the same distribution as 
$(\mathcal{N}^{\C}_k)_{k \geq 1}$, where $\operatorname{tr} (\mathcal{C}^k)$ is obtained by taking the formal expansion of the 
trace, removing all the terms involving indices larger than $n$  and letting $n \rightarrow \infty$. 
\end{corollary}

Such formulas have been investigated before, for example in \cite{GZ07}.

\subsection{A universal bound on traces}
Notice that $(\log \Phi_n^*(z))_{n \geq 0}$ is a complex martingale. As such, thanks to Jensen's inequality, for all $\sigma \in \C$, 
$$ (e^{\Re \sigma \log\Phi_n^*(z)})_{n \geq 0}$$ is a real submartingale. We shall now prove that it is uniformly bounded in $L^1(\Omega, \mathcal{B}, \P)$, using the description thanks to CMV matrices. Using, in  \cite{Sim05-1}, the equation following (4.2.48), p. 271,  we get, after taking care of the conventions of conjugation of the $\alpha_j$'s (which are not the same here and in   \cite{Sim05-1}), 
$$
 \log  \Phi^*_n(z) = -  \sum_{k=1}^{\infty} \frac{z^k}{k} \operatorname{tr}( \Cc_{[n]}^k ),
$$
 where $ \Cc_{[n]}$ is the $n \times n$ top-left block of the infinite matrix $\Cc$ introduced in Corollary \ref{corollary:Verblunsky2Gaussians}.
A careful look at the matrix product shows that the traces of powers of $\Cc_{[n]}$ are equal to the traces of powers of $\Cc$ after replacing all Verblunsky coefficients of index larger than or equal to $n$ by zero.

In the following computations, we then omit the subscript $n$ and we  always implicitly assume that $\alpha_j$ has been replaced by zero for $j \geq n$. We have the following universal bound on traces, which is clearly of independent interest, and which holds deterministically: 
\begin{lemma}
\label{lemma:cmv_trace}
For all $k \geq 2$:
\begin{align*}
   \tr \left( \Cc^k \right)
= -k \left( \sum_{j \geq -1} \overline{\alpha_j} \rho_{j+1}^2 \rho_{j+2}^2 \dots \alpha_{j+k} \right)
  + \Oc \left( k^3 \sum_{j \geq -1} |\alpha_j|^3  \right) \ ,
\end{align*}
the implicit constant in $\Oc$ being absolute.
\end{lemma}
\begin{proof}
We prove the Lemma by a 3-step commutator argument. We notice that in the present setting, finitely many Verblunsky coefficients are non-zero and then there is no issue of convergence for the series which are involved. 

Introduce the matrix $\Cc_\rho = \Lc_\rho \Mc_\rho$ where all the $\alpha$ terms have been replaced by zero, i.e. 
$$ \Lc_\rho = \operatorname{Diag} (\Theta_{0}^\rho, \Theta_{2}^\rho, \dots),
\; \Mc_\rho = \operatorname{Diag} (0, \Theta_{1}^\rho, \Theta_{3}^\rho, \dots),$$
where
$$\Theta_{j}^\rho = 
  \left( \begin{matrix} 
         0 & \rho_{j} \\ \rho_{j} & 0
  \end{matrix} \right) \ .
$$
For an infinite matrix $\mathcal{P}$ whose rows have finitely many non-zero entries, 
we define its operator norm as
$$|\mathcal{P}|_{op} := \sup_{(x_1, x_2, \dots) \in \mathbb{C}^{\mathbb{N}}, \; 
\sum_{j=1}^{\infty} |x_j|^2 = 1}  \left(  \sum_{j=1}^{\infty}  \left| \sum_{ k=1}^{\infty} \mathcal{P}_{j,k} x_k \right|^2  \right)^{1/2},$$
and its trace as
$$\tr (\mathcal{P}) = \sum_{j=1}^{\infty} \mathcal{P}_{j,j}$$
when the sum is absolutely convergent. 

We have $\left| \Lc_\rho \right|_{op}, \left| \Mc_\rho \right|_{op} \leq 1$ and the same goes for $\Lc$ and $\Mc$. Also $\Cc_\rho$ is a shift operator in the sense that:
$$ \Cc_\rho  e_k = u_k e_{k \pm 2} \ ,$$
depending whether $k$ is odd or even, for some $u_k \in [0,1]$.

Before presenting the 3-step commutator argument mentioned above, let us present 1-step and 2-step arguments which provide similar, but simpler and weaker estimates of $\operatorname{tr} \left( \Cc^k \right) $.

{\bf 1-step commutator:} Because $\Cc_\rho$ is the matrix of a weighted shifting operator, $\tr \Cc_\rho^k = 0$. Write
$$ \Cc^k        = \Lc \Mc \dots \Lc \Mc                     \textrm{ ($k$ times) } \ ,$$
$$ \Cc_{\rho}^k = \Lc_\rho \Mc_\rho \dots \Lc_\rho \Mc_\rho \textrm{ ($k$ times) } \ ,$$
and use the commutator:
\begin{align*}
	\tr \left( \Cc^k \right)
= & \tr \left( \Cc^k - \Cc_\rho^k \right)\\
= & \sum_{i=0}^{k-1} \tr \left( \Cc^i (\Cc - \Cc_\rho) \Cc_\rho^{k-i-1} \right) \\
= & \sum_{i=0}^{k-1} \tr \left( \Cc^i (\Lc - \Lc_\rho) \Mc_\rho \Cc_\rho^{k-i-1} \right)
                   + \tr \left( \Cc^i \Lc (\Mc - \Mc_\rho) \Cc_\rho^{k-i-1} \right)\\
= & \sum_{i_1 + i_2 = k-1}
        \tr \left( \Cc^{i_1} (\Lc - \Lc_\rho) \Mc_\rho \Cc_\rho^{i_2} \right)
      + \tr \left( \Cc^{i_1} \Lc (\Mc - \Mc_\rho) \Cc_\rho^{i_2} \right)
\end{align*}
The nice fact about this operation is that it forces out the appearance of diagonal matrices $\Lc-\Lc_\rho$ and $\Mc-\Mc_\rho$. We will freely invoke (see \cite[Chapter 6]{RS80}) the circular property of the trace $\tr(AB) = \tr(BA)$ when $A$ is trace-class and $B$ bounded, and the inequality $| \tr(AB) | \leq |AB|_1 \leq |A|_1 |B|_{op}$, where $|A|_1 = \tr \left[ (A^* A)^\half \right]$ is the trace-class norm  and $|B|_{op}$ is the operator norm (\cite[Chapter 6, Exercise 28]{RS80}). For example, those facts applied to the previous computation give a bound on the traces:
$$   \left| \tr \left( \Cc^k \right) \right|
\leq 2k  \sum_{j \geq -1} |\alpha_j|.
$$
Before pursuing, let us adopt a convention that will ease notations. Write
$$ 
   A^{(i)} = \left\{ 
             \begin{array}{cc}
             \Lc & \textrm{if $i$ is odd}\\
             \Mc & \textrm{if $i$ is even}
             \end{array} \right. ;
  \quad
   A_\rho^{(i)} = \left\{ 
             \begin{array}{cc}
             \Lc_\rho & \textrm{if $i$ is odd}\\
             \Mc_\rho & \textrm{if $i$ is even}
             \end{array} \right. \ .  
$$ 
In the same fashion:
$$
   B^{(i)} = \left\{ 
             \begin{array}{cc}
             \Lc - \Lc_\rho & \textrm{if $i$ is odd}\\
             \Mc - \Mc_\rho & \textrm{if $i$ is even}
             \end{array} \right. \ .  
$$
Then define:
$$ A^{(i_1, i_2)} := A^{(i_1)} \dots A^{(i_2)} \ ,  A_{\rho}^{(i_1, i_2)} := A_{\rho}^{(i_1)} \dots A_{\rho}^{(i_2)} \ ,$$
with the convention that $A^{(i_1, i_2)} = A_{\rho}^{(i_1, i_2)} =  \id$ if $i_1>i_2$. As such, the above commutator argument becomes:
\begin{align*}
	\tr \left( \Cc^k \right)
= & \tr \left( \Cc^k - \Cc_\rho^k \right)\\
= & \tr \left( A^{(1, 2k)} - A_\rho^{(1, 2k)} \right)\\
= & \sum_{j=0}^{2k-1}
      \tr \left( A^{(1, j)} A^{(j+1)} A_\rho^{(j+2, 2k)}
               - A^{(1, j)} A_\rho^{(j+1)} A_\rho^{(j+2, 2k)} \right)\\
= & \sum_{j=0}^{2k-1}
        \tr \left( A^{(1, j)} B^{(j+1)} A_\rho^{(j+2, 2k)} \right) \ .
\end{align*}

{\bf 2-step commutator:} We start by the same argument as before, that is to say that $A_\rho^{(j+2, 2k)} A_\rho^{(1, j)} = \dots \Lc_\rho \Mc_\rho \Lc_\rho \Mc_\rho \dots$ is a weighted shifting operator. Because $B^{(j+1)}$ is diagonal, $B^{(j+1)} A_\rho^{(j+2, 2k)} A_\rho^{(1, j)}$ is a shifting operator as well. As a consequence, by the circular property of the trace:
$$ \tr \left( A_\rho^{(1, j)} B^{(j+1)} A_\rho^{(j+2, 2k)} \right) = 0 \ .$$
Therefore, we can repeat the operation and obtain a two step commutator:
\begin{align*}
	\tr \left( \Cc^k \right)
= & \tr \left( \Cc^k - \Cc_\rho^k \right)\\
= & \sum_{j=0}^{2k-1}
       \left[  \tr \left( A^{(1, j)} B^{(j+1)} A_\rho^{(j+2, 2k)} \right)
      - \tr \left( A_\rho^{(1, j)} B^{(j+1)} A_\rho^{(j+2, 2k)} \right)  \right]\\
= & \sum_{0 \leq j_1 < j_2 \leq 2k-1}
       \left[  \tr \left( A^{(1, j_1)} A^{(j_1+1)} A_\rho^{(j_1+2, j_2)} B^{(j_2+1)} A_\rho^{(j_2+2, 2k)} \right) \right.\\
  & \quad \quad 
     \left. - \tr \left( A^{(1, j_1)} A_\rho^{(j_1+1)} A_\rho^{(j_1+2, j_2)} B^{(j_2+1)} A_\rho^{(j_2+2, 2k)} \right) \right]\\
= & \sum_{0 \leq j_1 < j_2 \leq 2k-1}
        \tr \left( A^{(1, j_1)} B^{(j_1+1)} A_\rho^{(j_1+2, j_2)} B^{(j_2+1)} A_\rho^{(j_2+2, 2k)} \right) \ .
\end{align*}
This gives another bound on the traces, as follows. For each of the indices in the sum $0 \leq j_1 < j_2 \leq 2k-1$, consider the trace which is more conveniently written as a sum over $\Z$:
\begin{align*}
  & \tr \left( A^{(1, j_1)} B^{(j_1+1)} A_\rho^{(j_1+2, j_2)} B^{(j_2+1)} A_\rho^{(j_2+2, 2k)} \right) \\
= & \sum_{i \in \Z} \left[ A^{(1, j_1)} B^{(j_1+1)} A_\rho^{(j_1+2, j_2)} B^{(j_2+1)} A_\rho^{(j_2+2, 2k)} \right]_{i,i} \\
= & \sum_{i \in \Z} \left[ B^{(j_1+1)} \right]_{i,i}
                    \left[ A_\rho^{(j_1+2, j_2)} \right]_{i,i+*}
                    \left[ B^{(j_2+1)} \right]_{i+*,i+*}
                    \left[ A_\rho^{(j_2+2, 2k)} A^{(1, j_1)}  \right]_{i+*,i} \ ,
\end{align*}
where $*$ is a shift, whose value is of no importance, and where terms are considered to be equal to zero if they involve nonpositive indices.  Taking absolute values and using the fact that for any operator $A$, $|[A]_{i,j}| \leq |A|_{op}$, we have:
\begin{align*}
     & \left|
    \tr \left( A^{(1, j_1)} B^{(j_1+1)} A_\rho^{(j_1+2, j_2)} B^{(j_2+1)} A_\rho^{(j_2+2, 2k)} \right)
    \right|    
    \\
\leq & \sum_{i \in \Z} |\left[ B^{(j_1+1)} \right]_{i,i}|
                       |\left[ B^{(j_2+1)} \right]_{i+*,i+*}|
                       \left| A_\rho^{(j_1+2, j_2)} \right|_{op}
                       \left| A_\rho^{(j_2+2, 2k)} A^{(1, j_1)} \right|_{op} \\
\leq & \sum_{i \in \Z} |\left[ B^{(j_1+1)} \right]_{i,i}|
                       |\left[ B^{(j_2+1)} \right]_{i+*,i+*}| \\
\leq & \half \sum_{i \in \Z}  \left( |\left[ B^{(j_1+1)} \right]_{i,i}|^2
                          +  |\left[ B^{(j_2+1)} \right]_{i+*,i+*}|^2 \right) \\
\ll  & \,  \sum_{i \geq -1} |\alpha_i|^2 \ .
\end{align*}

Notice that we did not even use trace-class inequalities, only that $B$ is diagonal and $A, A_\rho$ are bounded. In the end:
$$   \left| \tr \left( \Cc^k \right) \right|
\ll  k^2    \sum_{j \geq -1}  |\alpha_j|^2, 
$$
the implicit constant being absolute.

{\bf 3-step commutator:} At this level, the computation is different. The crucial point is that, for most indices $j_1, j_2$:
$$ \tr \left( A_\rho^{(1, j_1)} B^{(j_1+1)} A_\rho^{(j_1+2, j_2)} B^{(j_2+1)} A_\rho^{(j_2+2, 2k)} \right) = 0 \ ,$$
but not all. We need a closer inspection. 

If $\ell$ is even, $\Mc_\rho e_\ell \in \R e_{\ell+1}$; while if $\ell$ is odd, then $\Lc_\rho e_\ell \in \R e_{\ell+1}$. Reversing parity yields $e_{\ell-1}$ instead of $e_{\ell+1}$. As such, if $\ell$ is even:
\begin{align*}
          & A_\rho^{(1, j_1)} B^{(j_1+1)} A_\rho^{(j_1+2, j_2)} B^{(j_2+1)} A_\rho^{(j_2+2, 2k)}( \R e_\ell )\\
\subset \ & A_\rho^{(1, j_1)} B^{(j_1+1)} A_\rho^{(j_1+2, j_2)} B^{(j_2+1)} ( \R e_{\ell+2k-j_2-1} )\\
\subset \ & A_\rho^{(1, j_1)} B^{(j_1+1)} A_\rho^{(j_1+2, j_2)} ( \R e_{\ell+2k-j_2-1} )\\
\subset \ & A_\rho^{(1, j_1)} ( \R e_{\ell+2k-j_2-1-(j_2-j_1-1)} )\\
\subset \ & \R e_{\ell+2k-j_2-1-(j_2-j_1-1)+j_1}
          = \R e_{\ell+2k-2j_2+2j_1}
\end{align*}
Upon considering $\ell$ odd as well, one has for all $\ell$:
$$
A_\rho^{(1, j_1)} B^{(j_1+1)} A_\rho^{(j_1+2, j_2)} B^{(j_2+1)} A_\rho^{(j_2+2, 2k)}( \R e_\ell )
\subset \ \R e_{\ell \pm 2 (k-(j_2-j_1))} \ .
$$
Therefore, in the combinatorial computation of the trace, one has a ``closed loop'' only for the indices such that $j_2-j_1=k$. Consequently, 
\begin{align*}
  & \tr \left( \Cc^k - \Cc_\rho^k \right)\\
= & \sum_{0 \leq j_1 < j_2 \leq 2k-1}
       \left[ \tr \left( A^{(1, j_1)} B^{(j_1+1)} A_\rho^{(j_1+2, j_2)} B^{(j_2+1)} A_\rho^{(j_2+2, 2k)} \right) \right. \\ 
        & 
    \left.   - \tr \left( A_\rho^{(1, j_1)} B^{(j_1+1)} A_\rho^{(j_1+2, j_2)} B^{(j_2+1)} A_\rho^{(j_2+2, 2k)} \right) \right] \\
  & + \quad  \sum_{0 \leq j_1 \leq k-1}
        \tr \left( A_\rho^{(1, j_1)} B^{(j_1+1)} A_\rho^{(j_1+2, j_1+k)} B^{(j_1+k+1)} A_\rho^{(j_1+k+2, 2k)} \right) \\ .
= & \sum_{0 \leq j_1 < j_2 < j_3 \leq 2k-1}
        \tr \left( A^{(1, j_1)} B^{(j_1+1)}
                   A_\rho^{(j_1+2, j_2)} B^{(j_2+1)} 
                   A_\rho^{(j_2+2, j_3)} B^{(j_3+1)} 
                   A_\rho^{(j_3+2, 2k)}
            \right) \\
  & + \quad  \sum_{0 \leq j_1 \leq k-1}
        \tr \left( A_\rho^{(1, j_1)} B^{(j_1+1)} A_\rho^{(j_1+2, j_1+k)} B^{(j_1+k+1)} A_\rho^{(j_1+k+2, 2k)} \right) .
\end{align*}
where in the last line, we used the usual commutator trick to rearrange the first sum. This first sum can be controlled using the fact that the $A$'s are subunitary and the $B$'s are diagonal, like in step 2 (note that the symbols $*$ denote shifts which may not be the same for different factors of the same term):
\begin{align*}
  & \left|
    \sum_{0 \leq j_1 < j_2 < j_3 \leq 2k-1}
        \tr \left( A^{(1, j_1)} B^{(j_1+1)}
                   A_\rho^{(j_1+2, j_2)} B^{(j_2+1)} 
                   A_\rho^{(j_2+2, j_3)} B^{(j_3+1)} 
                   A_\rho^{(j_3+2, 2k)}
            \right)
   \right| \\
\leq & (2k)^3 \max_{0 \leq j_1, j_2, j_3 \leq 2k-1} 
              \left|
              \sum_{i \in \Z}
              \left[ A^{(1, j_1)} B^{(j_1+1)}
                   A_\rho^{(j_1+2, j_2)} B^{(j_2+1)} 
                   A_\rho^{(j_2+2, j_3)} B^{(j_3+1)} 
                   A_\rho^{(j_3+2, 2k)} \right]_{i,i}
              \right|
              \\
\leq & (2k)^3 \max_{0 \leq j_1, j_2, j_3 \leq 2k-1} 
              \Big |
              \sum_{i \in \Z}
              \left[ B^{(j_1+1)} \right]_{i,i}
              \left[ A_\rho^{(j_1+2, j_2)} \right]_{i,i+*}
              \left[ B^{(j_2+1)} \right]_{i+*,i+*}
              \left[ A_\rho^{(j_2+2, j_3)} \right]_{i+*,i+*}            
              \\
     &  \quad \quad
              \left[ B^{(j_3+1)} \right]_{i+*,i+*}
              \left[ A_\rho^{(j_3+2, 2k)} A^{(1, j_1)} \right]_{i+*,i}
              \Big |
              \\
\leq & (2k)^3 \max_{0 \leq j_1, j_2, j_3 \leq 2k-1} 
              \sum_{i \in \Z}
              \left|
              \left[ B^{(j_1+1)} \right]_{i,i}
              \left[ B^{(j_2+1)} \right]_{i+*,i+*}
              \left[ B^{(j_3+1)} \right]_{i+*,i+*}
              \right| \\
\ll  & (2k)^3   \sum_{j \geq -1} |\alpha_j|^3  \ .
\end{align*}
The second sum is explicitly obtained by following the ``loops'' in the combinatorial computation of the trace, without forgetting the entry $B^{(2)}_{1,1}$. We get 
$$ - k \sum_{j \geq -1} \overline{\alpha_j}\rho_{j+1}^2 \rho_{j+2}^2 \dots  \rho_{j+k-1}^2 \alpha_{j+k} \ ,$$
which gives the bound we require:
$$   \tr \left( \Cc^k \right)
   = - k \left( \sum_{j \geq -1} \overline{\alpha_j} \rho_{j+1}^2 \rho_{j+2}^2 \dots \alpha_{j+k} \right)
     + \Oc \left( k^3  \sum_{j \geq -1} |\alpha_j|^3  \right).
$$
\end{proof}

\subsection{A convergence result for OPUC with rotationally invariant \texorpdfstring{$(\alpha_j)_{j \geq 0}$}{alphaj}}
\label{subsection:exp_moments}

We can now prove the following general convergence result for $\log \Phi_\infty^*$ inside compact sets of $\D$. The setting is that of independent and  rotationally invariant Verblunsky coefficients, so that $(\log \Phi_n^*)_{n \geq 0}$ is an $\F$-martingale. Also, the condition on the decay of modulii $|\alpha_j|$ easily includes the square-root decay in Random Matrix Theory.

\begin{proposition}
\label{proposition:cv_opuc}
Assume that Verblunsky coefficients are independent and rotationally invariant and that
\begin{align}
\label{eq:Sigma2_hypothesis}
\forall \sigma>0, \ \forall k \geq 3, \ &
\E\left[ \exp\left( \sigma \sum_{j \geq -1} |\alpha_j|^k  \right) \right] < \infty \ .
\end{align}
As a consequence, for any compact set $K \subset \D$, $\left( \log \Phi_n^* \right)_{n \in \N}$ almost surely converges uniformly on $K \subset \D$ and 
\begin{align}
\label{eq:finiteExpMoments}
\forall \sigma \in \C, \ \sup_{z \in K} \sup_{n \in \N}
\E\left[ e^{ \Re \sigma \log \Phi_n^*(z) } \right] & < \infty.
\end{align}
\end{proposition}
\begin{proof}
We start by proving the finiteness of exponential moments given in \eqref{eq:finiteExpMoments}. Write: 
$$ F_{n,k} := \sum_{-1 \leq j \leq n-k-1} \overline{\alpha_j} \rho_{j+1}^2 \dots \rho_{j+k-1}^2 \alpha_{j+k} \,$$
where we recall that  $\alpha_{-1} := -1$ by convention. 
Thanks to Lemma \ref{lemma:cmv_trace}:
$$ \log \Phi_n^*(z) = \sum_{k=1}^\infty z^k F_{n,k} + \Oc\left(  1 + \sum_{j \geq 0} |\alpha_j|^3 \right) \ ,$$
the $\Oc$ being uniform on compact sets. Because of the hypothesis \eqref{eq:Sigma2_hypothesis} and the Cauchy-Schwarz inequality, it is sufficient to show that:
$$ \forall \sigma>0, \sup_{z \in K} \sup_{n \in \N} \E\left( e^{\sigma \sum_{k=1}^\infty |z|^k |F_{n,k}| } \right) < \infty \ ,$$
for \eqref{eq:finiteExpMoments} to be true.

 First, we work conditionnally on the $\sigma$-algebra $\Mc$ generated by the modulii $|\alpha_j|$. By seeing the random variable $F_{n,k}$ as a function of the bounded phases $\Theta_j$, it is easy to check that:
 \begin{align*}
  &
  \left| F_{n,k}(\dots, \Theta_i , \dots)
       - F_{n,k}(\dots, \Theta_i', \dots)
  \right|\\
\leq & | \Theta_i - \Theta_i'| \, |\alpha_i| \rho_{i+1}^2 \dots \rho_{i+k-1}^2 |\alpha_{i+k}|
     +| \alpha_{i-k}| \,  \rho_{i-k+1}^2 \dots \rho_{i-1}^2 |\alpha_{i}| \, |\Theta_i - \Theta_i'| \\
\leq & 2 |\alpha_i| \left( |\alpha_{i+k}| + |\alpha_{i-k}| \right) =: \Sigma_i,
 \end{align*}
 with the convention $\alpha_{-1} = -1$ and $\alpha_{k} = 0$ for $k \leq -2$. 
As such:
$$ \sum_i \Sigma_i^2 \ll \sum_{j \geq -1} |\alpha_j|^4 =: \Sigma^2\ .$$
Invoking McDiarmid's inequality, there are constants $C, c>0$ such that:
$$ \forall x>0, \ \P\left( |F_{n,k}| \geq \Sigma x \ | \ \Mc \right) \leq C e^{-c x^2} \ .$$
Classically, these subgaussian tails translate to bounds on moments:
$$    \E\left( e^{\sigma |F_{n,k}|} \ | \ \Mc \right)
 \leq 1 + \int_0^\infty dx \ e^{x} \P\left( \sigma |F_{n,k}| \geq x| \Mc \right)
 \ll e^{c \sigma^2 \Sigma^2} \ 
$$
with a possibly different constant $c$.
We finish using the H\"older's inequality with $ \sum_k \frac{1}{p_k} = 1$:
$$   \E\left( e^{\sigma \sum_{k=1}^\infty |z|^k |F_{n,k}| } \right)
\leq \E\left( \prod_{k} \E\left( e^{\sigma |z|^k |F_{n,k}| p_k} | \Mc \right)^{\frac{1}{p_k}} \right)
\leq \E\left( e^{c \sigma^2 \Sigma^2 \sum_k |z|^{2 k} p_k} \right)
< \infty$$
the finiteness coming from the assumption \eqref{eq:Sigma2_hypothesis}. All constants  involved here are uniform in $z$ on any compact subset of $\D$. 

Now, in order to prove uniform convergence of $\log \Phi_n^*$, consider the Hilbert space $B = L^2( \rho \partial \D , d\theta)$ of square-integrable functions on the circle of radius $\rho <1$. It is easy to see that $(\log \Phi_n^*(\rho \ \cdot))_{n \geq 0}$ is a $B$-valued martingale. We shall study its convergence in $B$. One could invoke the general theory of martingales in Banach spaces (for e.g \cite{P16}), but for the reader's convenience, let us explain why a Hilbert space such as $B = L^2( \rho \partial \D , d\theta)$ does not require such a machinery.

Thanks to bounds on moments, we have
$$ \sup_{n \geq 0} \E\left( \int_{\rho \partial \D} |\log \Phi_n^*|^2 \right) < \infty \ ,$$
hence the square of the $L^2( \rho \partial \D, d\theta)$ norm of $\log \Phi_n^*(\cdot)$ is a scalar submartingale (Remark \cite[1.12]{P16}), which is also $L^2(\Omega, \mathcal{B}, \P)$-bounded, hence convergent. Because we are dealing with holomorphic functions, the Banach norm in $B$ dominates the $C^0$-norm (and $C^1$, $C^2$, $\dots$) in smaller discs (because of Cauchy's formula), and for all $\rho < 1$, the square of the $C^1$-norm of $\log \Phi_n^*(\cdot)$ in the disc $\rho \D$  is also a convergent submartingale (a.s and in $L^2(\Omega, \mathcal{B}, \P)$), because the supremum of submartingales is a submartingale. By Ascoli-Arzela theorem, $\left( \log \Phi_n^* \right)_{n \in \N}$ is relatively compact (a.s.) as a family of continuous function on $\rho \D$.  To prove the a.s. uniform convergence of this sequence on $\rho \D$, it is then enough to check that the limit of any subsequence is uniquely determined.  This last statement is due to the a.s. convergence of each Fourier coefficient of  $z \mapsto \log \Phi_n^*(\rho z) $, which is a martingale,  bounded in $L^2(\Omega, \mathcal{B}, \P)$. 
We deduce that $\left( \log \Phi_n^* \right)_{n \in \N}$ a.s. converges uniformly on compact sets of $\D$. 
\end{proof}

\section{Beginning of the proof of the Main Theorem \ref{thm:main}}
\label{section:proofmain}
As explained while announcing the structure of the paper, we made the choice of giving a full proof of the Main Theorem \ref{thm:main}, at the cost of admitting some intermediate results. The missing ingredients are condensed in Lemma \ref{lemma:RenormalizedGMC_L1Limit} and Lemma \ref{lemma:OmegaRhoControl}. And both lemmas are formulated in this section, when needed. 

We first introduce the following quantity: 
\begin{align}
\label{def:M_infty}
M_\infty := & \prod_{j=0}^{\infty} \left( 1 - \left|\alpha_j\right|^2 \right)^{-1} e^{ - \frac{2}{\beta (j+1)} } \ .
\end{align}
This product is a.s. convergent. Indeed, let us consider the logarithm and truncate the series at rank $n$:
\begin{equation}
\log M_n := \sum_{j=0}^{n} \left(-\log\left( 1 - \left|\alpha_j\right|^2 \right) - \frac{2}{\beta (j+1)} \right)\ . \label{martingaleN}
\end{equation}
Because of \eqref{eq:log_beta}, $(\log M_n)_{n \geq 0}$ is an $\F$-martingale. Moreover, this martingale is bounded in $L^2$.
Indeed for $x \in [0,1)$, we have, after considering separately the cases $x \in [0,1/2]$ and $x \in [1/2,1)$,  
$$\log^2(1-x) =\mathcal{O}( x^2 (1-x)^{-\beta/10} ),$$
 and then 
\begin{align}\mathbb{E} [ \log^2 \left( 1 - \left|\alpha_j\right|^2 \right) ] 
& =  \beta_j \int_0^1  (1-x)^{\beta_j- 1} \log^2 (1-x) dx
\nonumber \\ & \ll   \beta_j  \int_0^{1} (1-x)^{\beta_j - 1 - (\beta/10)} x^2 dx
\nonumber \\ & =  \beta_j \, \frac{ \Gamma(\beta_j - (\beta/10)) \Gamma(3)}{\Gamma (\beta_j - (\beta/10) + 3)} 
 \nonumber \\& \leq \frac{2  \beta_j} {(\beta_j - (\beta/10))^3} \ll \frac{1}{j^2}.  \label{boundlogcarre}
\end{align} 
Hence, the $\F$-martingale $(\log M_n)_{n \geq 0}$ is a.s. convergent. 

Now, we fix a nonnegative smooth function $f: \partial \D \rightarrow \R_+$ on the unit circle. We will be interested in the integral of $f$ with respect to several random measures on the unit circle, and its conditional expectation given $\Fc_n$ for  $n \geq 0$. 

\medskip

As formulated in the beginning of the paper (see Question \ref{question:whatIsMu}), recall that the object of interest is $\mu^\beta$. A natural approximation is the Bernstein-Szeg\"o approximation of $\mu^\beta$ (see \cite[Theorem 1.7.8 p.95]{Sim05-1}) which we denote by $\mu^\beta_{n}$. By definition:
\begin{align}
\label{def:mu_BZ}
	\mu^\beta_{n}(d\theta) = & 
	\frac{d\theta}{2\pi} 
    \frac{ \prod_{j=0}^{n-1}\left( 1 - |\alpha_j|^2 \right) }
         { \left| \Phi_n^*( e^{i \theta} ) \right|^2  } \ .	  
\end{align}
The limit $\mu_n^\beta(f) \stackrel{n \rightarrow \infty}{\longrightarrow} \mu^\beta(f) $ holds {\it surely} by virtue of the previously referenced \cite[Theorem 1.7.8 p. 95]{Sim05-1}: this is a deterministic statement regarding the Bernstein-Szeg\"o approximation of a probability measure on the circle. Let us also prove that we have convergence in all $L^p(\Omega, \mathcal{B}, \P)$. We 
notice that for all smooth $f$, $(\mu^\beta_{n}(f))_{n \geq 0}$ is an $(\F, \P)$-martingale by integrating against $f$ the point-wise martingale:
$$ \E\left( \frac{\prod_{j=0}^{n}(1-|\alpha_j|^2)}{|\Phi_{n+1}^*(e^{i \theta})|^2} | \Fc_n \right)
 = \frac{\prod_{j=0}^{n-1}(1-|\alpha_j|^2)}{|\Phi_n^*(e^{i \theta})|^2} \ .
$$
The above computation uses crucially that $|Q_j(e^{i \theta})| = 1$, where we recall that 
$$Q_j(z) = z \, \frac{\Phi_n(z)}{\Phi^*_n(z)}. $$
Now, since $\mu^\beta_n$ is constructed to be a probability measure:
$$ \mu^\beta_{n}(f) \leq |f|_\infty \ ,$$
and $(\mu^\beta_{n}(f))_{n \geq 0}$ ends up being a bounded martingale! Then, the convergence (almost sure and in all $L^p(\Omega, \mathcal{B}, \P)$) holds by Doob's martingale convergence theorem. In any case:
$$ \E\left( \mu^\beta(f) \ | \ \Fc_n \right) = \mu_n^\beta(f) \ .$$

\medskip

The second quantity of interest is related to the integral of $f$ with respect to the Gaussian multiplicative chaos constructed from the Gaussian 
field $\log \Phi^*_{\infty}$. We set:
\begin{align}
\label{def:X}
       X_{r,n}(f)
:= & \ \E \left[ \frac{1}{M_{\infty}} GMC_r^\gamma(f) \ \Big| \ \Fc_n \right]\\
 = & \ \int \frac{d\theta}{2\pi}
            f(e^{i\theta})
        \ 
        \E \left[ \frac{1}{M_{\infty}}
                  \frac{(1-r^2)^{\frac{2}{\beta}}}{|\Phi_\infty^*(r e^{i\theta})|^2}
           \ \Big| \ \Fc_n \right] 
\nonumber
\end{align}
where 
$$GMC_r^{\gamma} = (1-r^2)^{2/\beta} |\Phi^*_{\infty} (r e^{i \theta}) |^{-2} d \theta.$$
Now all the positive moments of $M_{\infty}^{-1}$ are finite and the Gaussian multiplicative chaos can be  obtained as the limit 
$$GMC^{\gamma}(f) = \lim_{r \rightarrow 1-} GMC_{r}^{\gamma}(f) \ ,$$
in $L^1(\Omega, \mathcal{B}, \P)$. In fact, we have:
\begin{lemma}
\label{lemma:RenormalizedGMC_L1Limit}
$$
  \E \left[ \frac{1}{M_{\infty}} |GMC_r^{\gamma} (f) - GMC^{\gamma}(f) | \right]
  \underset{r \rightarrow 1^-}{\longrightarrow} 0 \ .
$$
\end{lemma}
This lemma is the first ingredient which is admitted for now, and proven in Section \ref{section:proofmain2}. Therefore, using the fact that the  conditional expectation is a contraction on $L^1(\Omega, \mathcal{B}, \P)$,   we get the $L^1(\Omega, \mathcal{B}, \P)$ convergence:
\begin{align}
\label{eq:X_CV}
     \lim_{r \rightarrow 1^-} X_{r, n}(f)
 = & \E \left[ \frac{1}{M_{\infty}} GMC^\gamma(f) \ \Big| \ \Fc_n \right] \ .
\end{align}

The entire point of the proof consists in  relating the two measures defined by  \eqref{def:mu_BZ} and \eqref{def:X}. 

From the product \eqref{def:M_infty} and the fact that 
$$ \left( 1 - r^2 \right)^{\frac{2}{\beta}} = e^{-\frac{2}{\beta} \sum_{j=0}^\infty \frac{r^{2j+2}}{j+1} } \ ,$$
$$ \Phi^*_{\infty}(z) = \prod_{j=0}^{\infty} (1 - \alpha_j Q_j(z)) \ ,$$
for all $z \in \D$, we obtain:
\begin{align*}
    X_{r,n}(f)
= & \int_0^{2 \pi} \frac{d \theta}{2 \pi} f(e^{i\theta})
				   \E\left[ 
                   \prod_{j=0}^{\infty} \frac{1 - |\alpha_j|^2}
                                          {|1 - \alpha_j Q_j(re^{i \theta})|^{2}}
                                          e^{ \frac{2}{\beta(j+1)}(1-r^{2j+2}) }
                   \ | \ \Fc_n \right] \\
= & \int_0^{2 \pi} \frac{d \theta}{2 \pi} f(e^{i\theta})
                   \frac{\prod_{j=0}^{n-1} (1 - |\alpha_j|^2) e^{ \frac{2}{\beta(j+1)}(1-r^{2j+2}) }}
                        {|\Phi_n^*(r e^{i\theta})|^2}\\
  & \quad \quad \times \E\left[ 
                   \prod_{j=n}^{\infty} \frac{1 - |\alpha_j|^2}
                                          {|1 - \alpha_j Q_j(re^{i \theta})|^{2}}
                                          e^{ \frac{2}{\beta(j+1)}(1-r^{2j+2}) }
                   \ | \ \Fc_n \right] \\
= & \int_0^{2 \pi} \frac{d \theta}{2 \pi} f(e^{i\theta}) \ 
                   R^{(0, n-1)}(\theta) \ 
				   \E\left[ R^{(n, \infty)}(\theta) | \Fc_n \right] \ ,
\end{align*}
where $R^{(0, n-1)}(\theta)$ and $R^{(n, \infty)}(\theta)$ represent respectively the products from $0$ to $n-1$ and from $n$ to $\infty$.

Let $L$ be a compact set of $\D^n$ and $\mathcal{A}_{L}$ the $\Fc_n$-measurable event corresponding to the fact that $(\alpha_0, \dots, \alpha_{n-1}) \in L$. 
Under the event $\mathcal{A}_L$, because the Verblunsky coefficients are away from the unit circle, 
$$
 \sup_{\theta} \left| 
   R^{(0,n-1)}(\theta)
 - \frac{d \mu_n^\beta}{d\theta}(\theta)
 \right|
 \leq 
 \sup_{\theta} \left| 
   |\Phi_n^*(r e^{i\theta})|^{-2}
   \prod_{j=0}^{n-1} e^{ \frac{2}{\beta(j+1)}(1-r^{2j+2}) }
 - |\Phi_n^*(e^{i\theta})|^{-2}
 \right|  
 $$
is bounded by a quantity $c_r$ depending only on $\beta, r$ and $L$, and tending to zero when $r$ goes to $1$ for $\beta, L$ fixed. Indeed 
$$\prod_{j=0}^{n-1} (1 - \alpha_j Q_j(z)) = \Phi^*_n(z)$$
is (by the Szeg\"o recursion) a polynomial in $z$, the $\alpha_j$'s and their conjugate, and then it is uniformly Lipschitz in $z \in \overline{\D}$ for $n$ fixed.  On the event $\mathcal{A}_L$, 
$|\Phi^*_n(z)|^{-2}$ is also uniformly Lipschitz in $z \in \overline{\D}$  for $n$ and $L$ fixed since  $\Phi^*_n(z)$ does not vanish, and then is uniformly away from zero by compactness of the sets  $L$ and $\overline{\D}$.

As such, for any constant $K_{\beta}$ - to be determined later:
\begin{align*}
     & \E \left[ \mathds{1}_{\mathcal{A}_L} \left|X_{r,n}(f) - K_{\beta} \mu_n^\beta(f) \right| \right]\\
\leq & \,  c_r \ K_{\beta} |f|_\infty 
       + \E \left[ \mathds{1}_{\mathcal{A}_L} \left|
          X_{r,n}(f)
		- K_{\beta} \int_0^{2 \pi} \frac{d \theta}{2 \pi} f(e^{i\theta}) R^{(0, n-1)} (\theta)
		  \right| \right]\\
=    &  \,  c_r  \ K_{\beta} |f|_\infty 
       + |f|_\infty 
         \E \left[ \mathds{1}_{\mathcal{A}_L} 
          \left|
          \int_0^{2 \pi} \frac{d \theta}{2 \pi} R^{(0, n-1)} (\theta) \ 
		  \left( \E\left[ R^{(n, \infty)}(\theta) | \ \Fc_n \right] - K_{\beta} \right)
		  \right|
		  \right] \\
\leq & \,  c_r  \ K_{\beta} |f|_\infty 
       + |f|_\infty 
         \int_0^{2 \pi} \frac{d \theta}{2 \pi}
           \E \left[
           R^{(0, n-1)} (\theta) \ 
		  \left| \E\left[ R^{(n, \infty)}(\theta) | \ \Fc_n \right] - K_{\beta} \right|
		  \right] \ ,
\end{align*}

From now on, we shall use, for $r \in (0,1)$ and $\theta \in \R$, a new probability measure $  \Q_{r, \theta}$, equivalent to $\P$ and defined by: 
\begin{align}
\label{def:Qr}
     \frac{d\Q_{r,\theta}}{d\P} = & \frac{ \prod_{j=0}^\infty \left( 1 - |\alpha_j Q_j(r e^{i\theta})|^2 \right) }
                              { \left| \Phi_\infty^*(r e^{i \theta}) \right|^2  } \ .
\end{align}
 The intuition behind introducing the measure $\Q_{r,\theta}$ is that its density is a multiplicative martingale, which is the analogue of the exponential martingale in the context of branching random walks or when constructing GMC from an approximation of the log-correlated field by independent increments. Furthermore, it is the natural (multiplicative) compensation of $\left| \Phi_\infty^*(r e^{i \theta}) \right|^{-2}$ when considering the filtration generated by Verblunsky coefficients.

 If the reference angle $\theta$ is not indicated, it means that we consider $\theta=0$ (note that all angles play symmetric roles by rotational invariance) and we denote the measure by $\Q_r$. Moreover, $\Q_{r}$ will be simply denoted by $\Q$ if there is no possible ambiguity. 
 In order to check that the  probability measure $\Q_{r, \theta}$ is well-defined, we first check that  for all $n \geq 0$, because of the Szeg\"o recursion: 
$$\frac{ \prod_{j=0}^{n-1}  \left( 1 - |\alpha_j Q_j(r)|^2 \right) }
                              { \left| \Phi_n^*(r ) \right|^2  } 
                              = \prod_{j=0}^{n-1} \frac{1 - |\alpha_j Q_j(r)|^2}{|1 - \alpha_j Q_j(r)|^2}$$

From the rotational invariance of $\alpha_j$ and its independence with $\Fc_j$, we have 
$$\E \left[ \frac{1 - |\alpha_j Q_j(r)|^2}{|1 - \alpha_j Q_j(r)|^2} \, \big| |\alpha_j|, \Fc_j \right]
= \frac{1 - u^2}{2 \pi} \int_0^{2 \pi} \frac{d \theta}{ |1 - u e^{i \theta}|^2}$$
where $u = |\alpha_j Q_j(r)| < 1$. Computing the integral (see Lemma \ref{lemma:circleMoment} below for more detail) gives that the last conditional expectation is equal to $1$, and then 
$$\left(\frac{ \prod_{j=0}^{n-1}  \left( 1 - |\alpha_j Q_j(r)|^2 \right) }{ \left| \Phi_n^*(r ) \right|^2  }  \right)_{n \geq 0}$$
is a $(\F,\P)$-martingale. Moreover, this martingale is bounded in all $L^p( \Omega, \mathcal{B}, \P)$ spaces, because it is dominated by $ (\left| \Phi_n^*(r ) \right|^{-2})_{n \geq 0}$, which has been proven to be bounded in $L^p( \Omega,  \mathcal{B}, \P)$. Hence, the martingale converges a.s. and in all $L^p( \Omega,  \mathcal{B},\P)$, and its limit has expectation $1$. It is then the density of a probability measure with respect to $\P$. 

Note that we have, by the martingale property, 
$$\frac{d\Q_r}{d\P} \,_{\big|\Fc_n} = \frac{ \prod_{j=0}^{n-1} \left( 1 - |\alpha_j Q_j(r)|^2 \right) }{ \left| \Phi_n^*(r ) \right|^2} $$
and then, for a $\Fc_{n+1}$-measurable quantity $X$, and a $\Fc_n$-measurable  quantity $Y$, both nonnegative, 
\begin{align*}
\E^{\Q_r} [X Y] & = \E^{\P} \left[ XY \frac{\prod_{j=0}^{n}  \left( 1 - |\alpha_j Q_j(r)|^2 \right) }{ \left| \Phi_{n+1}^*(r ) \right|^2  } \right]
\\ & =  \E^{\P} \left[ Y\frac{ \prod_{j=0}^{n-1}  \left( 1 - |\alpha_j Q_j(r)|^2 \right) }{ \left| \Phi_{n}^*(r ) \right|^2  } 
\E^{\P} \left[ X \frac{1 - |\alpha_n Q_n(r)|^2}{|1 - \alpha_n Q_n(r)|^2} | \Fc_n \right] \right]
\\ & = \E^{\Q_r}  \left[ Y
\E^{\P} \left[ X \frac{1 - |\alpha_n Q_n(r)|^2}{|1 - \alpha_n Q_n(r)|^2} | \Fc_n \right] \right],
\end{align*}
which gives 
\begin{equation}
\E^{\Q_r} [ X | \Fc_n] 
= \E^{\P}  \left[ X \frac{1 - |\alpha_n Q_n(r)|^2}{|1 - \alpha_n Q_n(r)|^2} | \Fc_n \right].
\label{conditionalexpectation}
\end{equation} 
 Similarly, if $X$ is a nonnegative $\mathcal{B}$-measurable variable, 
\begin{equation}
\E^{\Q_r} [ X | \Fc_n] 
= \E^{\P}  \left[ X \prod_{j = n}^{\infty} \frac{1 - |\alpha_j Q_j(r)|^2}{|1 - \alpha_j Q_j(r)|^2} | \Fc_n \right].
\label{conditionalexpectation2}
\end{equation}

Using the problem's rotational invariance, and \eqref{conditionalexpectation2}, we deduce that: 
\begin{align*}
  & \int_0^{2 \pi} \frac{d \theta}{2 \pi}
	\E \left[
	R^{(0, n-1)} (\theta) \ 
	\left| \E\left[ R^{(n, \infty)}(\theta) | \ \Fc_n \right] - K_{\beta} \right|
	\right]\\
= & \E \left[
	R^{(0, n-1)} (\theta=0) \ 
	\left| \E\left[ R^{(n, \infty)}(\theta=0) | \ \Fc_n \right] - K_{\beta} \right|
	\right]\\
= & \E \left[
		\prod_{j=0}^{n-1} \frac{1 - |\alpha_j|^2}
                          {|1 - \alpha_j Q_j(r)|^{2}}
                          e^{ \frac{2}{\beta(j+1)}(1-r^{2j+2}) }
	\left| \E\left[ \prod_{j=n}^{\infty}
	 \frac{1 - |\alpha_j|^2}
     	  {|1 - \alpha_j Q_j(r)|^{2}}
	 e^{ \frac{2}{\beta(j+1)}(1-r^{2j+2}) }
      | \ \Fc_n \right] - K_{\beta} \right|
	\right] \\
= & \E^\Q \left[
		\prod_{j=0}^{n-1} \frac{1 - |\alpha_j|^2}
                          {1 - |\alpha_j Q_j(r)|^{2}}
                          e^{ \frac{2}{\beta(j+1)}(1-r^{2j+2}) }
	\left| \E^\Q\left[ \prod_{j=n}^{\infty}
	 \frac{1 - |\alpha_j|^2}
     	  {1 - |\alpha_j Q_j(r)|^{2}}
	 e^{ \frac{2}{\beta(j+1)}(1-r^{2j+2}) }
      | \ \Fc_n \right] - K_{\beta} \right|
	\right] \\
\leq &  e^{ \sum_{j=0}^{n-1} \frac{2}{\beta(j+1)}(1-r^{2j+2}) }
        \E^\Q \left[
	\left| \E^\Q\left[ \prod_{j=n}^{\infty}
	 \frac{1 - |\alpha_j|^2}
     	  {1 - |\alpha_j Q_j(r)|^{2}}
	 e^{ \frac{2}{\beta(j+1)}(1-r^{2j+2}) }
      | \ \Fc_n \right] - K_{\beta} \right|
	\right] \ .
\end{align*}
The last inequality comes from the fact that $1 - |\alpha_j|^2 \leq 1 - |\alpha_j Q_j(r)|^2$ for $0 \leq j \leq n-1$, because $|Q_j(r)| \leq 1$. 

We now introduce the following quantities, for $N \geq n$: 
$$ \omega_{r,n,N} :=  \frac{2}{\beta}\sum_{k=n}^{N-1} \frac{|Q_k(r)|^2 - r^{2k+2}}{k+1} \ ,$$
and 
$$ \rho_{r,n,N} := \sum_{k=n}^{N-1} \left( - \log \left( \frac{1 - |\alpha_k Q_k(r)|^2}{1 - |\alpha_k|^2} \right)
+ \frac{2}{\beta} \frac{1 - |Q_k(r)|^2}{k+1} \right) \ ,$$
and their respective upper limits  $\omega_{r,n}$ and $\rho_{r,n}$ when $N$ goes to infinity. 
Note that in Proposition \ref{proposition:boundOmegaRho}, we  will show  that these upper limits are in fact limits, i.e. 
$$\omega_{r,n} :=  \frac{2}{\beta}\sum_{k=n}^{\infty} \frac{|Q_k(r)|^2 - r^{2k+2}}{k+1} \ ,$$
$$\rho_{r,n} := \sum_{k=n}^{\infty} \left( - \log \left( \frac{1 - |\alpha_k Q_k(r)|^2}{1 - |\alpha_k|^2} \right)
+ \frac{2}{\beta} \frac{1 - |Q_k(r)|^2}{k+1} \right).$$
In any case, we deduce, from the computation above: 
\begin{align*}
  & \int_0^{2 \pi} \frac{d \theta}{2 \pi}
	\E \left[
	R^{(0, n-1)} (\theta) \ 
	\left| \E\left[ R^{(n, \infty)}(\theta) | \ \Fc_n \right] - K_{\beta} \right|
	\right]\\
\leq &  e^{ \sum_{j=0}^{n-1} \frac{2}{\beta(j+1)}(1-r^{2j+2}) }
        \E^\Q \left|
              \E^\Q\left[ e^{\rho_{r,n}+\omega_{r,n}} | \Fc_n \right] - K_{\beta}
              \right| \ .
\end{align*}
In the end, since $c_r$ goes to zero as $r \rightarrow 1^-$:
$$ \limsup_{r \rightarrow 1^-}
   \E \left[ \mathds{1}_{\mathcal{A}_L} \left|X_{r,n}(f) - K_{\beta} \mu_n^\beta(f) \right| \right]
   \leq |f|_\infty 
        \limsup_{r \rightarrow 1^-}
        \E^\Q \left|
              \E^\Q\left[ e^{\rho_{r,n}+\omega_{r,n}} | \Fc_n \right] - K_{\beta}
              \right| \ .$$

This is where we invoke:
\begin{lemma}
\label{lemma:OmegaRhoControl}
There exists a constant $K_\beta$ such that:
$$ \E^\Q \left|
         \E^\Q\left[ e^{\rho_{r,n}+\omega_{r,n}} | \Fc_n \right] - K_\beta
         \right|
   \stackrel{r \rightarrow 1^-}{\longrightarrow} 0 \ .$$
\end{lemma}
This lemma is the second ingredient which is admitted for now, and proved in Section \ref{section:proofmain2}. We deduce:
 $$\E [ \mathds{1}_{\mathcal{A}_L} |X_{r,n}(f) - K_{\beta} \mu_n^\beta(f) | ]  \underset{r \rightarrow 1}{\longrightarrow} 0.$$
 From this limit, combined with \eqref{eq:X_CV}, the triangle inequality and the fact that the new quantities involved do not depend on $r$ anymore, we have
$$ \E \left[  \mathds{1}_{\mathcal{A}_L} \left|\E \left[ \frac{1}{M_\infty} GMC^{\gamma}(f) | \Fc_n \right] - K_{\beta} \mu_n^\beta(f) \right| \right] = 0 \ .$$
This is equivalent to
$$ K_{\beta} \mu_n^\beta(f) =  \E \left[ \frac{GMC^{\gamma}(f)}{M_\infty} | \Fc_n \right]$$
almost surely on $\mathcal{A}_L$ and then almost surely without extra restriction as the sample space $\Omega$ can be written as a countable union of events of the form $\mathcal{A}_L$. 
We have seen, just after equation \eqref{def:mu_BZ}, that the left-hand side of the equality is a bounded martingale, a fortiori uniformly integrable. cp
Taking the limit when $n$ goes to infinity, we get 
$$K_{\beta} \mu^\beta(f) =  \E \left[ \frac{GMC^{\gamma}(f)}{M_\infty} | \Fc_{\infty} \right]$$
almost surely, for  $\Fc_{\infty}$ equal to the $\sigma$-algebra generated by all the Verblunsky coefficients. Now, by construction, 
$ GMC^{\gamma}(f)/M_\infty$ is  $\Fc_{\infty}$-measurable, and we easily deduce that the measures   $GMC^{\gamma}$ and  $M_{\infty}K_{\beta} \mu^{\beta}$ 
almost surely coincide. Since the expectation of the total mass of $GMC^{\gamma}$ is $1$, 
we have 
$$GMC^{\gamma} =\frac{ M_{\infty}}{ \E [ M_{\infty}] }\mu^{\beta}.$$

Now, recalling the expression 
$$ M_n :=  \prod_{j=0}^{n-1} \left( 1 - \left|\alpha_j\right|^2 \right)^{-1} e^{ - \frac{2}{\beta (j+1)} } \ ,$$
we have that
$$\frac{M_n}{\E[M_n]} = \prod_{j=0}^{n-1} \frac{\left( 1 - \left|\alpha_j\right|^2 \right)^{-1} }{\E \left[ \left( 1 - \left|\alpha_j\right|^2 \right)^{-1} \right]}
= \prod_{j=0}^{n-1} ( 1 - \left|\alpha_j\right|^2 )^{-1} \left(1 - \frac{2}{\beta(j+1)} \right)$$
is a martingale in $n$. From straightforward computations on the Beta distribution, it is also bounded in $L^p$ for some $p > 1$, and a fortiori uniformly integrable. Therefore, we have that 
$$\E[M_n] = \prod_{j=0}^{n-1}  \left(1 - \frac{2}{\beta(j+1)} \right)^{-1} e^{ - \frac{2}{\beta (j+1)} }$$
converges to a limit $\mathcal{M}$ when $n$ goes to infinity, and then the almost sure and $L^1$ limit  of the martingale is $M_{\infty}/\mathcal{M}$, and necessarily $\mathcal{M} =  \E [ M_{\infty}]$ since the expectations of the martingale and its $L^1$ limit are equal to $1$. Hence,
$$\frac{ M_{\infty}}{ \E [ M_{\infty}] } = \lim_{n \rightarrow \infty} \frac{M_n}{\E[M_n]} 
= \lim_{n \rightarrow \infty}\prod_{j=0}^{n-1} ( 1 - \left|\alpha_j\right|^2 )^{-1} \left(1 - \frac{2}{\beta(j+1)} \right) \ ,
$$
which is the quantity $C_0$ introduced in the Theorem \ref{thm:main}. We deduce that 
$$ GMC^{\gamma} = C_0 \mu^{\beta} $$
almost surely, and a fortiori in distribution, which proves Theorem \ref{thm:main}.

\section{Some useful estimates}
\label{section:approximations}

We first recall that 
\begin{align}
\label{def:Q_j}
Q_j(z) = & \frac{z \Phi_j(z)}{ \Phi_j^*(z) } \ . 
\end{align}
Here are a few properties we will require later: 
\begin{proposition}
\label{ppty:Qj}
The following holds, for all $j \geq 0$:
\begin{itemize}
 \item $Q_j(z)$ is a Blaschke product, with modulus one on $\partial \D$.
 \item $|Q_j(z)| \leq 1$ for all $z \in \D$.
 \item If $|z|=r$, the following recurrence holds:
       \begin{align}
       \label{eq:recurrenceQj}
       1-|Q_{j+1}(z)|^2 & = (1-r^2) + r^2 \frac{(1-|\alpha_j|^2)(1-|Q_{j}(z)|^2)}{|1-\alpha_j Q_j(z)|^2} \ .
       \end{align}
 \item We have the conditional expectation bound, for all $n \geq 0$: 
       $$ r^{2j} |Q_n(z)|^2 \leq \E [|Q_{n+j}(z)|^2 | \Fc_n] \leq 1 \ .$$
\end{itemize}
\end{proposition}
\begin{proof}
The first point is due to the fact that
$$ Q_j(z) = z \prod_{\omega, \ \Phi_j(\omega) = 0} \frac{ z - \omega}{1 - \overline{\omega} z} \ ,$$
as a consequence of \eqref{def:Q_j}, and to the fact that all the roots of $\Phi_j$ have modulus strictly smaller than $1$.

The second point is due to the fact that $z \mapsto |Q_j(z)|$ is subharmonic, because $Q_j$ is holomorphic and the absolute value is convex.

For the third point, we  start by the recurrence relation \eqref{def:matrix_szego} and we write:
\begin{align*}
     1 - |Q_{j+1}(z)|^2 
 = & 1 - r^2 \left| \frac{\Phi_{j+1}(z)}{\Phi_{j+1}^*(z)}\right|^2\\
 = & 1 - r^2 \left| \frac{z\Phi_{j}(z) - \overline{\alpha_j} \Phi_{j}^*(z)}
                         { \Phi_{j}^*(z) - \alpha_j z \Phi_{j}(z)} \right|^2\\
 = & 1 - r^2 \left| \frac{ Q_j(z) - \overline{\alpha_j} }
                         { 1 - \alpha_j Q_j (z)         } \right|^2\\
 = & (1-r^2) + r^2 \frac{(1-|\alpha_j|^2)(1-|Q_{j}(z)|^2)}{|1-\alpha_j Q_j(z)|^2} \ .
\end{align*}
For the last point, taking conditional expectation in the recurrence yields:
\begin{align*}
       \E\left( 1-|Q_{j+1}(z)|^2 \ | \Fc_j \right)
 =   & (1-r^2) + r^2 \E\left( \frac{1-|\alpha_j|^2}{1-|\alpha_j Q_j(z)|^2} | \Fc_j \right) (1-|Q_{j}(z)|^2) \\
\leq & (1-r^2) + r^2 (1-|Q_{j}(z)|^2) \\
=    & 1 - r^2 |Q_{j}(z)|^2 \ . 
\end{align*}
The result follows from the previous inequality by induction.
\end{proof}

\subsection{Moment estimates}

\medskip

The following lemma will be useful:
\begin{lemma}
\label{lemma:circleMoment}
For $\Theta$ uniform random variable on $[0, 2\pi]$ and $u \in \D$, we have, for all $\lambda \in \R$,
\begin{align}
\label{eq:circleMoment}
\ \E\left( \left| 1 - e^{i \Theta} u\right|^{-2 \lambda} \right) & =
\sum_{k=0}^\infty  \binom{\lambda+k-1}{k}^2 |u|^{2k} \, 
\end{align}
\begin{align}
\label{eq:circleMomentBound}
    \E\left( \left| 1 - e^{i \Theta} u\right|^{-2 \lambda} \right)
= & 1 + \lambda^2 |u|^{2} + \Oc_\lambda\left( \frac{|u|^{4}}{\left( 1-|u|^2 \right)^{2|\lambda|}}\right)
\end{align}
and  in the case where $|\lambda| \leq 1$,
\begin{align}
\label{eq:circleMomentBound2}
    \E\left( \left| 1 - e^{i \Theta} u\right|^{-2 \lambda} \right)
= & 1 + \lambda^2 |u|^{2} + \Oc \left( \frac{|u|^{4}}{1-|u|^2}\right)
\end{align}

\end{lemma}
\begin{proof}
For \eqref{eq:circleMoment}, we have by series expansion:
\begin{align*}
    \E\left( \left| 1 - e^{i \Theta} u\right|^{-2 \lambda} \right)
= & \E\left( \sum_{k,\ell=0}^\infty \binom{-\lambda}{k} e^{i k \Theta} u^k
								 \overline{ \binom{-\lambda}{\ell} e^{i \ell \Theta} u^{\ell}}
      \right)\\
= & \sum_{k=0}^\infty \binom{-\lambda}{k}^2 |u|^{2k} \\
= & \sum_{k=0}^\infty  \binom{\lambda+k-1}{k}^2 |u|^{2k}.
\end{align*}
The two first terms of the sum are $1$ and $\lambda^2 |u|^2$. 
If $\lambda$ is a nonpositive integer, only finitely many  terms are non-zero. Otherwise, for $k \geq 2$, the coefficient of $|u|^{2k}$ is 
$$\left(\frac{ \Gamma (k + \lambda)}{ \Gamma (\lambda) \Gamma (k+1)} \right)^2
= \Oc_{\lambda} ( k^{2(\lambda - 1)} ),$$
whereas the coefficient of $|u|^{2k}$ in the expansion of $|u|^4 (1 - |u|^2)^{-2 |\lambda|}$ is 
$$(-1)^{k-2} \binom{-2|\lambda|}{k-2} = \binom{ 2 |\lambda| + k-3}{k-2} = \frac{ \Gamma(2 |\lambda| + k-2)}{ \Gamma(2 | \lambda|) \Gamma(k-1)}
\gg_{\lambda} k^{2 |\lambda| -1} \gg_{\lambda} k^{2(\lambda - 1)} $$
for $\lambda \neq 0$. 
This gives bound \eqref{eq:circleMomentBound} for $\lambda \neq 0$, and this bound is obvious for $\lambda = 0$. 

If $|\lambda| \leq 1$, we have 
$$ \binom{-\lambda}{k}^2
= \prod_{j=0}^{k-1} \left( \frac{ |-\lambda - j|}{1+j} \right)^2 \leq 1$$
which gives \eqref{eq:circleMomentBound2}.

\end{proof}

 The following estimates will be useful: 
\begin{proposition}
\label{properties:underQ}
For $\beta> 0$, $j$ large enough depending on $\beta$, and $Q_j = Q_j(r)$, we have  almost surely, for $\Q = \Q_r$: 
\begin{align}
  \label{eq:conditionalQj}
  & \E^\Q\left( 1 - |Q_{j+1}|^2 \Big | \Fc_j \right) \\
\nonumber
= & 1-r^2 + r^2 \left( 1 - |Q_{j}|^2 \right)
           \left( 1 - \frac{2}{\beta(j+1)}\left( 1 - |Q_{j}|^2 \right)
                    + \frac{4}{\beta(j+1)} |Q_{j}|^2 + \Oc \left( \frac{1}{(j+1)^2} \right) \right) \ .
\end{align}
\begin{align}
  \label{eq:conditionalVarQj}
    \Var^\Q\left( 1 - |Q_{j+1}|^2 \Big | \Fc_j \right) 
= & r^4 \left( 1 - |Q_{j}|^2 \right)^2
        \left( \frac{4 |Q_j|^2}{\beta(j+1)}
               + \Oc\left( \frac{1}{(j+1)^2}\right)
        \right) \ .
\end{align}
\begin{align}
  \label{eq:boundmomentorder4}
   \E^\Q\left(( |Q_j|^2 - |Q_{j+1}|^2 )^4\Big | \Fc_j \right) 
= &  \Oc \left( (1-r^2)^4 + \frac{1}{(j+1)^2} \right). 
\end{align}
We also have 
$$ \E^{\Q}\left( - 2 \Re  \log (1 - \alpha_j Q_j )\ | \Fc_{j} \right)
 = \frac{4}{\beta (j+1)} |Q_j|^2 + \Oc \left(\frac{1}{(j+1)^2} \right) \ .$$ 
$$ \Var^{\Q}\left( - 2 \Re \log (1 - \alpha_j Q_j )\ | \Fc_{j} \right)
 = \frac{4}{\beta (j+1)} |Q_j|^2 + \Oc \left(\frac{1}{(j+1)^2} \right) \ .$$ 
 We recall that the implicit constant in $\Oc$ depends only on $\beta$. 
\end{proposition}
\begin{proof}
For the first equation, taking the conditional expectation $\E^\Q( \ \cdot \ | \Fc_j )$ in \eqref{eq:recurrenceQj}, we only have to prove:
$$ \E^\Q \left( \frac{1-|\alpha_j|^2}{|1-\alpha_j Q_j|^2} | \Fc_j \right)
 = 1 - \frac{2}{\beta(j+1)}\left( 1 - |Q_{j}|^2 \right)
     + \frac{4}{\beta(j+1)} |Q_{j}|^2 + \Oc \left( \frac{1}{(j+1)^2} \right) \ .
$$
On the other hand, by using the  rotational invariance of $\alpha_j$ and its independence with 
$\Fc_j$, we get,  for all $\lambda \in \mathbb{R}$,
$$ \E [ |1 - \alpha_j Q_j|^{-2\lambda} \,  | \Fc_j, |\alpha_j|  ]
=  \E \left( |1 - e^{i \Theta} u|^{-2\lambda} \right) $$
with $u = | \alpha_j Q_j|$, and $\Theta$ a uniform random variable as in Lemma \ref{lemma:circleMoment}.

Now, using the change of measure \eqref{conditionalexpectation} and then \eqref{eq:circleMomentBound} for $\lambda=2$ and $u=|\alpha_j Q_j|$, we deduce 
\begin{align*}
   & \E^\Q \left( \frac{1-|\alpha_j|^2}{|1-\alpha_j Q_j|^2} | \Fc_j \right)\\
 = & \E \left( \frac{(1-|\alpha_j|^2)(1-|\alpha_j Q_{j}|^2)}{|1-\alpha_j Q_j|^4} | \Fc_j \right)\\
 = & \E \left( (1-|\alpha_j|^2)(1-|\alpha_j Q_{j}|^2)
               \left( 1 + 4 |\alpha_j Q_j|^2 + \Oc\left( \frac{|\alpha_j Q_j|^4}{(1-|\alpha_j Q_j|^2)^4} \right) 
               \right)  | \Fc_j \right)\\
 = & \E \left( 1-|\alpha_j|^2 -|\alpha_j Q_{j}|^2 + 4 |\alpha_j Q_j|^2  | \Fc_j \right)
     + \Oc\left( \E \frac{|\alpha_j|^4}{(1-|\alpha_j|^2)^4} \right)\\
 = & 1 - \frac{2}{\beta(j+1)}\left( 1 - |Q_{j}|^2 \right)
     + \frac{4}{\beta(j+1)} |Q_{j}|^2
     + \Oc\left( \E \frac{|\alpha_j|^4}{(1-|\alpha_j|^2)^4} \right) \ .
\end{align*}
The first estimate holds as $\E \frac{|\alpha_j|^4}{(1-|\alpha_j|^2)^4} = \Oc( \frac{1}{(j+1)^2})$ for $j$ large enough depending on $\beta$. 

For the second equation, the proof is similar, by taking the variance under $\Q$ conditionally to $\Fc_j$ in  \eqref{eq:recurrenceQj}. 
We only have to prove:
$$ \Var^\Q \left( \frac{1-|\alpha_j|^2}{|1-\alpha_j Q_j|^2} | \Fc_j \right)
 = \frac{4 |Q_j|^2}{\beta(j+1)}
   + \Oc\left( \frac{1}{(j+1)^2}\right)  \ .
$$
The required expansion uses \eqref{eq:circleMomentBound} for $\lambda=3$:
\begin{align*}
   & \Var^{\Q} \left( \frac{1-|\alpha_j|^2}{|1-\alpha_j Q_j|^2} | \Fc_j \right)\\
 = & \E \left( \frac{(1-|\alpha_j|^2)^2(1-|\alpha_j Q_{j}|^2)}{|1-\alpha_j Q_j|^6} | \Fc_j \right)
   - \E^\Q \left( \frac{1-|\alpha_j|^2}{|1-\alpha_j Q_j|^2} | \Fc_j \right)^2\\
 = & \E \left( (1-|\alpha_j|^2)^2 (1-|\alpha_j Q_{j}|^2)
               \left( 1 + 9 |\alpha_j Q_j|^2 + \Oc\left( \frac{|\alpha_j Q_j|^4}{(1-|\alpha_j|^2)^6} \right) 
               \right)  | \Fc_j \right)\\
   & \quad - \E^\Q \left( \frac{1-|\alpha_j|^2}{|1-\alpha_j Q_j|^2} | \Fc_j \right)^2\\   
 = & \E \left( 1-2|\alpha_j|^2 -|\alpha_j Q_{j}|^2 + 9|\alpha_j Q_j|^2  | \Fc_j \right) \\
   & \quad 
     + \Oc\left( \E \frac{|\alpha_j|^4}{(1-|\alpha_j|^2)^6} \right)
     - \E^\Q \left( \frac{1-|\alpha_j|^2}{|1-\alpha_j Q_j|^2} | \Fc_j \right)^2\\
 = & 1 - \frac{4}{\beta(j+1)}\left( 1 - |Q_{j}|^2 \right)
     + \frac{12}{\beta(j+1)} |Q_{j}|^2\\
   & \quad
     + \Oc\left( \E \frac{|\alpha_j|^4}{(1-|\alpha_j|^2)^6} \right) 
     - \E^\Q \left( \frac{1-|\alpha_j|^2}{|1-\alpha_j Q_j|^2} | \Fc_j \right)^2\\
 = & \frac{4}{\beta(j+1)} |Q_{j}|^2 + \Oc \left( \frac{1}{(j+1)^2} \right) \ , 
\end{align*}
which gives the desired estimate for $j$ large enough depending on $\beta$. 

For the fourth moment, we first deduce from  \eqref{eq:recurrenceQj}: 
\begin{align*}
    |Q_j|^2 - |Q_{j+1}|^2 
= & (1-r^2) + (1 - |Q_j|^2) \left(  \frac{r^2 (1 - |\alpha_j|^2)}{ |1 - \alpha_j Q_j|^2} - 1 \right)\\
= & |Q_j|^2 (1-r^2)
  + r^2 (1 - |Q_j|^2) \left(  \frac{(1 - |\alpha_j|^2)}{ |1 - \alpha_j Q_j|^2} - 1 \right) \ .
\end{align*}
Since $|Q_j|^2 \in [0,1]$, the fourth moment of the first term is smaller than $(1-r^2)^4$ and for the second term, it is enough to get the estimate
\begin{equation}
\E^{\Q} \left[ \left(\frac{ 1 - |\alpha_j|^2}{|1 - \alpha_j Q_j|^2}  - 1 \right)^4 \, \big| \Fc_j \right] 
= \Oc \left( \frac{1}{(j+1)^2} \right). \label{4L}
\end{equation}

We have, for $p \in \{0,1,2,3,4\}$, thanks to \eqref{eq:circleMomentBound},
\begin{align*}
   & \E^\Q \left( \frac{(1-|\alpha_j|^2)^p}{|1-\alpha_j Q_j|^{2p}} | \Fc_j \right)\\
 = & \E \left( \frac{(1-|\alpha_j|^2)^p(1-|\alpha_j Q_{j}|^2)}{|1-\alpha_j Q_j|^{2p+2}} | \Fc_j \right)\\
 = & \E \left( (1-|\alpha_j|^2)^p (1-|\alpha_j Q_{j}|^2)
               \left( 1 + (p+1)^2 |\alpha_j Q_j|^2 + \Oc\left( \frac{|\alpha_j Q_j|^4}{(1-|\alpha_j|^2)^{2p+2}} \right) 
               \right)  | \Fc_j \right)\\
 = & \E \left( 1- p |\alpha_j|^2 -|\alpha_j Q_{j}|^2 + (p+1)^2 |\alpha_j Q_j|^2  | \Fc_j \right)
     + \Oc\left( \E \frac{|\alpha_j|^4}{(1-|\alpha_j|^2)^{2p+2}} \right)\\
 = & 1 - \frac{2p}{\beta(j+1)}
     + \frac{2((p+1)^2 - 1) }{\beta(j+1)} |Q_{j}|^2
     + \Oc\left( \frac{1}{(j+1)^2} \right) \ .
\end{align*}
Multiplying this estimate respectively by $1, -4, 6, -4, 1$ for $p$ equal to $0,1,2,3,4$, and adding the terms, we get the desired bound for $j$ large enough depending on $\beta$. 

For the last equations that concern the logarithm, we get 
\begin{align*}
    \E^{\Q}\left( -\log (1 - \alpha_j Q_j) \ | \Fc_{j} \right)
= & -\E\left( \frac{1-|\alpha_j Q_j |^2}{|1-\alpha_j Q_j|^2} 
                    \log (1 - \alpha_j Q_j ) | \Fc_{j} \right)
\end{align*}
If we condition on $\Fc_j$ and $|\alpha_j|$, the expectation is, for $u = |\alpha_j Q_j|$, 
\begin{align*}
& \frac{1-u^2}{2 \pi} \int_0^{ 2 \pi} (1 - u e^{i \theta})^{-1} (1 - u e^{- i \theta})^{-1} \log ( 1- u e^{i \theta}) d \theta
\\ & = - \frac{1-u^2}{2 \pi}   \int_0^{ 2 \pi}  \sum_{k, \ell  \geq 0, \, m \geq 1} \frac{u^{k+ \ell + m}}{m}  e^{ i \theta( k - \ell + m)} d \theta
= - (1-u^2) \sum_{k \geq 0, \, m \geq 1} \frac{u^{2(k+m)}}{m}
\\ & = - (1 - u^2) \sum_{k \geq 0}  u^{2k} \sum_{m \geq 1} \frac{u^{2m}}{m} = \log (1 - u^2).
\end{align*}
We deduce 
\begin{align*}
    \E^{\Q}\left( -\log (1 - \alpha_j Q_j) \ | \Fc_{j} \right) 
= & \E\left(   - \log (1 - |\alpha_j Q_j |^2) | \Fc_{j} \right)\\
= & \E\left(  |\alpha_j Q_j |^2 | \Fc_{j} \right)
    + \Oc \left( \E \frac{ |\alpha_j|^4}{1- |\alpha_j|^2} \right) \\
= & \frac{2}{\beta(j+1)} |Q_j|^2 + \Oc \left( \frac{1}{(j+1)^2} \right)
\end{align*}
for $j$ large enough depending on $\beta$. 
 
 Let us now estimate the variance. We get 
 $$E^{\Q}\left( (- 2 \Re \log (1 - \alpha_j Q_j) )^2 \ | \Fc_{j} \right) 
 = \E\left( \frac{1-|\alpha_j Q_j |^2}{|1-\alpha_j Q_j|^2} 
                    (- 2 \Re \log (1 - \alpha_j Q_j) )^2| \Fc_{j} \right).$$
                   If we condition on $|\alpha_j|$, we obtain 
  \begin{align*}
& \frac{1-u^2}{2 \pi} \int_0^{ 2 \pi} (1 - u e^{i \theta})^{-1} (1 - u e^{- i \theta})^{-1}  [-\log ( 1- u e^{i \theta}) - \log (1 - u e^{-i \theta})]^2 d \theta
\\ & =  \frac{1-u^2}{2 \pi}   \int_0^{ 2 \pi}  \sum_{k \geq 0} u^k e^{ i \theta k }  \sum_{\ell \geq 0} u^{\ell} e^{- i \theta \ell } 
\left(\sum_{m \in \mathbb{Z} \backslash \{0\}} \frac{u^{|m|}}{|m|} e^{i \theta m} \right)^2
d \theta
\\ & = (1-u^2) \sum_{k, \ell \geq 0, \, m, p  \in \mathbb{Z} \backslash \{0\}}
\frac{u^{k + \ell + |m| + |p|}}{|mp|} \mathds{1}_{k - \ell + m + p = 0}
\end{align*}
 If we fix $m$ and $p$, we prescribe the difference $k - \ell = -m -p$. 
 The possible values for $k+ \ell$ are then $|m+p| + 2r$ for $r  = \min(k, \ell) \geq 0$. 
 Hence, we get 
$$
  (1-u^2) \sum_{r \geq 0} u^{2r} \sum_{m, p  \in \mathbb{Z} \backslash \{0\}}\frac{u^{|m+p|+ |m| + |p|}}{|mp|}
  = \sum_{m, p  \in \mathbb{Z} \backslash \{0\}}\frac{u^{|m+p|+ |m| + |p|}}{|mp|}
  = 2 \sum_{ \substack{m \geq 1 \\ p  \in \Z \backslash \{0\}}} \frac{u^{|m+p|+ m + |p|}}{m|p|} \ .
$$
 If we split the sum in function of the sign of $p$ and then isolate the terms for $(m,p) = (1,1), (1,2), (2,1)$ in the second sum, we get after 
 minoring $2\max(m,p)$ by $m + p$: 
 $$ 2\sum_{m, p \geq 1} \frac{u^{2(m+p)}}{ mp} 
  + 2 \sum_{m, p \geq 1}  \frac{u^{2 \max(m,p)}}{ mp}
  =  \Oc \left(\sum_{m, p \geq 1} u^{2m} u^{2p} 
\right) + 2 u^2 + 2 u^4 + \Oc \left( \sum_{m,p \geq 2} u^{m} u^p \right),
   $$
   which gives 
   $$2 u^2 + \Oc \left( \frac{u^4}{(1-u)^2} \right) = 2 u^2 + \Oc \left( \frac{u^4}{(1-u^2)^2} \right).$$
 We deduce 
 $$E^{\Q}\left( (- 2 \Re \log (1 - \alpha_j Q_j) )^2 \ | \Fc_{j} \right) 
 = 2 \E\left(  |\alpha_j Q_j |^2) | \Fc_{j} \right)
+ \Oc \left( \E \frac{ |\alpha_j|^4}{(1- |\alpha_j|^2)^2} \right)
$$ $$= \frac{4}{\beta(j+1)} |Q_j|^2 + \Oc \left( \frac{1}{(j+1)^2} \right)$$
for $j$ large enough depending on $\beta$.
Subtracting the square of previous estimate of  the expectation gives the desired bound for  the variance. 

\end{proof}

\subsection{Bounds of useful quantities related to \texorpdfstring{$(Q_k)_{k \geq 0}$}{Qk}}

We recall the expressions, available for $N \geq n$: 
$$\omega_{r,n,N} :=  \frac{2}{\beta}\sum_{k=n}^{N-1} \frac{|Q_k(r)|^2 - r^{2k+2}}{k+1}$$
and 
$$\rho_{r,n,N} := \sum_{k=n}^{N-1} \left( - \log \left( \frac{1 - |\alpha_k Q_k(r)|^2}{1 - |\alpha_k|^2} \right)
+ \frac{2}{\beta} \frac{1 - |Q_k(r)|^2}{k+1} \right).$$

The analysis of these random variables is intimately related to the following result: 
\begin{proposition}
\label{proposition:boundBox}
For $r \in (0,1)$, the series:
$$ \sum_{k \geq 0} \frac{|Q_{k}(r)|^2}{k+1} $$
converges $\Q$-almost surely. 

Moreover, if for every $A>0$, we define the index $A_{r,n}  =  \max(n, \lfloor A(1-r^2)^{-1} \rfloor)$ and:
$$ \square_{r,n}^{(A, \infty)} := \sum_{k \geq A_{r,n}} \frac{|Q_{k+1}(r)|^2}{k+2},$$
then we have, for all $p \in \R$:
$$
     \E^\Q\left( e^{p \square_{r,n}^{(A, \infty)}} | \Fc_n \right)
\leq \exp\left( \Oc_p ( A^{-1} )\right) 
$$
almost surely.
\end{proposition}
\begin{proof}
In this proof, we again write $Q_k$ for $Q_k(r)$ in order to simplify the notation. 
We know that $(\log \Phi_k^*(r))_{k \geq 0}$ is a $(\F,\P)$-martingale, bounded in $L^2(\Omega, \mathcal{B}, \P)$ since it has bounded exponential moments, and then the expectation of its bracket is bounded. In particular, the bracket is $\P$-almost surely finite, and then $\Q$-almost surely finite since $\P$ and $\Q$ are equivalent measures. Now, since
$$ \log \Phi_k^*(r)
 = \sum_{j=0}^{k-1} \log\left( 1 - \alpha_j Q_j \right), $$
the bracket is 
$$ \langle \log \Phi_\cdot^*(r) \rangle_k
 = \sum_{j=0}^{k-1} \E\left( \left|\log\left( 1 - \alpha_j Q_j \right) \right|^2 \ | \ \Fc_j \right)
 = \sum_{j=0}^{k-1} \E \left( \sum_{m = 1}^{\infty} \frac{ |\alpha_j Q_j|^{2m}}{m^2} \ | \ \Fc_j \right)$$
 by rotational invariance of $\alpha_j$ and independence of $\alpha_j$ and $\Fc_j$, and then 
 $$\langle \log \Phi_\cdot^*(r) \rangle_k
 = \frac{2}{\beta} \sum_{j=0}^{k-1}   \left(\frac{ |Q_j|^2}{j+1} + \Oc (\E [ |\alpha_j|^4]) \right)
  =\Oc(1) +   \frac{2}{\beta} \sum_{j=0}^{k-1} \frac{ |Q_j |^2}{j+1} \ .$$
  Hence, the last sum is almost surely bounded when $k$ varies, for fixed $r<1$. 

\medskip

Now, let us bound $\square_{r,n}^{(A, \infty)}$. We start by proving the following equation: 
\begin{align}
\label{eq:selfConsistent_result}
  & (1-r^2)^{-1} \square_{r,n}^{(A, \infty)} \\
= & (1-r^2)^{-1} \sum_{k=A_{r,n}}^\infty \frac{|Q_{k+1}|^2 - \E^\Q\left( |Q_{k+1}|^2 | \Fc_k \right)}{k+2}
    + r^2 (1-r^2)^{-1} \square_{r,n}^{(A, \infty)} + \Oc \left(\frac{1}{A} \right).
\nonumber
\end{align}
In order to see that, we write:
\begin{align}
\label{equationsquare}
    (1-r^2)^{-1} \square_{r,n}^{(A, \infty)}
= & (1-r^2)^{-1} \sum_{k=A_{r,n}}^\infty \frac{|Q_{k+1}|^2 - \E^\Q\left( |Q_{k+1}|^2 | \Fc_k \right)}{k+2}\\
  &  + (1-r^2)^{-1} \sum_{k=A_{r,n}}^\infty \frac{\E^\Q\left( |Q_{k+1}|^2 | \Fc_k \right)}{k+2} \ ,
\nonumber
\end{align}
and from \eqref{eq:conditionalQj}, 
$$
  \frac{\E^\Q\left( |Q_{k+1}|^2 | \Fc_k \right)}{k+2}
= \frac{r^2 |Q_{k}|^2 }{k+1} + \Oc\left( \frac{1}{(k+1)^2} \right) \ .
$$
Then the combination of the two previous equations yields:
\begin{align*}
  & (1-r^2)^{-1} \square_{r,n}^{(A, \infty)} \\
= & (1-r^2)^{-1} \sum_{k=A_{r,n}}^\infty \frac{|Q_{k+1}|^2 - \E^\Q\left( |Q_{k+1}|^2 | \Fc_k \right)}{k+2}
    + r^2 (1-r^2)^{-1} \square_{r,n}^{(A, \infty)}
    \\  & + \Oc\left( \frac{r^2(1-r^2)^{-1}}{A_{r,n} + 1} + (1-r^2)^{-1} \sum_{k=A_{r,n}}^\infty \frac{1}{(k+1)^2} \right)\\
= & (1-r^2)^{-1} \sum_{k=A_{r,n}}^\infty \frac{|Q_{k+1}|^2 - \E^\Q\left( |Q_{k+1}|^2 | \Fc_k \right)}{k+2}
    + r^2 (1-r^2)^{-1} \square_{r,n}^{(A, \infty)}
    + \Oc\left( \frac{(1-r^2)^{-1} }{A_{r,n}+1} \right) \\
= & (1-r^2)^{-1} \sum_{k=A_{r,n}}^\infty \frac{|Q_{k+1}|^2 - \E^\Q\left( |Q_{k+1}|^2 | \Fc_k \right)}{k+2}
    + r^2 (1-r^2)^{-1} \square_{r,n}^{(A, \infty)}
    + \Oc\left( \frac{1}{A} \right) \ ,
\end{align*}
which is \eqref{eq:selfConsistent_result}.

\medskip

Rearranging this equation, we get 
\begin{align*}
    \square_{r,n}^{(A, \infty)}
= & \Oc\left( \frac1A \right) - (1-r^2)^{-1} \sum_{k=A_{r,n}}^\infty \frac{1-|Q_{k+1}|^2 - \E^\Q\left( 1-|Q_{k+1}|^2 | \Fc_k \right)}{k+2}\\
= & \Oc\left( \frac1A \right) - \sum_{k=A_{r,n}}^\infty \Delta M_k \ ,
\end{align*}
where $\Delta M_k$ are martingale differences, bounded by $1$. Estimating the increments of the bracket from \eqref{eq:conditionalVarQj} yields:
$$ \E^\Q\left( (\Delta M_k)^2 | \Fc_k \right) \ll \frac{(1-r^2)^{-2}}{(k+1)^3} \ .$$
Hence, the bracket of the corresponding martingale is bounded by 
$$ \langle M \rangle_\infty \ll \sum_{k \geq A_{r,n}} \frac{(1-r^2)^{-2}}{(k+1)^3} \ll_A 1 $$
and the martingale converges in $L^2(\Omega, \mathcal{B},  \Q)$. In order to bound exponential moments, we use the following variant of conditional Chernoff bounds. For all $p \in \R$:
\begin{align*}
         \E^\Q\left( e^{p \Delta M_{k}} | \Fc_k \right)
\leq \ & \E^\Q\left( 1 + p \Delta M_{k} + \half \left( p \Delta M_{k} \right)^2 e^{|p \Delta M_{k}|} | \Fc_k \right)\\
\leq \ & \E^\Q\left( 1 + \half \left( p \Delta M_{k} \right)^2 e^{|p|} | \Fc_k \right)\\
\leq \ & 1 + \Oc \left( \frac{p^2 (1-r^2)^{-2} e^{|p|}}{2(k+1)^3}  \right)\\
\leq \ & \exp\left( \Oc \left( \frac{p^2 (1-r^2)^{-2} e^{|p|}}{2(k+1)^3} \right) \right)
\end{align*}
In the end, by the tower property of conditional expectation, 
and the fact that $A_{r,n} \geq n$, we get 
$$
     \E^\Q\left( e^{p \sum_{k=A_{r,n}}^\infty \Delta M_k} | \Fc_n\right)
\leq \exp\left( \Oc \left( \half p^2 e^{|p|} \sum_{k \geq A_{r,n}} \frac{(1-r^2)^{-2} }{(k+1)^3} \right) \right)
\leq \exp\left( \Oc (p^2 e^{|p|} A^{-2} )\right). 
$$
Therefore,
$$
     \E^\Q\left( e^{p \square_{r,n}^{(A, \infty)}} | \Fc_n \right)
\leq \exp\left( \Oc ( p A^{-1} + p^2 e^{|p|} A^{-2} )\right)
   = \exp\left( \Oc_p ( A^{-1} )\right).
$$
\end{proof}

The main result of this subsection is the following: 
\begin{proposition}
\label{proposition:boundOmegaRho}
Almost surely, under $\Q$, $\omega_{r,n,N}$ and $\rho_{r,n,N}$ converge when $N \rightarrow \infty$, respectively to 
$$\omega_{r,n} =  \frac{2}{\beta}\sum_{k=n}^{\infty} \frac{|Q_k(r)|^2 - r^{2k+2}}{k+1}$$
 and 
 $$\rho_{r,n} := \sum_{k=n}^{\infty} \left( - \log \left( \frac{1 - |\alpha_k Q_k(r)|^2}{1 - |\alpha_k|^2} \right)
+ \frac{2}{\beta} \frac{1 - |Q_k(r)|^2}{k+1} \right).$$
Moreover, 
for all $p \geq 0$,  there exists $K > 0$, depending only on $p$ and $\beta$, such that for all $r \in (0,1)$, 
$$\E^{\Q} [ e^{p \, \omega_{r,n}} | \Fc_n ] \leq K, \; 
\E^{\Q} [ e^{p \rho_{r,n}} | \Fc_n ] \leq K$$
a.s., and if $\varepsilon > 0$, and $L$ is a compact subset of the $n$-th power of the open unit disc, then there exists $\eta$, depending on $ \beta, \varepsilon, n, L$ and $r$, 
and tending to zero when $r \rightarrow 1-$ and the other parameters are fixed, such that 
$$\Q [ |\rho_{r,n}| > \varepsilon | \Fc_n]  \leq \eta$$
a.s. on the event where $(\alpha_0, \dots, \alpha_{n-1}) \in L$. 
In particular, for $n = 0$, we get that $\omega_{r,0}$ and $\rho_{r,0}$ have bounded positive exponential moments
and $\rho_{r,0}$ converges to $0$ in probability when $r$ tends to $1$ from below. 
\end{proposition}
\begin{proof}
The convergence of the series defining $\omega_{r,n}$ is a direct consequence of the convergence of the series 
$$ \square_{r,n} := \sum_{ k \geq n} \frac{|Q_{k+1}|^2}{k+2} \ ,$$
proven in Proposition \ref{proposition:boundBox}. 
  Let us now bound the positive exponential moments of $\omega_{r,n}$. Using the index $A_{r,n}  =  \max(n, \lfloor (1-r^2)^{-1} \rfloor)$ for $A=1$, 
  we get
  \begin{align*}
       \omega_{r,n}
   = \ & \frac{2}{\beta} \sum_{n \leq k \leq A_{r,n}} \frac{|Q_k|^2-r^{2k+2}}{k+1} + \frac{2}{\beta} \sum_{k=A_{r,n}+1}^\infty \frac{|Q_k|^2-r^{2k+2}}{k+1}\\
\leq \ & \frac{2}{\beta} \sum_{n \leq k \leq A_{r,n}} \frac{1-r^{2k+2}}{k+1} + \frac{2}{\beta}  \sum_{k=A_{r,n}+1}^{\infty} \frac{|Q_k|^2}{k+1} \\
\leq \ & \frac{2}{\beta} \sum_{k=n}^{A_{r,n}} \frac{1-r^{2k+2}}{k+1} + \frac{2}{\beta} \square_{r,n}^{(A=1, \infty)} \\
   = \ & \Oc(1) + \frac{2}{\beta} \square_{r,n}^{(A=1, \infty)} \ .
\end{align*}
In the last step, we recognized the convergent (hence bounded) Riemann sum:
$$ \sum_{k=n}^{A_{r,n}} \frac{1-r^{2k+2}}{k+1} \stackrel{r \rightarrow 1}{\longrightarrow} \int_0^1 dt \ \frac{1-e^{-t}}{t} \ .$$
As a consequence, it suffices to prove the bound for $\square_{r,n}^{(A=1, \infty)}$ instead of $\omega_{r,n}$, which is already contained in Proposition \ref{proposition:boundBox}. 

Let us now consider $\rho_{r,n}$. We have
\begin{align*}
 \rho_{r,n,N}
 = & -\sum_{j=n}^{N-1} \left( -\log\left(1-|\alpha_j|^2\right) - \frac{2}{\beta(j+1)} \right)\left( 1 - |Q_j|^2 \right) \\
   & \ \ - \sum_{j=n}^{N-1} \left( \log\left(1-|\alpha_j Q_j|^2\right) - |Q_j|^2 \log\left(1-|\alpha_j|^2\right) \right)\\
 =: & \ \ E^{(N)}_1+ E^{(N)}_2\ .
\end{align*}
We control each of the terms separately. We easily check, using  \eqref{eq:log_beta},  that  $(E^{(N)}_1)_{N \geq n}$  is a $\P$-martingale.
 It is also a $\Q$-martingale, since 
$(|\alpha_j|)_{j \geq 0}$ has the same law under the two probability measures. Indeed, for any measurable function $F$ from $\R$ to $\R_+$,  and any $j \geq 0$, 
we have 
\begin{align*} \E^{\Q} [ F(|\alpha_j|) | \Fc_j ] 
& = \E^{\P} \left[ F(|\alpha_j|)  \frac{1 - |\alpha_j Q_j|^2}{|1 - \alpha_j Q_j|^2} | \Fc_j \right] \\ 
& =  \E^{\P} \left[ F(|\alpha_j|)  \E^{\P} \left[ \frac{1 - |\alpha_j Q_j|^2}{|1 - \alpha_j Q_j|^2}  | \Fc_j, |\alpha_j| \right]| \Fc_j \right]  \\
& =  \E^{\P} \left[ F(|\alpha_j|)  | \Fc_j \right],
\end{align*} 
the last equality coming from the fact that for $u = |\alpha_j Q_j| < 1$, 
$$\frac{1}{2\pi} \int_0^{2 \pi} \frac{1  - u^2}{ |1 - u e^{i \theta}|^2} d \theta = 1.$$ 
The conditional $L^2$ norm of the martingale is bounded as follows, using \eqref{eq:log_beta} and \eqref{boundlogcarre}:
$$ \E^{\Q} \left( (E^{(N)}_1)^2  | \Fc_n \right)
 = \sum_{j=n}^{N-1} \E^{\Q} \left( (1-|Q_j|^2)^2 | \Fc_n \right) \Var\left( \log(1-|\alpha_j|^2) \right)\\
 $$ $$\ll \sum_{j=n}^{N-1} \E^{\Q} \left( (1-|Q_j|^2)^2 | \Fc_n \right) \frac{1}{(j+1)^2} \ ,
$$
which a.s. converges when $N$ goes to infinity. We deduce that the martingale $(E_1^{(N)})_{N \geq n}$ is a.s convergent, which means that 
$$-\sum_{j=n}^{\infty} \left( -\log\left(1-|\alpha_j|^2\right) - \frac{2}{\beta(j+1)} \right)\left( 1 - |Q_j|^2 \right)$$
 tends a.s. to a limit $E_1$. 
Furthermore, by iterating Sz\"ego recursion, we deduce that for $j \geq n$, 
$$Q_j = \Xi_{n,j} ( r,  Q_n, (\alpha_k)_{n \leq k \leq j-1},  (\overline{\alpha_k})_{n \leq k \leq j-1})$$
where $\Xi_{n,j}$ is a universal rational function, which has modulus $1$ if the two first arguments have modulus $1$. 

Moreover,  $\Xi_{n,j}$ is  well-defined when the two first arguments have modulus at most $1$ and the others have modulus strictly less than $1$, 
and then it is continuous when these constraints are satisfied, and uniformly continuous if we restrict the $\alpha_k$'s to a compact set of the open unit disc. 
We deduce that for fixed $(b_k)_{n \leq k \leq q-1}$ of modulus strictly less than $1$, 
$$S_{n,j} (r, q, (b_k)_{n \leq k \leq j-1}) := \inf_{|Q| \in [ q,1], (|a_k| = b_k)_{n \leq k \leq j-1} } | \Xi_{n,j} ( r, Q,  (a_k)_{n \leq k \leq j-1},  (\overline{a_k})_{n \leq k \leq j-1})|$$
goes to $1$ when $r$ and $q$ go to $1$ from below. 
Now, we have
$$\E^{\Q}\left( (E^{(N)}_1)^2  | \Fc_n \right) 
\ll  \sum_{j=n}^{N-1} \frac{1}{(j+1)^2} \E^{\Q} \left( (1-S_{n,j} (r, q, (|\alpha_k|)_{n \leq k \leq j-1})^2)^2  | \Fc_n \right)  $$
on the event when $|Q_n| \in [q,1]$. 
Conditionally on $\Fc_n$, $(|\alpha_k|)_{n \leq k \leq j-1}$ has the same distribution under $\P$ and under $\Q$, as proven above.
Since  $(|\alpha_k|)_{n \leq k \leq j-1}$ is independent of  $\Fc_n$ under $\P$, it is also independent under $\Q$, with the same distribution, and then 
$$\E^{\Q}\left( (E^{(N)}_1)^2  | \Fc_n \right) 
\ll  \sum_{j=n}^{N-1} \frac{1}{(j+1)^2} \E^{\P} \left( (1-S_{n,j} (r, q, (|\alpha_k|)_{n \leq k \leq j-1})^2)^2  \right)  $$
on the event when $|Q_n| \in [q,1]$. 
Now, $|Q_n|$ is a continuous function of $r$, $\alpha_0, \dots, \alpha_{n-1}$, and then it is uniformly continuous 
if we assume that $(\alpha_0, \dots, \alpha_{n-1})$ is in a given compact set $L \in \D^n$. 
Hence, under this assumption, $|Q_n| \in [g_{n,L}(r),1]$ when $g_{n,L}$ is a function tending to $1$ when $r \rightarrow 1-$. 
We deduce 
$$\E^{\Q}\left( (E^{(N)}_1)^2  | \Fc_n \right) 
\ll  \sum_{j=n}^{\infty} \frac{1}{(j+1)^2} \E^{\P} \left( (1-S_{n,j} (r, g_{n,L}(r), (|\alpha_k|)_{n \leq k \leq j-1})^2)^2  \right), $$
when $(\alpha_0, \dots, \alpha_{n-1}) \in L$. 
By dominated convergence, the right-hand side, which depends on $\beta, n, L$ and $r$, converges to $0$ when $r \rightarrow 1^-$ with fixed other parameters.
Now, if we let $N \rightarrow \infty$, we deduce, by Fatou's lemma, 
$$\E^{\Q}\left( E_1^2  | \Fc_n \right) 
\ll  \sum_{j=n}^{\infty} \frac{1}{(j+1)^2} \E^{\P} \left( (1-S_{n,j} (r, g_{n,L}(r), (|\alpha_k|)_{n \leq k \leq j-1})^2)^2  \right), $$
and then
\begin{equation}
\Q [ |E_{1}|  > \varepsilon  | \Fc_n ] \leq \eta_1 \label{convproba1},
\end{equation}
where 
$$\eta_1 \ll \varepsilon^{-2} \sum_{j=n}^{\infty} \frac{1}{(j+1)^2} \E^{\P} \left( (1-S_{n,j} (r, g_{n,L}(r), (|\alpha_k|)_{n \leq k \leq j-1})^2)^2  \right)$$
depends on  $\beta, \varepsilon, n, L$ and $r$ and tends to $0$ when $r$ goes to $1$. 

In order to estimate $E^{(N)}_2$, we observe that for $a, q \in [0,1)$,  
$$ - \log ( 1 - a q) + q \log (1-a) 
= \sum_{k= 1}^{\infty} \frac{ a^k (q^k - q)}{k},$$
and then this quantity is nonpositive. Moreover,
since 
$$q - q^k = \sum_{r = 1}^{k-1} (q^r - q^{r+1}) \leq (k-1)(1-q),$$ 
 the absolute value of the quantity above is 
at most 
$$\sum_{k=1}^{\infty} \frac{ a^k (k-1)(1-q)}{k}
\leq (1-q) \sum_{k=2}^{\infty} a^k 
= \frac{(1-q) a^2}{1-a}. $$

We deduce that $E^{(N)}_2$ is a sum of nonpositive terms: let $E_2$ be its limit (real or equal to $-\infty$) when $N \rightarrow \infty$. 
We have
$$|E_2| \leq \sum_{j = n}^{\infty} (1 - |Q_j|^2) \frac{|\alpha_j|^4}{ 1- |\alpha_j|^2}.$$
Since for $j$ such that $\beta_j -1 > 0$, 
\begin{align*} \mathbb{E}^{\Q}  \left[ \frac{|\alpha_j|^4} {1 - |\alpha_j|^2} \right] 
& = \mathbb{E} \left[ \frac{|\alpha_j|^4} {1 - |\alpha_j|^2} \right] 
\\ & = \beta_j \int_{0}^1 x^2 (1 -x)^{\beta_j - 2} dx
= \beta_j \, \frac{ \Gamma(3) \Gamma(\beta_j-1)}{ \Gamma(\beta_j + 2)} 
\\ & = \frac{2 }{(\beta_j - 1)(\beta_j + 1)} = \mathcal{O} (1/j^2), 
\end{align*}
the last series is a.s. convergent and $|E_2|$ is a.s. finite. 
Moreover, if $(\alpha_0, \dots, \alpha_{n-1}) \in L$, we get
$$|E_2| \leq \sum_{j = n}^{\infty} (1 - S_{n,j}(r, g_{n,L}(r), (|\alpha_k|)_{n \leq k \leq j-1})^2) \frac{|\alpha_j|^4}{ 1- |\alpha_j|^2}.$$
Hence,  for $\varepsilon > 0$, 
\begin{align*}
\Q [ |E_2| >  \varepsilon | \Fc_n] 
& \leq  \Q \left[ \max_{n \leq j \leq p-1}  \frac{|\alpha_j|^4}{ 1- |\alpha_j|^2} > R  | \Fc_n \right] 
\\ & + \varepsilon^{-1} \E^{\Q} \left[ R \sum_{j = n}^{p-1} (1 - S_{n,j}(r, g_{n,L}(r), (|\alpha_k|)_{n \leq k \leq j-1})^2) \right. 
\\ &  \left. + \sum_{j = p}^{\infty} (1 - S_{n,j}(r, g_{n,L}(r), (|\alpha_k|)_{n \leq k \leq j-1})^2) \frac{|\alpha_j|^4}{ 1- |\alpha_j|^2}   | \Fc_n  \right]
\end{align*}
for any $p \geq n$ and $R > 0$. 
The first term is a quantity depending on $\beta, n, R, p$, and tending to zero when $R$ goes to infinity. 
The second term depends on $\beta, \varepsilon, n, R, p, L, r$, is finite as soon as $p$ is large enough depending on $\beta$, 
and under such assumption, it tends to zero when $r \rightarrow 1-$ by dominated convergence. 
We deduce that for fixed $\beta, \varepsilon, n, L$, which also allows to fix $p$, we get 
$$\Q [ |E_2| > \varepsilon | \Fc_n] \leq \delta_1(R) + \delta_2(R,r)$$
where $\delta_1(R)$ goes to $0$ when $R \rightarrow \infty$ and $\delta_2(R,r)$ goes to $0$ when $R$ is fixed 
and $r \rightarrow 1-$. Since the left-hand side is in fact independent of $R$, we have 
\begin{equation}
\Q [ |E_{2}|  > \varepsilon  | \Fc_n ] \leq \eta_2 \label{convproba2},
\end{equation}
where 
$$\eta_2 := \inf_{R > 0} (\delta_1(R) + \delta_2(R,r))$$
may depend on $\beta, \varepsilon, n, L, r$. 
For each $R > 0$, we get 
$$\underset{ r \rightarrow 1-}{\lim \sup}  \; \eta_2 \leq \delta_1(R) + \lim_{r \rightarrow 1-}  \delta_2(R,r)
= \delta_1(R),$$
and then $\eta_2$ goes to zero when $r \rightarrow 1-$. 

From the convergence of $E_1^{(N)}$ to $E_1$ and the convergence of $E_2^{(N)}$ to $E_2$, we deduce the a.s. convergence of $\rho_{r,n,N}$ towards 
 $$\rho_{r,n} = \sum_{k=n}^{\infty} \left( - \log \left( \frac{1 - |\alpha_k Q_k(r)|^2}{1 - |\alpha_k|^2} \right)
+ \frac{2}{\beta} \frac{1 - |Q_k(r)|^2}{k+1} \right),$$
the infinite series being a.s. convergent. Moreover, from the estimates \eqref{convproba1} and \eqref{convproba2}, we deduce that for $\varepsilon > 0$, and $L$ a compact subset of the $n$-th power of the open unit disc, there exists $\eta$, depending on $ \beta, \varepsilon, n, L$ and $r$, 
tending to zero when $r \rightarrow 1-$ and the other parameters are fixed, such that 
$$\Q [ |\rho_{r,n}| > \varepsilon | \Fc_n]  \leq \eta$$
a.s. on the event $(\alpha_0, \dots, \alpha_{n-1}) \in L$.  
In the particular case $n = 0$, the $\sigma$-algebra $\Fc_n$ is trivial, and no restriction to an event of the form   $(\alpha_0, \dots, \alpha_{n-1}) \in L$
is involved. We get the existence of  $\eta$, depending only on $\beta, \varepsilon$ and $r$, tending to zero when $r \rightarrow 1$, 
such that 
$$\Q [ |\rho_{r,0}| > \varepsilon]  \leq \eta.$$
Hence, for any $\varepsilon > 0$, 
$$\Q [ |\rho_{r,0}| > \varepsilon]  \underset{r \rightarrow 1}{\longrightarrow} 0,$$
which means that $\rho_{r,0}$ tends to zero in probability when $r \rightarrow 1$.

It remains to bound the positive exponential moments of $\rho_{r,n}$. 
We notice that since $E_2$ has nonpositive terms, it is enough to bound the positive exponential moments of $E_1$. 

We have, for all $p \geq 0$,
\begin{align*}
  & \E^\Q\left[ \exp\left( p \left( \log\left(1-|\alpha_j|^2\right) + \frac{2}{\beta(j+1)} \right)
                             \left( 1 - |Q_j|^2 \right)
                    \right) | \Fc_j \right]\\
= & \E^\Q\left[ \left(1-|\alpha_j|^2\right)^{p\left( 1 - |Q_j|^2 \right)}
                e^{\frac{2p}{\beta(j+1)} \left( 1 - |Q_j|^2 \right)}
    |  \Fc_j \right]\\
= & \frac{(\beta/2)(j+1)}{(\beta/2)(j+1)+ p\left( 1 - |Q_j|^2 \right)} e^{\frac{2p}{\beta(j+1)} \left( 1 - |Q_j|^2 \right)} \\
= & \frac{1}{1 + \frac{2p}{\beta(j+1)}\left( 1 - |Q_j|^2 \right)} e^{\frac{2p}{\beta(j+1)} \left( 1 - |Q_j|^2 \right)} \ .
\end{align*}
Now, using the inequality $-\log(1+x) \leq -x + \half x^2$, available for all $x \geq 0$, we deduce:
\begin{align*}
 & \E^\Q\left[ \exp\left( p \left( \log\left(1-|\alpha_j|^2\right) + \frac{2}{\beta(j+1)} \right)
               \left( 1 - |Q_j|^2 \right)
                \right) | \Fc_j \right]\\
\leq & \ e^{\frac{2p^2}{\beta^2 (j+1)^2} \left( 1 - |Q_j|^2 \right)^2}\\
\leq & \ e^{\frac{2p^2}{\beta^2 (j+1)^2}} \ .
\end{align*}

Using the  tower property of conditional expectation, we deduce 
$$\E^{\Q} [ e^{p E_1} | \Fc_n ] 
\leq e^{\Oc(p^2)},$$
which finishes the proof of the proposition. 
\end{proof}

\section{A diffusive limit for \texorpdfstring{$|Q_j(r e^{i \theta})|^2$}{modulii of Q}}
\label{section:convergenceSDE}

In the proof of Lemma \ref{lemma:OmegaRhoControl}, we will require various estimates regarding 
$\left(|Q_j(r e^{i \theta})|^2 \right)_{j \geq n}$
as well as the diffusive limit for fixed $\theta$, while $j$ is large and $r$ close to $1$. These $\F$-adapted processes satisfy some discrete approximations of SDEs. We will need to know if their distribution converges to some solutions of the corresponding continuous SDEs, in the same way as the simple random walk converges to the Brownian motion after suitable scaling.

The precise statement is as follows, and its proof is the topic of the current section. 

\begin{proposition}
\label{prop:convergencetosde}
We assume $\beta > 2$. For $\varepsilon \in (0,1)$, $A > 0$, $r \in (0,1)$, recall that the indices $A_{r,n}$ and $\varepsilon_{r,n}$ are:
$$ A_{r,n} = \max( n, \lfloor A/(1-r^2) \rfloor ) \ ,
   \quad
   \varepsilon_{r,n} = \max( n, \lfloor \varepsilon/(1-r^2) \rfloor ) \ .$$
Moreover, we fix an event $\mathcal{G} \in \mathcal{F}_n$, which may depend on $r$, and under which $(\alpha_0, \dots, \alpha_{n-1})$ is in some deterministic compact subset of $\D^n$, independent of $r$. 
Then, for every $\xi>0$, we have 
\begin{align}
\label{eq:cv2sde_estimate1}
\limsup_{A \rightarrow \infty} \sup_{r \in (0,1)} 
\Q \left[ \sum_{j = A_{r,n}}^{\infty}  \frac{|Q_j(r)|^2}{j+1} \geq \xi \; |  \; \mathcal{G} \right] = & 0
\end{align}
and
\begin{align}
\label{eq:cv2sde_estimate2}
\limsup_{\varepsilon \rightarrow 0} \sup_{r \in (0,1)} 
\Q \left[ \sum_{j = n}^{\varepsilon_{r,n}-1}  \frac{1 - |Q_j(r)|^2}{j+1} \geq \xi \; | \; \mathcal{G} \right] = & 0 \ .
\end{align}

Moreover, consider the process $X_t^{(r)}$ equal to $|Q_{n + \frac{t}{\log ( r^{-2})}}(r)|^2$ at time $t$ when $t$ is a multiple of $\log (r^{-2})$, and linearly interpolated for other values of $t$. Under $\Q$ and conditionally on $\Gc$, the law of $X^{(r)}$ tends, for the topology of uniform convergence on compact sets, to the distribution of a continuous solution $X$ of the following stochastic differential equation 
\begin{equation}
\label{mainsde}
d(1-X_t) =   X_t dt
              - \left( 1 - X_t \right)^2 \frac{2 dt}{\beta t}
              + 4\left( 1 - X_t \right) X_t \frac{dt}{\beta t}
              + \sqrt{ \left( 1 - X_t \right)^2 \frac{4 X_t}{\beta t} } dB_t \ , 
\end{equation}
where $B$ is a Brownian motion. 

Furthermore, we have that almost surely, 
\begin{equation}
\sup_{t \in (0,1]} \frac{1 - X_t}{t^{1-\varepsilon'}} < \infty \label{supbidule} 
\end{equation} 
for all $\varepsilon' \in (0,1)$, and
$$\int_{0}^{\infty} \frac{X_t - e^{-t}}{t} dt < \infty \ . $$

Finally, the law of $X$ is uniquely determined by the properties given above.
\end{proposition}
\begin{proof}[Strategy of proof]
In Subsection \ref{section:cv2sde_1}, we provide the proofs of the estimates \eqref{eq:cv2sde_estimate1} and \eqref{eq:cv2sde_estimate2}.

In Subsection \ref{section:cv2sde_2}, we prove that the sequence of laws $\Lc\left( X_t^{(r)} ; t  \geq 0\right)$ is tight for $r<1$. Furthermore, every limit point as $r \rightarrow 1$ is solution of the announced SDE \eqref{mainsde} - for $t>0$. At that stage, uniqueness will still be required to finish the proof of the convergence in law.

In Subsection \ref{section:cv2sde_3}, we study the entrance law for any solution to the SDE \eqref{mainsde} satisfying \eqref{supbidule} for $t>0$ and we relate it to Dufresne's identity.

Subsection \ref{section:cv2sde_4} finally proves uniqueness, and thus concludes the proof of Proposition \ref{prop:convergencetosde}.
\end{proof}

\subsection{Proofs of Eq. \texorpdfstring{\eqref{eq:cv2sde_estimate1}}{} and \texorpdfstring{\eqref{eq:cv2sde_estimate2}}{} }
\label{section:cv2sde_1}
For the first statement, with the notation of Proposition \ref{proposition:boundBox}, it is enough to prove that
\begin{equation} 
\E^{\Q} \left( \square_{r,n}^{(A, \infty)} | \Fc_n \right) = \Oc( A^{-1} ) \label{squareAinfty}
\end{equation}
almost surely, the implicit constant being independent from $r$. From the proof of that proposition, we already have:
\begin{align*}
    \square_{r,n}^{(A, \infty)}
= & \Oc(1/A) - \sum_{k=A_{r,n}}^\infty \Delta M_k \ ,
\end{align*}
where $\Delta M_k$ are martingale differences, bounded by $1$. Estimating the increments of the bracket from \eqref{eq:conditionalVarQj} yields:
$$ \E^\Q\left( (\Delta M_k)^2 | \Fc_k \right) \ll \frac{(1-r^2)^{-2}}{(k+1)^3} \ .$$
Hence, the bracket of the corresponding martingale is bounded by 
$$ \langle M \rangle_\infty \ll \sum_{k \geq A_{r,n}} \frac{(1-r^2)^{-2}}{(k+1)^3} \ll 1/A^2 $$
and the martingale converges in $L^2(\Omega, \mathcal{B}, \Q)$, with 
$$\E [ (M_{\infty} - M_{A_{r,n}} )^2 | \mathcal{F}_n] \ll 1/A^2,$$
which gives the desired estimate.

\medskip

For the second statement \eqref{eq:cv2sde_estimate2}, we will need the following notion. We say that a family of random variables $(X(r))_{r \in (0,1)}$ is tight conditionally to $\mathcal{G}$ when almost surely: 
\begin{align}
\label{eq:conditionalTightness}
\limsup_{a \rightarrow \infty} \sup_{r \in (0,1)} \Q\left( |X(r)|>a \ | \ \mathcal{G} \right) = 0 \ .
\end{align}
Thanks to the recurrence  \eqref{eq:recurrenceQj} and the fact that $Q_0(r) =r$,  we have
\begin{align*}
  & 1 - |Q_{n+k}|^2 \\
= & \sum_{j=0}^{k-1} (1-r^2) r^{2j} \prod_{\ell=k+n-j}^{k+n-1} \frac{1-|\alpha_{\ell}|^2}{|1-\alpha_{\ell} Q_{\ell}|^2}
    + ( 1 - |Q_n|^2) r^{2k} \prod_{\ell=n}^{k+n-1} \frac{1-|\alpha_{\ell}|^2}{|1-\alpha_{\ell} Q_{\ell}|^2} 
\\
\leq & (1 - |Q_n|^2) \prod_{\ell=n}^{k+n-1} \frac{1-|\alpha_{\ell}|^2}{|1-\alpha_{\ell} Q_{\ell}|^2}
     + (1-r^2) \sum_{j=n+1}^{n+k} \prod_{\ell=j}^{n+k-1} \frac{1-|\alpha_{\ell}|^2}{|1-\alpha_{\ell} Q_{\ell}|^2} \ .
\end{align*}
The last inequality is obtained by exchanging the two terms of the sum, by using the fact that $r \leq 1$ and by  changing the index $j$ to $n+k-j$. 
By introducing the two following $(\F, \Q)$-martingales:
$$ \left\{
   \begin{array}{cc}
   \Nc_j := & \sum_{\ell=0}^{j-1} \left( -\log (1 - |\alpha_\ell|^2) - \frac{2}{\beta(\ell+1)} \right)\\
   \Mc_j := & \Mc_j(r) = \sum_{\ell=0}^{j-1} \left( -\log |1-\alpha_{\ell} Q_{\ell}|^2
          + \E^{\Q} \left[ \log |1-\alpha_{\ell} Q_{\ell}|^2 | \Fc_{\ell} \right] \right) \
   \end{array} \right. \ ,
$$
we can rewrite the previous expression as:
\begin{align*}
     & 1 - |Q_{n+k}|^2 \\
\leq & (1 - |Q_n|^2) \exp\left( \sum_{\ell=n}^{k+n-1} \log \left(\frac{1 - |\alpha_\ell|^2}{|1 - \alpha_\ell Q_\ell|^2} \right) \right) \\
     & \quad \quad + (1-r^2) \sum_{j=n+1}^{n+k} \exp\left( \sum_{\ell=j}^{n+k-1} \log \left(\frac{1 - |\alpha_\ell|^2}{|1 - \alpha_\ell Q_\ell|^2} \right) \right) \\
\leq & (1 - |Q_n|^2) \exp\left( -(\Nc_{k+n} - \Nc_{n}) + \Mc_{k+n} - \Mc_n 
       + \sum_{\ell=n}^{k+n-1} \E^{\Q} \left[ \log \left(\frac{1 - |\alpha_\ell|^2}{|1 - \alpha_\ell Q_\ell|^2} \right) | \Fc_\ell \right] \right)\\
     & + (1-r^2) \sum_{j=n+1}^{n+k}
       \exp\left( -(\Nc_{k+n} - \Nc_{j}) + \Mc_{k+n} - \Mc_j 
       + \sum_{\ell=j}^{k+n-1} \E^{\Q} \left[ \log \left(\frac{1 - |\alpha_\ell|^2}{|1 - \alpha_\ell Q_\ell|^2} \right) | \Fc_\ell \right] \right) \ .
\end{align*}
In this computation, we have used the fact that the conditional distribution of $|\alpha_j|^2$ given $\Fc_j$ is the same under $\P$ and under $\Q$, which implies  
$$ \E^{\Q} [ -\log (1 - |\alpha_j|^2) | \Fc_j]
 = \E^{\P} [ -\log (1 - |\alpha_j|^2) | \Fc_j] 
 = \frac{2}{\beta(j+1)} \ .$$
Combining these estimates with the last estimates of Proposition \ref{properties:underQ}, we have a.s., for $n$ large enough depending on $\beta$:
\begin{align*}
  \E^{\Q} \left[ \log \left(\frac{1 - |\alpha_j|^2}{|1 - \alpha_j Q_j|^2} \right) | \Fc_j \right]
= & \frac{4}{\beta(j+1)}|Q_j|^2  -\frac{2}{\beta(j+1)} +  \Oc \left(\frac{1}{(j+1)^2} \right) \\
\leq & \frac{2}{\beta(j+1)} + \Oc \left(\frac{1}{(j+1)^2} \right) ,
\end{align*}
which becomes upon summing:
$$ \sum_{\ell=j}^{n+k-1} 
   \E^{\Q} \left[ \log \left(\frac{1 - |\alpha_\ell|^2}{|1 - \alpha_\ell Q_\ell|^2} \right) | \Fc_\ell \right]
   \leq \Oc(1) + \frac{2}{\beta} \log\left( \frac{n+k}{j+1} \right) \ .$$
We deduce a.s.:
\begin{align*}
         1 - |Q_{n+k}|^2
\ll \  & (1 - |Q_n|^2) \exp\left( -(\Nc_{k+n} - \Nc_{n}) + \Mc_{k+n} - \Mc_n \right)
                     \left( \frac{n+k}{n+1} \right)^{\frac{2}{\beta}}\\
     & + (1-r^2) \sum_{j=n+1}^{n+k}
       \exp\left( -(\Nc_{k+n} - \Nc_{j}) + \Mc_{k+n} - \Mc_j \right)
       \left( \frac{n+k}{j+1} \right)^{\frac{2}{\beta}} \ .
\end{align*}

Now, we have to control the martingales $\Nc$ and $\Mc$. We start by the easier $\Nc$, which is a convergent martingale, 
since it is (up to a shift of the index) equal to the martingale defined by \eqref{martingaleN}, which has been proven to be bounded in $L^2$. 
We have that:
$$ C_\omega^\Nc := \sup_{j \geq n} \left| \Nc_j - \Nc_n \right| < \infty \ ,$$
since $\left( \Nc_j - \Nc_n \right)_{j \geq n}$ is almost surely a Cauchy sequence. Furthermore, since $\Nc$ has independent increments, $C_\omega^\Nc$ is independent from $\Fc_n$, and its distribution under $\Q$ does not depend on $r$. As such, because single random variables are tight:
$$ \limsup_{a \rightarrow \infty}
   \Q\left( |C_\omega^\Nc| > a | \mathcal{G} \right) = 0,
$$
and \eqref{eq:conditionalTightness} is satisfied for $C_\omega^\Nc$. 

In order to control the contribution of $\left( \Mc_j(r) \right)_{j \geq n}$, we crucially use the epsilon of room between $\frac{2}{\beta}$ and $1$ in the subcritical regime. To that end, pick $\eta>0$, depending only on $\beta$,  such that $\frac{2}{\beta} + \eta < 1$, and consider the random variable
\begin{align}
\label{eq:defCMc}
C_\omega^{\Mc,k}(r) := & \sup_{n \leq j \leq n+k} \left[ \Mc_{n+k}(r) - \Mc_j(r) - \eta \log \frac{n+k+1}{j+1} \right] \ .
\end{align}
This quantity has a distribution that depends on $r$ and $k$. Upon using the bound 
\begin{align*}
     & \exp\left[ -(\Nc_{k+n} - \Nc_{j}) + \Mc_{k+n}(r) - \Mc_j(r) \right] \\
   = & \exp\left[ -(\Nc_{k+n} - \Nc_{n}) + \Nc_{j} - \Nc_n) \right] \\
     & \times \exp\left[ \Mc_{k+n}(r) - \Mc_{j}(r) - \eta \log \frac{n+k+1}{j+1} \right]
       \left( \frac{n+k+1}{j+1} \right)^{\eta} \\
\leq & e^{2 C_\omega^\Nc + C_\omega^{\Mc,k}(r)}
       \left( \frac{n+k+1}{j+1} \right)^{\eta} \ ,
\end{align*}
we have:
\begin{align*}
     & \ 1 - |Q_{n+k}|^2\\
\ll  & \ e^{2 C_\omega^\Nc + C_\omega^{\Mc,k}(r)} \left[ 
                (1 - |Q_n|^2)
                \left( \frac{n+k+1}{n+1} \right)^{\frac{2}{\beta}+\eta}
       + (1-r^2) \sum_{j=n+1}^{n+k}
       \left( \frac{n+k+1}{j+1} \right)^{\frac{2}{\beta}+\eta} \right] \ .
\end{align*}
From the classical comparison between series and integrals, we have:
$$ \forall k \in \N, \ 
   \sum_{j=0}^k
   \frac{1}{(j+1)^{\frac{2}{\beta} + \eta}}
   \leq \frac{ (k+1)^{1-\eta-\frac{2}{\beta}} }{1-\eta-\frac{2}{\beta}} \ ,
$$
and from the mean value theorem, we have:
$$ (1 - |Q_n|^2) \leq (1-r) |2 Q_n' Q_n|_{L^\infty(\D)} \leq 2 (1-r^2) |Q_n'|_{L^\infty(\partial \D)}\ .$$
In this last inequality, we used the fact that $|Q_n| \leq 1$ and that, by the maximum principle for subharmonic functions, $|Q_n'|$ reaches its maximum on the boundary of the disc. Therefore, the previous inequality becomes:
\begin{align*}
         1 - |Q_{n+k}|^2
\ll  \ & e^{2 C_\omega^\Nc + C_\omega^{\Mc,k}(r)} \ (1-r^2)
                \left[ 
                2|Q_n'|_{L^\infty(\partial \D)}
                \left( \frac{n+k+1}{n+1} \right)^{\frac{2}{\beta}+\eta}
       + \frac{n+k+1}{1-(\frac{2}{\beta}+\eta)} \right] \\
\ll \ & e^{2 C_\omega^\Nc + C_\omega^{\Mc,k}(r)} (n+k+1) \ (1-r^2)
                \left[ 
                \frac{ 2|Q_n'|_{L^\infty(\partial \D)} }{ n+1 }
       + \frac{1}{1-(\frac{2}{\beta}+\eta)} \right] \ .
\end{align*}
Moreover, by Szeg\"o recursion, $Q'_n(z)$ is a rational function of  $z$, the Verblunsky coefficients of index between $0$ and $n-1$, and their conjugates: 
it is then continuous in $z, \alpha_0, \dots, \alpha_{n-1}$, and bounded if we restrict to $|z| = 1$ and
$(\alpha_0, \dots, \alpha_{n-1}) \in L$ for some compact set $L \in \D^n$. 
Hence, $|Q_n'|_{L^\infty(\partial \D)}$ is uniformly bounded, independently of $r$, under the event $\mathcal{G}$. 
By absorbing all the constants into a single one, we deduce that  there exists a $C_\omega>0$, independent of $r$, which satisfies \eqref{eq:conditionalTightness} and such that:
\begin{align*}
         1 - |Q_{n+k}|^2
\leq  \ & C_\omega \ e^{C_\omega^{\Mc,k}(r)} (n+k+1) (1-r^2) \ .
\end{align*}

In order to complete the proof of \eqref{eq:cv2sde_estimate2},  it is now enough to show that for $\varepsilon'>0$, the random variables
\begin{equation}
\label{eq:todo}
        \left(  S(r) := \sup_{k \leq 1/ \log(1/r^2)}  \left( C_\omega^{\Mc,k}(r) + \varepsilon' \log [(n+k+1) (1-r^2)] \right) \right)_{r \in (0,1)}
\end{equation}
form a tight family, under $\Q$ and conditionally on $\Gc$. Indeed, by combining \eqref{eq:todo} with the previous equation, since
$$ \varepsilon_{r,n}
 = \max\left(n,  \left\lfloor \frac{\varepsilon}{1-r^2} \right\rfloor \right)
 \leq  n + \frac{1}{\log(1/r^2)} + \Oc(1) \ , $$
we obtain, by taking $\varepsilon' = 1/2$,
\begin{align*}
     & \sum_{j=n}^{\varepsilon_{r,n}-1} \frac{1 - |Q_{j}(r)|^2}{j+1} \leq 
  \Oc \left( \min (1, \log (1/r^2) )\right) +  \sum_{j=n}^{\min(\varepsilon_{r,n}-1, n + 1/ \log(1/r^2))} \frac{1 - |Q_{j}(r)|^2}{j+1}  \\
\leq &  \Oc \left( 1 - r^2 \right) + C_\omega \ e^{S(r)} \sum_{j=n}^{\varepsilon_{r,n}-1}
                \frac{1}{j+1} \left( (j+1)(1-r^2) \right)^{\half} \\
\ll   &  \Oc \left( 1 -r^2 \right) + C_\omega \ e^{S(r)} \sqrt{\varepsilon}.
\end{align*}
Now, the sum we want to estimate is non-empty only if $\varepsilon_{r,n} \geq n+1$ which implies $\varepsilon/(1-r^2) \geq 1$, i.e. 
$1-r^2 \leq \varepsilon$. We deduce 
$$\sum_{j=n}^{\varepsilon_{r,n}-1} \frac{1 - |Q_{j}(r)|^2}{j+1} \ll \varepsilon + C_\omega \ e^{S(r)} \sqrt{\varepsilon}.$$
and then 
$$\sup_{r \in (0,1)} \Q \left[ \sum_{j=n}^{\varepsilon_{r,n}-1} \frac{1 - |Q_{j}(r)|^2}{j+1}  \geq \xi | \mathcal{G} \right]
\leq \sup_{r \in (0,1)} \Q \left[ (1 +  C_\omega \ e^{S(r)} ) \gg \xi \varepsilon^{-1/2} | \mathcal{G} \right],$$
which goes to zero with $\varepsilon$ by the tightness of 
$(S(r))_{r \in (0,1)}$. 

 In order to be truly done with the proof of \eqref{eq:cv2sde_estimate2}, it remains to prove the tightness of \eqref{eq:todo}. 

\medskip

The proof of tightness of \eqref{eq:todo} is rather technical but essentially boils down to Doob's martingale inequality in order to control suprema and a dyadic decomposition argument. First, let us start with controlling the variations of the martingale $\mathcal{M}$, via estimates on conditional exponential moments. For $\lambda \in [-1,1]$, and $\ell \geq n$:
\begin{align*}
    \E^\Q\left( e^{\lambda (\Mc_{\ell+1} - \Mc_{\ell})} \ | \Fc_\ell \right)
= & \E\left( \frac{1-|\alpha_\ell Q_\ell|^2}{\left| 1 - \alpha_\ell Q_\ell \right|^{2(1+\lambda)}}
      \ | \Fc_\ell \right)
    e^{\lambda \E^{\Q} \left[ \log |1-\alpha_{\ell} Q_{\ell}|^2 | \Fc_{\ell} \right] } \ .
\end{align*}
By applying \eqref{eq:circleMomentBound} to the uniformly bounded exponent $\lambda+1$ and to 
$u = |\alpha_{\ell} Q_{\ell}| \leq |\alpha_{\ell}|$, and by using the  estimate of 
Proposition \ref{properties:underQ} on the conditional expectation of $-2 \Re  \log (1 - \alpha_j Q_j)$ under $\mathbb{Q}$, we get, for $\ell$ large enough depending on $\beta$:
\begin{align*}
  &  \E^\Q\left( e^{\lambda (\Mc_{\ell+1} - \Mc_{\ell})} \ | \Fc_\ell \right)\\
= & \E\left( (1-|\alpha_\ell Q_\ell|^2)
             \left( 1 + (1+\lambda)^2|\alpha_\ell Q_\ell|^2 + \Oc \left( \frac{|\alpha_\ell|^4}{(1 - |\alpha_{\ell}|^2)^{\mathcal{O}(1)}} \right) \right)
      \ | \Fc_\ell \right)\\
      & \times
    e^{-\frac{4}{\beta (\ell+1)} \lambda |Q_\ell|^2 + \Oc( (\ell+1)^{-2} ) } \\
= & \left( 1 + ((1+\lambda)^2-1) \frac{2|Q_\ell|^2}{\beta(\ell+1)} + \Oc\left( (\ell+1)^{-2} \right) \right)
    e^{-\frac{4}{\beta (\ell+1)} \lambda |Q_\ell|^2 + \Oc( (\ell+1)^{-2} ) } \\
= & \exp\left( ((1+\lambda)^2-1) \frac{2|Q_\ell|^2}{\beta(\ell+1)}
             - \frac{4}{\beta (\ell+1)} \lambda |Q_\ell|^2 + \Oc( (\ell+1)^{-2} ) \right) \\
= & \exp\left( \lambda^2 \frac{2|Q_\ell|^2}{\beta(\ell+1)}
               + \Oc( (\ell+1)^{-2} ) \right) \leq  \exp\left( \lambda^2 \frac{2}{\beta(\ell+1)}
               + \Oc( (\ell+1)^{-2} ) \right)  \,
\end{align*}
the last inequality coming from the fact that $|Q_{\ell}|$ is always smaller than or equal to $1$. 
Therefore, there exists a constant $c>0$ such that 
$$ \Ec^\lambda_{j,j'}
:= \exp\left( \lambda (\Mc_{j} - \Mc_{j'})
    - \lambda^2 \frac{2}{\beta} \sum_{\ell=j'}^{j-1} \frac{1}{\ell+1}
    - c \sum_{\ell=j'}^{j-1} \frac{1}{(\ell+1)^2} \right) \ ,$$
is a positive $(\F, \Q)$-supermartingale in $j \geq j'$, starting at $1$, for $j'$ large enough depending on $\beta$. 
We deduce that the probability that this supermartingale reaches a level $M > 0$ is at most $1/M$. 
Applying this for $\lambda \in (0,1)$ and for $-\lambda$, we deduce that 
for $a$ large enough depending on $\beta$ and $b \geq a \geq n$ (recall that $\mathcal{G}$ is $\mathcal{F}_n$-measurable): 
$$ \Q\left( \sup_{a \leq j \leq b} |\Mc_j - \Mc_a| \geq x  | \mathcal{G} \right)
   \ll e^{-\lambda x + \lambda^2 \frac{2}{\beta}\left( 1 + \log \frac{b+1}{a+1} \right)} \ .
$$
Hence for $\lambda \in (0,1)$:
\begin{equation} \Q\left( \sup_{a \leq j \leq b} |\Mc_j - \Mc_a| \geq \left( 1 + \log \frac{b+1}{a+1} \right)^\half x | \mathcal{G} \right)
   \ll e^{-\lambda x + \frac{2}{\beta} \lambda^2}. \label{boundincrementsM}
\end{equation}

Recall that, by    \eqref{eq:defCMc} and \eqref{eq:todo}, 
\begin{align*}
& S(r)   = \sup_{k \leq 1/ \log(1/r^2)}  \left( C_\omega^{\Mc,k}(r) + \varepsilon' \log [(n+k+1) (1-r^2)]  \right)
\\ & = \sup_{k \leq 1/ \log(1/r^2), n \leq j \leq n+k} \left( \Mc_{n+k}- \Mc_j- \eta \log \frac{n+k+1}{j+1}  + \varepsilon' \log [(n+k+1) (1-r^2)  \right).
\end{align*} 

Intuitively, we  will  show that the increments $\Mc_{n+k}- \Mc_j$ have, with high probability, an order of magnitude  dominated by a power of  $1 +  \log [(n+k+1)/(j+1)]$ with exponent 
smaller than $1$, and then its possible growth when $j$ becomes far from $n+k$ is compensated  by the negative term  $-\eta \log [(n+k+1)/(j+1)]$.  The last term 
$\varepsilon' \log [(n+k+1) (1-r^2)]$, is negative, and it will be useful in order to control a union bound on the different possible orders of magnitude of $k$. For such union bound, 
we need to control moments of $\Mc_{n+k}- \Mc_j$ with order larger than one: the choice of $1.25$ below is quite arbitrary (this value is small enough in order to allow us to control 
the moments). 
The following sum introduces two scales, corresponding to the fact that $S(r)$ involves a supremum on two indices: the order of magnitude of $k$ (approximately $2^p$) and the order of magnitude of $ \log [(n+k+1)/(j+1)]$ (approximately $2^m$). The non-zero terms of the sum corresponds to increments of  $\Mc_{n+k}- \Mc_j$ which are 
larger than $\log^{0.8}[(n+k+1)/(j+1)]$, and for which some care is needed in order to make sure they are well-controlled by the negative term  $-\eta \log [(n+k+1)/(j+1)]$. 

More precisely, for each integer $p \geq 0$, let us define: 

$$  S_p(r)
 := \sum_{m = 0}^{\infty} \max\left(0, 
	\sup_{k \in [2^p - 1, 2^{p+1} - 1]}
	\sup_{ \substack{ n \leq j \leq n+k \\
	                  \log \frac{n+k+1}{j+1} \leq 2^m} }
	       |\mathcal{M}_{n+k} - \mathcal{M}_j|^{1.25} - 2^{m}\right).$$
	       
	         For $k \in  [2^p - 1, 2^{p+1} - 1]$, and $n \leq j \leq n+k$, we can consider, in $S_p(r)$, the term of index 
	$m \geq 0$ such that $(1/2) + \log ((n+k+1)/(j+1)) \in [2^{m-1}, 2^m)$. 
	We deduce
	$$  |\mathcal{M}_{n+k} - \mathcal{M}_j|^{1.25} - 2^m \leq S_p(r)$$
	and then 
		$$  |\mathcal{M}_{n+k} - \mathcal{M}_j| \ll 2^{0.8 m} + S_p(r)^{0.8}
	\ll 1 +  \log^{0.8} ((n+k+1)/(j+1))  + S_p(r)^{0.8}.$$
	By \eqref{eq:defCMc}, we deduce, for $k \in [2^p-1, 2^{p+1} - 1]$, 
$$ C_\omega^{\Mc, k}(r)
   \ll S_p(r)^{0.8} + \sup_{n \leq j \leq n+k} \left( \log^{0.8} \frac{n+k+1}{j+1} - \eta \log \frac{n+k+1}{j+1} \right)
   \ll S_p(r)^{0.8} + 1 \ . $$

	Now, let us define   $p_r$ by $ 1/\log(1/r^2) \in [2^{p_{r} }-1, 2^{p_{r} +1} - 1]$. We see that $S(r)$ from \eqref{eq:todo} is controlled by
 $$\widetilde{S(r)} := \max_{p \leq p_{r} } \left(  S_p(r)^{0.8} - \varepsilon'' (p_{r} - p) \right)$$
 for $\varepsilon'' = \varepsilon' \log 2$. 
 We have, for all $x >1$, using a union bound and Markov's inequality:
 $$ \Q [ \widetilde{S(r)} \geq x | \mathcal{G}]
    \leq \sum_{p \leq p_{r}}  \Q [ S^{0.8}_p(r) \geq x + \varepsilon''(p_{r} - p) | \mathcal{G}]
     $$ $$ \leq \sum_{p \leq p_{r}} \frac{1}{(x + \varepsilon''(p_{r} - p))^{1.25}} \E^{\Q} [S_p(r) | \mathcal{G}]
  \ll_{\varepsilon''} x^{-0.25} \E^{\Q} [S_p(r) | \mathcal{G}].$$
  Notice that the introduction of an exponent larger than $1$ in the definition of $S_p(r)$ is used here in order to get a sum in $p$ which is bounded by a
  quantity
  tending to zero when $x \rightarrow \infty$. This convergence to zero ensures the tightness of $(S(r))_{r \in (0,1)}$, provided that we check that 
 $\E[ S_p(r)| \mathcal{G}]$ is bounded independently of $p$ and 
$\mathcal{G}$.

	In the definition of $S_p(r)$ given above, the double supremum is bounded by the supremum of  $|\mathcal{M}_j - \mathcal{M}_{j'}|^{1.25}$ where $j$ and $j'$ are in an interval $[a,b]$ such that $b = n +  2^{p+1} - 1$, and $\log ((n + 2^p)/(a+1)) \leq 2^m$.  Now:
    $$ \sup_{a \leq j,j' \leq b} |\mathcal{M}_j - \mathcal{M}_{j'}|^{1.25}
       \ll
       \sup_{a \leq j \leq b} |\mathcal{M}_j - \mathcal{M}_{a}|^{1.25} \ .
    $$

    Because of the estimate  \eqref{boundincrementsM}, this quantity has a $k$-th moment (conditionally on $\mathcal{G}$) dominated by $2^{1.25 \half k m}$ for fixed $k \geq 0$, and $a$ large enough depending on $\beta$. This last constraint can in fact be dropped since the individual increments of $\mathcal{M}$ after $n$ have all bounded conditional moments given $\mathcal{G}$ (they are dominated by the moments of the quantity $\log (1 - |\alpha_j|^2)$). 
   
Since the fourth moment of the double supremum $\Sigma_m$ is dominated by
    $$ 2^{1.25 \times 0.5 \times 4 m} = 2^{2.5 m} \ ,$$
   we have
	\begin{align*}
	\E^{\Q} [ \max(0, \Sigma_m - 2^{m}) | \mathcal{G} ] & \leq \E^{\Q} [  \Sigma_m \mathds{1}_{\Sigma_m \geq 2^{ m}} | \mathcal{G}] 
	\leq 2^{-3m} \E^{\Q} [\Sigma_m^{4}| \mathcal{G}]
	\leq 2^{-3m} 2^{2.5m}
	= 2^{-0.5m}
	\end{align*}	
	which implies the boundedness of  $\E^{\Q} [ S_p(r) | \mathcal{G}] $  by summing in $m$.

\subsection{Weak convergence to a solution of the SDE}
\label{section:cv2sde_2}

We start this subsection by a general theorem of convergence of discrete stochastic processes towards the solution of a SDE. Then, we apply this theorem to the setting of Proposition  \ref{prop:convergencetosde}.

\subsubsection{A general theorem for convergence of stochastic processes}
\begin{proposition}
\label{proposition:diffusiveLimit}
Let $(\varepsilon_n)_{n \geq 1}$ be a positive sequence converging to zero. Let $(X^{(n)})_{n \geq 1}$ be a family stochastic processes defined on  intervals $(I_n)_{n \geq 1}$ containing a fixed compact interval $I \subset \mathbb{R}_+$ and whose endpoints are multiples of $\varepsilon_n$, $X^{(n)}$ being continuous and piecewise linear on the intervals of the form $[k \varepsilon_n, (k+1) 
\varepsilon_n]$,  uniformly bounded, and  satisfying the following equation: 
$$ X^{(n)}_{(k+1)\varepsilon_n} -  X^{(n)}_{k\varepsilon_n}
 = b_n \left( k\varepsilon_n, X^{(n)}_{k\varepsilon_n} \right) \varepsilon_n
 + \sigma_n \left( k\varepsilon_n, X^{(n)}_{k\varepsilon_n} \right) \sqrt{\varepsilon_n} Y^{(n)}_k
$$
where $b_n: I_n \times \R \rightarrow \R$ and $\sigma_n: I_n \times \R \rightarrow \R_+$ are given functions.

We assume that $b_n$ and $\sigma_n$ are uniformly converging on $\mathbb{R} \times I$ to continuous and bounded functions $b$ and $\sigma$ when $n$ goes to infinity, and
$$ \E [ Y^{(n)}_k | (  X^{(n)}_{j \varepsilon_n} )_{ j \leq k}  ] = 0, \;  \E [ (Y^{(n)}_k)^2 | (  X^{(n)}_{j \varepsilon_n} )_{ j \leq k}  ] = 1 \ ,$$
$$\E [  (X^{(n)}_{(k+1)\varepsilon_n} -  X^{(n)}_{k\varepsilon_n} )^4  | (  X^{(n)}_{j \varepsilon_n} )_{ j \leq k}  ] = \Oc(\varepsilon_n^2) \ .$$
Moreover,  we suppose that $b$ and $\sigma$ satisfy the estimates: 
$$|b (t, x) - b(t,y)| \ll |x-y|, \; |\sigma(t,x) - \sigma(t,y)| \ll \sqrt{|x-y|}.$$
Then, the family of the laws of $X^{(n)}$ restricted to $I$ for $n \geq 1$ is tight and any subsequencial limit has the law of a solution of the SDE:
$$dX_t = b(t, X_t) dt + \sigma(t, X_t) dB_t \ ,$$
$B$ being a Brownian motion. 
\end{proposition} 

\begin{proof}
For the tightness, by the classical Kolmogorov criterion, it is enough to show that 
$$\E [ (X^{(n)}_t - X^{(n)}_s)^4 ] = \Oc ((t-s)^2)$$
for $|t-s|$ smaller than some absolute constant. 
Since we assume a linear interpolation, we can suppose $s < t$, $s = k \varepsilon_n $, $t = m \varepsilon_n $ for $k$ and $m$ integers. 
For $m \geq k$, we define 
$$\Delta_{k,m} = X^{(n)}_{m \varepsilon_n} - X^{(n)}_{k \varepsilon_n}$$ 
and we expand: 
$$\E [ \Delta_{k,m+1}^4 ] = \E [ \Delta_{k,m}^4 ] +  4  \E [ \Delta_{k,m}^3  \E [ \Delta_{m,m+1} | \Hc_m] ] 
+ 6   \E [ \Delta_{k,m}^2  \E [ \Delta_{m,m+1}^2 | \Hc_m] ] 
$$ $$+ 4  \E [ \Delta_{k,m}  \E [ \Delta_{m,m+1}^3| \Hc_m] ] 
+ \E[ \Delta_{m,m+1}^4],$$
where $\Hc_m$ is the $\sigma$-algebra generated by $X^{(n)}_{k \varepsilon_n}$ for $k \leq m$. 
We have  $$ \E [ \Delta_{m,m+1} | \Hc_m] =  b_n ( m\varepsilon_n,  X^{(n)}_{m \varepsilon_n}) \varepsilon_n = \Oc(\varepsilon_n)$$
since $b_n$ converges uniformly to $b$ and $b$ is bounded. 
Similarly, 
$$ \E [ \Delta^2_{m,m+1} | \Hc_m]
=  b^2_n (  m \varepsilon_n, X^{(n)}_{m \varepsilon_n}) \varepsilon^2_n  +   \sigma^2_n (m \varepsilon_n,  X^{(n)}_{m \varepsilon_n}) \varepsilon_n 
=  \Oc(\varepsilon_n),$$
$$ \E [ \Delta^4_{m,m+1} | \Hc_m] = \Oc(\varepsilon^2_n)$$
by assumption, 
and by using Cauchy-Schwarz inequality, 
$$ \E [| \Delta^3_{m,m+1}| \, \big| \Hc_m] = \Oc(\varepsilon^{3/2}_n).$$
Using H\"older inequality, we deduce that if $E_m = \E [ \Delta_{k,m}^4 ]$, 
we have 
$$E_{m+1} - E_m \ll (E_m^{3/4} + E_m^{1/2}) \varepsilon_n + E_m^{1/4}  \varepsilon_n^{3/2} +   \varepsilon_n^2.$$
As soon as $E_m \leq 1$ we deduce 
$$E_{m+1} - E_m \ll  E_m^{1/2} \varepsilon_n +  E_m^{1/4}  \varepsilon_n^{3/2} +  \varepsilon_n^2.$$
The term $ E_m^{1/4}  \varepsilon_n^{3/2} $ can be absorbed by the sum of the two others, because of Cauchy-Schwarz inequality. 
Hence, if $E_{m+1} \geq E_m \geq \varepsilon_n^2$, 
$$E^{1/2}_{m+1} - E^{1/2}_m  \leq \frac{ E_{m+1} - E_m}{2 E^{1/2}_m} \ll \varepsilon_n + \varepsilon_n^2  E^{-1/2}_m \ll \varepsilon_n,$$
which remains true if $E_{m+1} \geq E_m$ and $E_m \leq  \varepsilon_n^2$ since in this case 
$$E_{m+1} \leq E_m + \Oc (  E_m^{1/2} \varepsilon_n +   \varepsilon_n^2) \ll \varepsilon_n^2.$$
By induction, we deduce $E^{1/2}_m \ll (m-k) \varepsilon_n$ as soon  as $ (m-k) \varepsilon_n$ is smaller than some absolute constant. This is enough for tightness. 

Now, let $X$ be any limit in law of a subsequence of $(X^{(n)})_{n \geq 1}$. In order to prove that the law of $X$ is necessarily the unique weak solution of the above SDE, we prove that $X$ solves a well-posed martingale problem. Let $f$ be a smooth function with compact support. 
We have by Taylor's formula: 
$$ f( X^{(n)}_{(k+1)\varepsilon_n}) - f( X^{(n)}_{k\varepsilon_n}) = 
f'(X^{(n)}_{k\varepsilon_n}) \Delta_{k,k+1} + \frac{1}{2} f''(X^{(n)}_{k\varepsilon_n})\Delta^2_{k,k+1}  + \Oc_f (|\Delta_{k,k+1} |^3),$$
where the subscript $f$ means that the implicit constant may depend on the function $f$. 
Since 
$$
  \left\{
  \begin{array}{ccc}
  \E [  \Delta_{k,k+1}   | \Hc_k] & = & b_n ( k\varepsilon_n, X^{(n)}_{k\varepsilon_n}) \varepsilon_n \ ,\\
  \E [  \Delta^2_{k,k+1} | \Hc_k] & = & \sigma^2_n ( k \varepsilon_n, X^{(n)}_{k\varepsilon_n}) \varepsilon_n + \Oc( \varepsilon_n^2) \ , \\
  \E [| \Delta^3_{k,k+1}| \, \big| \Hc_k] & = & \Oc(\varepsilon^{3/2}_n) \ ,
  \end{array}
  \right.
$$
we get 
$$ \E \left(  f( X^{(n)}_{(k+1)\varepsilon_n}) - f( X^{(n)}_{k\varepsilon_n}) 
- \varepsilon_n [ f'(X^{(n)}_{k\varepsilon_n})  b_n ( k\varepsilon_n, X^{(n)}_{k\varepsilon_n} )   \right. $$
$$\left. +  \frac{1}{2} f''(X^{(n)}_{k\varepsilon_n}) \sigma^2_n ( k\varepsilon_n, X^{(n)}_{k\varepsilon_n})] \, \big| \Hc_k \right]
 =
 \Oc_f(\varepsilon^{3/2}_n) \ .
$$

Summing for consecutive values of $k$, and conditioning, we deduce that for $s <t $ multiples of $\varepsilon_n$ in $I$, 
\begin{align*}
&
\E \left[   f( X^{(n)}_{t}) - f( X^{(n)}_{s})
      - \int_s^t \left( f' ( X^{(n)}_{ \lfloor u  \rfloor_{\varepsilon_n}})
      b_n ( \lfloor u  \rfloor_{\varepsilon_n}, X^{(n)}_{ \lfloor u  \rfloor_{\varepsilon_n}})   \right. \right. \\
&
\left. \left.+ \frac{1}{2}  f'' ( X^{(n)}_{ \lfloor u  \rfloor_{\varepsilon_n}})  \sigma^2_n ( \lfloor u  \rfloor_{\varepsilon_n}, X^{(n)}_{ \lfloor u  \rfloor_{\varepsilon_n}})  \right) du \, \big|   (X^{(n)}_{v})_{v \leq s}  \right]
\end{align*}
is dominated by  $\varepsilon^{1/2}_n$ (with a constant depending on $f$), $\lfloor u  \rfloor_{\varepsilon_n}$ denoting the largest multiple of $\varepsilon_n$ which is smaller than or equal to $u$. 
Replacing $b_n$ by $b$ and $\sigma_n$ by $\sigma$ changes the integral by at most 
$$ |I| (\|f'\|_{\infty} \|b_n - b\|_{\infty}  + \|f''\|_{\infty} \| \sigma^2_n - \sigma^2\|_{\infty}) \underset{n \rightarrow \infty}{\longrightarrow} 0,$$
$|I|$ denoting the length of $I$. 
After that, replacing $ \lfloor u  \rfloor_{\varepsilon_n}$ by $u$ changes the integrand by at most $ w_A(f'b + \half f'' \sigma^2, \varepsilon_n + \eta_n)$, defined as 
$$\sup_{ \substack{x, y \in [-A,A]\\
                   u, v \in I\\
                   |x-y| + |u-v| \leq   \varepsilon_n + \eta_n}
       }
  \left|  \left(f' (x) b(u,x) + \half f''(x) \sigma^2 (u,x) \right)
       -  \left(f' (y) b(v,y) + \half f''(y) \sigma^2 (v,y) \right)   \right|,$$
as soon as $X^{(n)}_{ \lfloor u  \rfloor_{\varepsilon_n}}$ and $X^{(n)}_{u}$ are in $[-A,A]$ and their difference is at most $\eta_n$. 
Since  $X^{(n)}$ is uniformly bounded, we can choose $A$ in such a way that  the integral is changed by at most 
$$ |I|  w_A(f'b + f'' \sigma^2/2, \varepsilon_n + \eta_n)  $$ $$+ 
2 \|f'b + f'' \sigma^2/2\|_{\infty} \int_s^t  du \left(  \mathds{1}_{ |X^{(n)}_{ \lfloor u  \rfloor_{\varepsilon_n}} - X^{(n)}_{ u}| \geq \eta_n} \right).$$
The variation of the last conditional expectation   is then at most 
$$ |I|  w_A(f'b + f'' \sigma^2/2, \varepsilon_n + \eta_n)  $$ $$ + 
2  |I| \|f'b + f'' \sigma^2/2\|_{\infty} \eta_n^{-2} \sup_{u \in [s,t]} \E [(X^{(n)}_{ \lfloor u  \rfloor_{\varepsilon_n}} - X^{(n)}_{ u})^2  |   (X^{(n)}_{v})_{v \leq s} ]. $$
 Now, 
we have, with the previous notation, 
$$ \E [ \Delta_{k,k+1}^2 | \Hc_k]  = \Oc( \varepsilon_n)$$
and then the last conditional expectation is dominated by $\varepsilon_n$. 

The variation of the conditional expectation of the integral involving $f$ is then dominated by: 
$$ w_A(f'b + f'' \sigma^2/2, \varepsilon_n + \eta_n)  + 
 \|f'b + f'' \sigma^2/2\|_{\infty}  \eta_n^{-2} \varepsilon_n.$$

By uniform continuity, the first term goes to zero when $n$ goes to infinity if $\eta_n$ goes to zero.  We deduce, by taking $\eta_n = \varepsilon_n^{1/4}$, that 
$$ \E \left[   f( X^{(n)}_{t}) - f( X^{(n)}_{s}) - \int_s^t \left( f'( X^{(n)}_{u} ) b(u, X^{(n)}_u)  + \frac{1}{2}  f''(X^{(n)}_u)  \sigma^2 ( u, X^{(n)}_u )  \right) du \, \big|   (X^{(n)}_{v})_{v \leq s}  \right]$$
is bounded by a deterministic quantity $\delta_{n}$ which goes to zero when we let $n \rightarrow \infty$.
We then get, for all measurable functionals $G$,
$$ \left|\E \left[   \left( f( X^{(n)}_{t}) - f( X^{(n)}_{s}) - \int_s^t \big( f' ( X^{(n)}_{u} ) b ( u, X^{(n)}_u) \right. \right. \right. $$
$$\left. \left. \left.  + \frac{1}{2}  f'' ( X^{(n)}_{u})  \sigma^2 ( u, X^{(n)}_u )  \big) du \, \right) G \left(   (X^{(n)}_{v})_{v \leq s} \right)   \right]   \right|  \leq \|G\|_{\infty}  \delta_{n}.$$
if $s < t$ are multiples of $\varepsilon_n$. If in the big parenthesis, we replace $t$ by $t' \in [t, t+ \varepsilon_n]$ and $s$ by $s' \in [s, s + \varepsilon_n]$, 
we change the corresponding quantity by at most 
$$\|f'\|_{\infty} (| X^{(n)}_{t'} -X^{(n)}_{t}| +| X^{(n)}_{s'} -X^{(n)}_{s}| )   + 2  \|f'b + f'' \sigma^2/2\|_{\infty}  \varepsilon_n $$
 which shows that the estimate just above still occurs if we replace $s$ by $s'$ and $t$ by $t'$, with a possibly different $\delta_{n}$ satisfying the same properties. 
We can then write, by changing notation,
$$ \left|\E \left[   \left( f( X^{(n)}_{t}) - f( X^{(n)}_{s}) - \int_s^t \big( f' ( X^{(n)}_{u} ) b ( u, X^{(n)}_u) \right. \right. \right. $$
$$\left. \left. \left.  + \frac{1}{2}  f'' ( X^{(n)}_{u})  \sigma^2 ( u, X^{(n)}_u)  \big) du \, \right) G \left(   (X^{(n)}_{v})_{v \leq \lfloor s \rfloor_{\varepsilon_n} } \right) \right]   \right|  \leq \|G\|_{\infty}  \delta_{n}$$
 for all $s <t$ in any interval $[k \varepsilon_n, m \varepsilon_n]$ contained in $I$. Hence, 
for $s' < s < t$ fixed in the interior of $I$, we have for $n$ large enough and all measurable functionals $G$. 
$$ \left|\E \left[   \left( f( X^{(n)}_{t}) - f( X^{(n)}_{s}) - \int_s^t \big( f' ( X^{(n)}_{u} ) b ( u, X^{(n)}_{ u} ) \right. \right. \right. $$
$$\left. \left. \left.  + \frac{1}{2}  f'' ( X^{(n)}_{u})  \sigma^2 ( u, X^{(n)}_{u} )  \big) du \, \right) G \left(   (X^{(n)}_{v})_{v \leq s' } \right) \right]   \right|  \leq \|G\|_{\infty}  \delta_{n}$$
If $G$ is a continuous, bounded functional from $\mathcal{C}(I, \mathbb{R}) $ to $\mathbb{R}$, the quantity inside the expectation is continuous and bounded in the trajectory of $X^{(n)}$, since this process is uniformly bounded by some constant $A$ and $f''$, $\sigma$ and $b$ are uniformly continuous in $[-A,A]$, $[-A,A] \times I$ and $[-A,A] \times I$ respectively. 
We deduce that if the law of $X$ is a limit point of the family of laws of  $X^{(n)}$ restricted to $I$, then 
$$\E \left[   \left( f( X_{t}) - f( X_{s}) - \int_s^t \big( f' ( X_{u} ) b( u, X_u ) + \frac{1}{2}  f'' ( X_{u})  \sigma^2 ( u, X_u)  \big) du \, \right) G \left(   (X_{v})_{v \leq s' } \right) \right] = 0$$
for all $s' < s < t$ in the interior of $I$. Taking limits and using dominated convergence, we can let $s = s'$ and allow $s'$ and $t$ to be at the boundary of $I$. As such for all smooth functions $f$ with bounded support, and bounded continuous functionals $G$, and $s<t$ in $I$, we have:
$$\E \left[   \left( f( X_{t}) - f( X_{s}) - \int_s^t \big( f' ( X_{u} ) b( u, X_u ) + \frac{1}{2}  f'' ( X_{u})  \sigma^2 ( u, X_u)  \big) du \, \right) G \left(   (X_{v})_{v \leq s } \right) \right] = 0 \ .$$

Now, let $t_0$ be  the left-hand point of the interval $I$, and let us fix $x_0$. In order to invoke the machinery of the martingale problem, we will need a fixed initial condition - as the literature is stated in that form. Since $X$ belongs to the space of continuous functions, which is a separable complete metric space, the regular conditional probability $\P^{t_0, x_0} = \P\left( \cdot \ | \ X_{t_0} = x_0 \right)$ does exist (see \cite{D19}). The above equation says that for all smooth $f$ with bounded support:
$$ M^f_t := f(X_{t}) - f(X_{t_0}) - \int_{t_0}^t \big( f' ( X_{u} ) b( u, X_u ) + \frac{1}{2}  f'' ( X_{u})  \sigma^2 ( u, X_u)  \big) du$$
is an $\left( \P, \Fc^X \right)$-martingale, where $\Fc^X$ denotes the natural filtration of $X$. We then get, for $t_0 \leq s \leq t$, and for a bounded continuous functional 
$G$,  
\begin{align*}
\mathbb{E}^{\P^{t_0,X_{t_0}}}  [  M^f_t G \left(   (X_{v})_{  v \leq s } \right) ] & = \mathbb{E}^{\P}  [   M^f_t G \left(   (X_{v})_{ v \leq s } \right) | X_{t_0} ] 
 = \mathbb{E}^{\P} [  \mathbb{E}^{\P}    [   M^f_t G \left(   (X_{v})_{t_0 \leq v \leq s } \right) |  \Fc^X_s  ] | X_{t_0}] \\ & =  \mathbb{E}^{\P} [   M^f_ s G \left(   (X_{v})_{t_0 \leq v \leq s } \right) | X_{t_0}]  = \mathbb{E}^{\P^{t_0,X_{t_0}}}  [  M^f_s G \left(   (X_{v})_{  v \leq s } \right) ] 
\end{align*}
almost surely. Hence, for  $\P_{X_{t_0}}(dx)$-almost every $x_0$, we have 
$$\mathbb{E}^{\P^{t_0,x_0}}  [  M^f_t \tilde{G} \left(   (X_{sv})_{  v \leq 1 } \right) ] 
=\mathbb{E}^{\P^{t_0,x_0}}  [  M^f_s  \tilde{G} \left(   (X_{sv})_{  v \leq 1 } \right) ] $$
for $t_0 \leq s \leq t$, $\tilde{G}$ bounded and continuous functional on $\mathcal{C}([0,1], \mathbb{R})$, $f$ smooth with compact support, all these elements being restricted to arbitrary countable sets. 

Replacing $s$ and $t$ by $s'$ and $t'$ for $s' > s$, $t' > t$ in the chosen countable set, supposed to be dense, and letting $s' \rightarrow s$ and $t' \rightarrow t$, we can drop the restriction on $s$ and $t$ by dominated convergence.
Moreover, since for two smooth functions $f$ and $g$ with compact support, $$ |  M^f_t -  M^g_t  | \ll_{b,\sigma,I}(1+t) ( \|f-g\|_{\infty} +  \|f'-g'\|_{\infty}  +  \|f''-g''\|_{\infty}),$$
we can drop the restriction on $f$ after considering a dense subset of smooth functions with respect to the $C^2$ norm. 

We then have for  $\P_{X_{t_0}}(dx)$-almost every $x_0$, 
$$\mathbb{E}^{\P^{t_0,x_0}}  [ ( M^f_t -  M^f_s) H( X_{sv_0}, \dots, X_{s v_r}) ] = 0$$
for all smooth $f$ with compact support, all $t \geq s \geq t_0$, all $v_0, \dots v_r \in \mathbb{Q} \cap [0,1]$ and all polynomials $H$ with rational coefficients. By dominated convergence, one can drop the assumption $v_0, \dots, v_r \in \mathbb{Q}$ and only assume that  $H$ is a continuous function. Using again dominated convergence, the equality remains true 
when $H$ is the indicator of a product of intervals. By monotone class theorem, we easily deduce that for $\P_{X_{t_0}}(dx)$-almost every $x_0$ and all smooth $f$ with compact support, 
$M^f$ is a $(\P^{t_0, x_0}, \Fc^X)$-martingale, i.e. the law $\Lc\left( X_{t}, \ t \geq t_0 \ | \ X_{t_0} = x_0 \right)$ is a solution of the martingale problem associated to the SDE and with initial condition $x_0$. The regularity assumptions on the coefficients of $b$ and $\sigma$ allow us to apply the result by Yamada and Watanabe given in \cite[Theorem 8.2.1]{SV07}.  It says that the SDE satisfies Itô uniqueness and that the martingale problem is well-posed for deterministic initial conditions. This means that  for $\P_{X_{t_0}}(dx)$-every $x_0$, the law $\Lc\left( X_{t}, \ t \geq t_0 \ | \ X_{t_0} = x_0 \right)$, is uniquely determined and coincides with the law of the solution of the SDE with initial condition $x_0$. 

In the end, as required, the law of $X$ is uniquely determined by its initial probability distribution at the left-hand point of $I$. And $X$ has the same distribution as the solution of the SDE with the same initial distribution. 

\end{proof} 

\subsubsection{An application of the previous result}
\label{section:convergenceSDEchaos}
We will now apply the result of the previous subsection to the particular case we are interested in. Recall that the goal is to prove that $\left( X^{(r)} ; r < 1\right)$ is tight and that any limit point solves the SDE \eqref{mainsde}.

We start by observing that, thanks to Proposition \ref{properties:underQ}, there exists an $(\F, \Q)$-martingale $N$ with normalized bracket such that, 
for $Q_j := Q_j(r)$, 
\begin{align}
  \label{eq:doobMeyerQj}
  & 1 - |Q_{j+1}|^2\\
\nonumber
= & 1-r^2 + r^2 \left( 1 - |Q_{j}|^2 \right)
           \left( 1 - \frac{2}{\beta(j+1)}\left( 1 - |Q_{j}|^2 \right)
                    + \frac{4}{\beta(j+1)} |Q_{j}|^2 + \Oc( \frac{1}{(j+1)^2}) \right)\\
\nonumber
  & + \sqrt{r^4 \left( 1 - |Q_{j}|^2 \right)^2
        \left( \frac{4 |Q_j|^2}{\beta(j+1)}
               + \Oc\left( \frac{1}{(j+1)^2}\right)
        \right)} \Delta N_k \ .
\end{align}

We choose $\varepsilon \in (0,1)$ and $A > 1$, a sequence $(r_m)_{m \geq 1}$ in $(0,1)$ which tends to $1$ and such that $r_m$ is sufficiently close to $1$ for all $m$, and we apply Proposition \ref{proposition:diffusiveLimit} to: 

$$\varepsilon_m  = \log(1/r_m^2),  \; X^{(m)}_{k \varepsilon_m}  =  |Q_{n+k}|^2, \;  I_m = \R_+, \; I = [\varepsilon, A],$$
$$b_m (t,x) = (\log (1/r_m^2))^{-1}  \left[ (1-x) - (1 - r_m^2) - r_m^2 (1-x) \left( 1 - \frac{2}{ \beta (j + 1)}(1-x) \right. \right. $$
$$\left. \left. + \frac{4}{\beta (j+1)} x + \Oc ( (j+1)^{-2}) \right)
\right],$$
and 
$$\sigma^2_m (t,x) = (\log (1/r_m^2) )^{-1} r_m^4 (1-x)^2  \left( \frac{4x}{\beta(j+1)} + \Oc ((j+1)^{-2}) \right),$$
for $j = n + t/ \log (1/r_m^2)$, which tend uniformly to 
$$b (t,x) = -x + (1-x) \left( \frac{2}{\beta t} (1-x) - \frac{4x}{\beta t} \right)$$
and 
$$\sigma^2(t,x) = \frac{4 x (1-x)^2}{\beta t}$$
when $t \in I$. 
We then also have the uniform convergence of $\sigma_m$ towards $\sigma$ for $t \in I$ since the square root function is uniformly continuous.
Moreover, the condition on the conditional fourth moment is ensured by 
  \eqref{eq:boundmomentorder4}, and the Lipschitz-like conditions for $b$ and $\sigma$ are also satisfied. 
From Proposition \ref{proposition:diffusiveLimit},  we deduce that the family, indexed by $m$,  of distributions of the linear interpolation 
$X^{(r_m)}_t$ 
of $t \mapsto |Q_{n + t/ \log (1/r_m^2)}|^2$, restricted to the interval $[\varepsilon, A]$, is tight, and 
any limit point satisfies the SDE of Proposition \ref{prop:convergencetosde} on the interval $[\varepsilon, A]$. 

From the tightness of $(S(r))_{r \in (0,1)}$ (see \eqref{eq:todo}), we get 
that for $k \leq 1/\log (1/r^2)$, and $\varepsilon' \in (0,1)$, 
$$ 1 - |Q_{n+k}|^2 \leq C_{\omega} e^{S(r)} \left[(n+k+1)(1-r^2) \right]^{1- \varepsilon'} $$ $$
\leq  C_{\omega} e^{S(r)}  [ [(1-r^2)(n+1)]^{1- \varepsilon'} + (k \log (1/r^2))^{1- \varepsilon'}] $$
which implies that the family of conditional distributions of
$$\sup_{t \in [0,1]} \frac{1 - X_t^{(r)}}{ t^{1 - \varepsilon'} + (1-r^2)^{1 - \varepsilon'}},$$
given $\mathcal{G}$, remains tight  with respect to  $r \in (0,1)$.

This allows us to extend the tightness of the conditional law  of 
 $X^{(r_m)}$ given $\mathcal{G}$ from the interval $[\varepsilon,A]$ to the interval $[0,A]$. 
Indeed, tightness is ensured by the fact that for all $\xi > 0$, 
$$\underset{m \rightarrow \infty}{\lim \sup} \, \Q [ w_{X^{(r_m)}, [0,A]} (\delta) > \xi  | \mathcal{G}] 
\underset{\delta \rightarrow 0}{\longrightarrow} 0,$$
where $w_{X^{(r_m)}, [0,A]} $ denotes the modulus of continuity of 
$X^{(r_m)}$ restricted to the interval $[0,A]$. 
We already know tightness on the interval $[\varepsilon, A]$
and then 
$$\underset{m \rightarrow \infty}{\lim \sup} \, \Q [ w_{X^{(r_m)}, [\varepsilon,A]} (\delta) > \xi /2 | \mathcal{G}] 
\underset{\delta \rightarrow 0}{\longrightarrow} 0.$$
Moreover, for $\delta \in (0, \varepsilon)$,
\begin{align*}
& \; \; \underset{m \rightarrow \infty}{\lim \sup} \,  \Q [ w_{X^{(r_m)}, [0,\varepsilon]} (\delta) > \xi /2 | \mathcal{G}] 
 \leq \underset{m \rightarrow \infty}{\lim \sup}  \,  \Q [ w_{X^{(r_m)}, [0,\varepsilon]} (\varepsilon) > \xi /2 | \mathcal{G}] 
 \\ & \leq \underset{r \rightarrow 1}{\lim \sup} \, \Q \left[ 
\sup_{t \in (0,1]} \frac{1 - X_t^{(r)}}{ t^{1 - \varepsilon'} + (1-r^2)^{1 - \varepsilon'}} 
> \frac{\xi/2}{\varepsilon^{1 - \varepsilon'} + (1-r^2)^{1 - \varepsilon'}} |  \mathcal{G} \right].
\end{align*}
We deduce, for all $\varepsilon \in (0,1)$, 
\begin{align*}
& \underset{\delta \rightarrow 0}{\lim \sup} \; 
\underset{m \rightarrow \infty}{\lim \sup} \, \Q [ w_{X^{(r_m)}, [0,A]} (\delta) > \xi  | \mathcal{G}] 
\\ & \leq  \underset{r \rightarrow 1}{\lim \sup} \,  \Q \left[ 
\sup_{t \in (0,1]} \frac{1 - X_t^{(r)}}{ t^{1 - \varepsilon'} + (1-r^2)^{1 - \varepsilon'}} 
> \frac{\xi/2}{\varepsilon^{1 - \varepsilon'} + (1-r^2)^{1 - \varepsilon'}} |  \mathcal{G} \right]
\\ & \leq \underset{r \rightarrow 1}{\lim \sup} \,  \Q \left[ 
\sup_{t \in (0,1]} \frac{1 - X_t^{(r)}}{ t^{1 - \varepsilon'} + (1-r^2)^{1 - \varepsilon'}} 
> \frac{\xi/2}{2 \varepsilon^{1 - \varepsilon'}} |  \mathcal{G} \right]
\end{align*}
Letting $\varepsilon \rightarrow 0$ gives the tightness we are looking for. 

This tightness shows that any sequence $(X^{(r_m)})_{m \geq 1}$ for $r_m \rightarrow 1$ has a 
subsequence which tends in law (conditionally on $\mathcal{G}$) to a limiting process.
Moreover, this limit should satisfy the SDE of the proposition on the full open interval $(0, \infty)$, 
since it satisfies the equation on any interval of the form $[\varepsilon,A]$. 

Let $X$ be a process which is the limit in law of a subsequence of $(X^{(r_m)})_{m \geq 1}$, and let us now show that 
$$\sup_{t \in (0,1]} \frac{ 1 - X_t}{t^{1-\varepsilon'}} < \infty$$
almost surely. 
Since $X$ has the limiting distribution of a subsequence of  $(X^{(r_m)})_{m \geq 1}$,  we have, by continuity of the underlying functional,
for all $\varepsilon \in (0,1)$, $\xi > 2$, 
$$\P \left[ \sup_{t \in (\varepsilon,1]} \frac{ 1 - X_t}{t^{1-\varepsilon'}} > \xi \right]
 \leq \underset{r \rightarrow 1}{\lim \sup} \ \Q \left[ 
\sup_{t \in (\varepsilon,1]} \frac{1 - X_t^{(r)}}{ t^{1 - \varepsilon'}} 
> \xi - 1 |  \mathcal{G} \right]$$
and then 
\begin{align*}
\P \left[ \sup_{t \in (\varepsilon,1]} \frac{ 1 - X_t}{t^{1-\varepsilon'}} > \xi \right]
 & \leq \underset{r \rightarrow 1}{\lim \sup} \ \Q \left[ 
\sup_{t \in (\varepsilon,1]} \frac{1 - X_t^{(r)}}{ t^{1 - \varepsilon'} + (1-r^2)^{1- \varepsilon'}} 
> \xi - 2 |  \mathcal{G} \right]
\\ & \leq \underset{r \rightarrow 1}{\lim \sup} \ \Q \left[ 
\sup_{t \in (0,1]} \frac{1 - X_t^{(r)}}{ t^{1 - \varepsilon'} + (1-r^2)^{1- \varepsilon'}} 
> \xi - 2 |  \mathcal{G} \right],
\end{align*}
which implies 
\begin{align*}
\P \left[ \sup_{t \in (0,1]} \frac{ 1 - X_t}{t^{1-\varepsilon'}} = \infty \right]
& \leq \underset{\varepsilon \rightarrow 0}{\lim \sup}  \, 
\P \left[ \sup_{t \in (\varepsilon,1]} \frac{ 1 - X_t}{t^{1-\varepsilon'}} > \xi \right]
\\ & \leq \underset{r \rightarrow 1}{\lim \sup} \ \Q \left[ 
\sup_{t \in (0,1]} \frac{1 - X_t^{(r)}}{ t^{1 - \varepsilon'} + (1-r^2)^{1- \varepsilon'}} 
> \xi - 2 |  \mathcal{G} \right],
\end{align*}
for all $\xi > 2$, and then by letting $\xi \rightarrow \infty$, 
$$\P \left[ \sup_{t \in (0,1]} \frac{ 1 - X_t}{t^{1-\varepsilon'}} = \infty \right] = 0.$$
This immediately provides the a.s. integrability of $(X_t - e^{-t})/t$ on the interval $(0,1]$. 
For the integrability at infinity, we observe that from 
\eqref{squareAinfty}, 
$$\E^{\Q} \left[ \sum_{k \geq \max(n, \lfloor (1-r^2)^{-1} \rfloor} 
\frac{|Q_{k+1}(r)|^2}{k+2} | \mathcal{G} \right] = \Oc(1),$$
which easily implies for $r$ close enough to $1$ (depending on $n$), 
$$\E^{\Q} \left[ \int_1^{\infty} \frac{X_t^{(r)}}{t} dt  | \mathcal{G} \right] = \Oc(1),$$
and then 
$$\E^{\Q} \left[ \int_1^{B} \frac{X_t^{(r)}}{t} dt  | \mathcal{G} \right] = \Oc(1)$$
uniformly in $B > 1$. 
Since $X$ is a limit point of a subsequence of $(X^{(r_m}))_{m \geq 1}$, and the integral on a compact set with respect to $dt/t$ is a continuous functional, 
we have 
$$\E \left[ \int_1^{B} \frac{X_t}{t} dt \right] = \Oc(1),$$
again independently of $B$. By monotone convergence, 
$$\E \left[ \int_1^{\infty} \frac{X_t}{t} dt \right] = \Oc(1),$$
which implies the integrability of  $\frac{X_t - e^{-t}}{t}$ on $[1, \infty)$. 

To end the proof of the Proposition \ref{prop:convergencetosde}, it now remains to prove that the law of $X$ is uniquely determined by the fact that it satisfies the SDE and that
$$\sup_{t \in (0,1]} \frac{ 1 - X_t}{t^{1-\varepsilon'}} < \infty.$$
The proof is given in the next two subsections. 

\subsection{Entrance law from Dufresne's identity}
\label{section:cv2sde_3}

An amusing fact is that the entrance law of the process $X$ uses Dufresne's identify which relates the perpetuity of a Brownian motion with drift to the inverse of a Gamma random variable.

More precisely, if $W$ is a Brownian motion, then Dufresne's identity states that (see \cite[p.78]{BS12}) for all $b > 0$:
\begin{align}
\label{eq:dufresne}
2\int_0^{\infty} \exp\left(- 2 ( W_{w} + b w ) \right) d w \eqlaw  \frac{1}{\gamma_b}
\end{align}
where $\gamma_{b}$ is a Gamma random variable of parameter $b$. This plays an important role in the following:
\begin{lemma}
Let $(X_t)_{t \geq 0}$ be a continuous process, satisfying the SDE \eqref{mainsde}, and such that 
$$\sup_{t \in (0,1)} \frac{1- X_t}{t^{1- \varepsilon'}} dt < \infty$$
almost surely, for all $\varepsilon' \in (0,1)$. Then, the following convergence in law holds:
$$ \frac{1-X_t}{t} 
   \stackrel{t \rightarrow 0}{\longrightarrow}
   \frac{\beta}{2 \gamma_\nu}
$$ 
where $\nu = \frac{\beta}{2} -1$.
\end{lemma}
\begin{proof}
Recall that $X_0 = 1$ and $X$ satisfies the SDE:
$$ d(1-X_t) =   X_t dt
              - \left( 1 - X_t \right)^2 \frac{2 dt}{\beta t}
              + 4\left( 1 - X_t \right) X_t \frac{dt}{\beta t}
              + \sqrt{ \left( 1 - X_t \right)^2 \frac{4 X_t}{\beta t} } dB_t \ .$$
             
Via It\^o's formula, we can write, for all $t_0 > 0$,  on the event when $X_{t_0} < 1$ and in  the interval between $t_0$ and the first  hitting time $T_{t_0}^1$ of $1$ after $t_0$: 
\begin{align}
    -d\log(1-X_t) 
= & \frac{1}{1-X_t} dX_t
    + \half \frac{1}{(1-X_t)^2} d\langle X, X \rangle_t \nonumber \\
= & - \frac{X_t}{1-X_t} dt
    + \left( 1 - X_t \right) \frac{2 dt}{\beta t}
    - 4 X_t \frac{dt}{\beta t}
    - \sqrt{ \frac{4 X_t}{\beta t} } dB_t
    + \frac{2X_t}{\beta t} dt \nonumber \\
= & - \frac{X_t}{1-X_t} dt
    + \frac{2 dt}{\beta t}
    - 4 X_t \frac{dt}{\beta t}
    - \sqrt{ \frac{4 X_t}{\beta t} } dB_t \ . \label{edslog}
\end{align}
Let us show that this equation is in fact satisfied for all $t>0$. The Dambis-Dubins-Schwarz Theorem
states that a continuous martingale can be written as a time-changed Brownian motion (see \cite{RevuzYor}, Chapter V). More precisely, if $M$ is a continuous local martingale vanishing at zero, then there exists a Brownian motion 
$(W_s)_{s \geq 0}$ such that 
$$W_s = M_{ \inf \{u, \langle M, M \rangle_u > s \}}$$
when $s \leq  \langle M, M \rangle_{\infty}$, $ \langle M, M \rangle$ denoting the quadratic variation of $M$. 
If $\langle M, M \rangle$ is continuous and strictly increasing, we deduce that 
$$M_u = W_{\langle M, M \rangle_u}$$
for all $u \geq 0$. In the present situation, 
we get that 
$$- \int_{t_0}^u \sqrt{ \frac{4 X_t}{\beta t} } dB_t
= W_{ \int_{t_0}^u \frac{4 X_t}{\beta t} dt }$$
for all $u \geq t_0$. Notice that the quadratic variation is strictly increasing since $X$ cannot be equal to zero on a nontrivial interval: this would contradict
the SDE.  Since $X$ is always bounded by $1$, we deduce that for all $t_1 > t_0$, 
$$\sup_{t \in [t_0, \min(t_1, T_{t_0}^1) )} [-\log (1-X_{t}) + \log (1- X_{t_0})] \leq \frac{2}{\beta} \log (t_1/t_0) + \sup_{ 0 \leq s \leq \frac{4}{\beta} \log (t_1/t_0)} W_s
< \infty.$$
Taking limit at the end of the interval, we deduce 
$$-\log (1-X_{\min(t_1, T_{t_0}^1)}) + \log (1- X_{t_0}) < \infty$$
i.e. $X_{\min(t_1, T_{t_0}^1)} < 1$ and then $T_{t_0}^1 > t_1$. Since, $t_1 > t_0$ is arbitrary, $T_{t_0}^1 = \infty$, which means that $X$ almost surely never returns to $1$ after $t_0$, conditionally on the event $X_{t_0} < 1$. 
In particular, for all $t_0 < t_1$, 
$$\P [ X_{t_0} < 1,  X_{t_1} = 1 ] = 0$$
and then taking a countable union, 
$$\P [X_{t_1} = 1, \exists t \in \Q \cap (0, t_1), X_t < 1] = 0,$$
and by continuity
$$\P [X_{t_1} = 1, \exists t \in (0, t_1), X_t < 1] = 0,$$
i.e. 
$$\P [X_{t_1} = 1] = \P [ \forall t \in (0, t_1], X_t = 1 ] = 0$$
since the fact that $X$ remains equal to $1$ in a non-trivial interval contradicts the SDE satisfied by $X$. 
Hence, for all $t_0 > 0$, $X_{t_0} < 1$ almost surely, which by the previous reasoning, implies a.s. that $X$ never hits $1$ after $t_0$. 
Taking the countable intersection of the events for $t_0 \in \Q \cap (0, 1]$, we deduce that a.s. $X$  is strictly smaller than $1$ everywhere except at time $0$. Hence, \eqref{edslog} is almost surely satisfied for all $t \in (0, \infty)$. 

\medskip

As we shall see, Eq. \eqref{edslog} can be recast into the following Volterra equation:
\begin{align}
\label{eq:volterra}
(1-X_t) & = \int_0^t ds \ e^{-(t-s)} \left( \frac{s}{t} \right)^{\frac{2}{\beta}}
            \exp\left(\int_s^t \frac{4 X_u}{\beta u} du
            + \int_s^t \sqrt{\frac{4 X_u}{\beta u} } dB_u \right) \ .
\end{align}

To do so, fix an arbitrary time, say $1$, and write:
\begin{align}
\label{eq:volterraCandidate}
 Y_t & = (1-X_t) e^t t^{\frac{2}{\beta}}
                 \exp\left( \int_t^1 \frac{4 X_u}{\beta u} du
                           +\int_t^1 \sqrt{\frac{4 X_u}{\beta u} } dB_u \right) \ .
\end{align}

First, let us prove that almost surely, $\lim_{t \rightarrow 0} Y_t = 0$. Fix an integer $r \geq 1$. Using the Dambis-Dubins-Schwarz Theorem, there exists a Brownian motion $W^{(r)}$ such that for $t \geq e^{-r}$:
$$ \int_{e^{-r}}^t \sqrt{\frac{X_u}{u} } dB_u
   =
   W^{(r)}_{\int_{e^{-r}}^t \frac{X_u}{u} du} \ .
$$
Because $X_u \leq 1$, we obtain:
$$\sup_{e^{-r} \leq t \leq 1} \left| \int_{e^{-r}}^t \sqrt{\frac{X_u}{u} } dB_u \right|
  \leq \sup_{0 \leq s \leq \int_{e^{-r}}^1 \frac{1}{u} du}| W^{(r)}_s|
  =    \sup_{0 \leq s \leq r}| W^{(r)}_s | \eqlaw \sqrt{r} V
$$
 where $V$ is a random variable with subexponential tails. By  the Borel-Cantelli Lemma, we deduce that there is random  variable $C_\omega < \infty$, such that
$$\forall r \geq 1, 
  \sup_{e^{-r} \leq t \leq 1}  \left| \int_{e^{-r}}^t \sqrt{\frac{X_u}{u} } dB_u \right|
  \leq C_\omega r^{3/4} \ ,$$
and then
$$\forall r \geq 1, \ 
  \forall e^{-r} \leq t \leq 1, \
\left|  \int_{t}^1  \sqrt{\frac{X_u}{u} } dB_u \right|
  \leq 2 C_\omega r^{3/4} \ ,$$
which gives 
$$ \forall t \in (0,1],
 \ \left| \int_{t}^1 \sqrt{\frac{X_u}{u} } dB_u \right| \leq 2 C_\omega \left( 1+|\log t| \right)^{3/4}
 \ .
$$

The previous bound applied to Eq. \eqref{eq:volterraCandidate} gives
$$Y_t  \leq  (1-X_t) e^t t^{-\frac{2}{\beta}}
       e^{2 C_\omega \left( 1+|\log t| \right)^{3/4}} \ .$$
By assumption on the process $X$ and the fact that $\frac{2}{\beta} < 1$ we deduce that $Y_t$ almost surely goes to zero when  $t \rightarrow 0$. We are now ready to recast the SDE on $X$ into a Volterra equation. From \eqref{edslog}, we deduce that  $d \log Y_t = dt/(1-X_t)$, and then, since $Y$ tends to zero at zero: 
$$ Y_t = \int_0^t ds \frac{Y_s}{1-X_s}.
$$
which implies the Volterra equation \eqref{eq:volterra}.
 
 \medskip 
 
 We are now ready to finish the argument. We have 
 $$0 \leq \int_s^t \frac{4 (1-X_u)}{\beta u} du \leq \int_0^t \frac{4 (1-X_u)}{\beta u} du$$
 which tends to zero in probability when $t$ goes to zero, since $(1-X_u)/u$ is integrable because of the assumption on $X$. 

On the other hand, by using again the  Dambis-Dubins-Schwarz theorem, we have for $\varepsilon \in (0,t)$, 
    $$\sup_{\varepsilon \leq s \leq t}  \left| \int_\varepsilon^s  (1- \sqrt{X_u}) \sqrt{\frac{4}{\beta u} } dB_u \right| 
    \leq \sup_{0 \leq \ell \leq L} |\gamma^{(\varepsilon)}_\ell|$$
    where $\gamma^{(\varepsilon)}$ is a Brownian motion and 
$$L = \int_\varepsilon^t  (1- \sqrt{X_u})^2  \frac{4}{\beta u} du
    \leq \int_0^t  (1- X_u)  \frac{4}{\beta u} du \ .
$$
  We deduce that for $a, \alpha > 0$, 
\begin{align*}
     & \ \P\left[ \sup_{\varepsilon \leq s \leq t}  \left| \int_\varepsilon^s  (1- \sqrt{X_u}) \sqrt{\frac{4}{\beta u} } dB_u \right|  > a \right]\\
\leq & \ \P\left[ \int_0^t  (1-X_u)  \frac{4}{\beta u} du \geq \alpha \right]
       + \P\left[  \sup_{0 \leq \ell \leq \alpha} |\gamma^{(\varepsilon)}_\ell| \geq a \right] \\ 
=    & \ \P\left[ \int_0^t  (1-X_u)  \frac{4}{\beta u} du \geq \alpha \right]
       +   \P\left[ \sqrt{\alpha} V \geq a \right] .
\end{align*}
Using the triangle inequality and letting $\varepsilon \rightarrow 0$, we get 
$$ \P \left[ \sup_{0 < s \leq t}  \left| \int_s^t  (1- \sqrt{X_u}) \sqrt{\frac{4}{\beta u} } dB_u \right|  >2 a \right]
    \leq
    \P \left[ \int_0^t  (1- X_u)  \frac{4}{\beta u} du \leq \alpha \right]
    +  \P\left[ \sqrt{\alpha} V \geq a \right] .$$
    
    By the assumption on the behavior of $X$ at zero, we have almost surely
$$\int_0^t (1- X_u)  \frac{4}{\beta u} du
  \underset{t \rightarrow 0}{\longrightarrow} 0 \ .
$$
    We deduce 
    $$\underset{t \rightarrow 0}{\lim \sup}  \, \mathbb{P} \left[ \sup_{0 < s \leq t}  \left| \int_s^t  (1- \sqrt{X_u}) \sqrt{\frac{4}{\beta u} } dB_u \right|  >2 a \right] 
    \leq   \P\left[ \sqrt{\alpha} V \geq a \right] \ .$$
    By letting $\alpha \rightarrow 0$, we deduce that 
    $$\sup_{0 < s \leq t}  \left| \int_s^t  (1- \sqrt{X_u}) \sqrt{\frac{4}{\beta u} } dB_u \right|  \underset{t \rightarrow 0}{\longrightarrow} 0$$
in probability. Hence, 
$$A_t := \sup_{0 < s \leq t}  \left|\int_s^t \frac{4 (1-X_u)}{\beta u} du  +  \int_s^t  (1- \sqrt{X_u}) \sqrt{\frac{4}{\beta u} } dB_u \right|
      = o_{\mathbb{P}} (1)
$$
where $o_{\mathbb{P}} (1)$ denotes any quantity tending to zero in probability with  $t$. As such, by making the substitution in Eq. \eqref{eq:volterra}, we find:
\begin{align*}
1-X_t = & e^{o_{\mathbb{P}}(1)}
          \int_0^t ds \ \left( \frac{s}{t} \right)^{\frac{2}{\beta}}
                   \exp\left(\int_s^t \frac{4 }{\beta u} du
                   + \int_s^t \sqrt{\frac{4}{\beta u} } dB_u \right) \\
      = & (1+o_{\mathbb{P}}(1))
          \int_0^t ds \
                   \exp\left( \frac{2}{\beta} \log \frac{t}{s}
                   + \int_s^t \sqrt{\frac{4}{\beta u} } dB_u \right) \ .
\end{align*}
We deduce, by using the Dambis-Dubins-Schwarz theorem for Wiener integrals, there exists a Brownian motion $W^{(t)}$ such that:
$$ \left( \int_s^t \sqrt{\frac{1}{\beta u} } dB_u \ ; \ 0 < s \leq t \right)
 = \left( -W^{(t)}_{\frac{1}{\beta}\log \frac{t}{s}} \ ; \ 0 < s \leq t   \right) \ .$$
The change of variables $s = t e^{-\beta w}$ yields
\begin{align*}
1-X_t = & (1+o_{\mathbb{P}}(1)) 
          \int_0^\infty t \beta e^{-\beta w} dw \
                   \exp\left( 2w
                   - 2 W^{(t)}_w \right) \\
      = & (1+o_{\mathbb{P}}(1)) \ t \ \frac{\beta}{2} \
          \left[2 \int_0^\infty dw \
          \exp\left( - 2 (W^{(t)}_w + (\frac{\beta}{2}-1) w) \right) \right].
\end{align*}
We are done thanks to Dufresne's identity \eqref{eq:dufresne}.
\end{proof}

\subsection{Uniqueness in law of solutions of the SDE}
\label{section:cv2sde_4}

\begin{lemma}
Let $(X_t)_{t \geq 0}$ be a continuous process, satisfying the SDE \eqref{mainsde}, and such that 
$$\sup_{t \in (0,1)} \frac{1- X_t}{t^{1- \varepsilon'}} dt < \infty$$
almost surely, for all $\varepsilon' \in (0,1)$.  Then the law of $X$ is uniquely determined.
\end{lemma}
\begin{proof}
Let $X^1$ and $X^2$ be two solutions. In order to prove that $X^1 \eqlaw X^2$ as processes, we need to prove that finite dimensional distributions on $\{ t > 0\}$ coincide.

Nevertheless, because $X^1$ and $X^2$ solve the same SDE which admits strong solutions for $t>0$, both processes enjoy the Markov property, with the same transition kernels. Therefore, we only need to prove that marginals match:
$$ \forall t>0, \ X^1_t \eqlaw X^2_t \ .$$
In other words, one can fix $t_0 \in (0,1)$ and prove that $X^1_{t_0}$ and $X^2_{t_0}$ have the same law. 
Although $X^1_0 = X^2_0 = 1$ is the natural continuation at $t=0$, we do not have the Markov property in order to directly make the transition from time $0$ to time $t_0$. From the previous section, since both $X^1$ and $X^2$ have the required tightness property at $0$ and solve the SDE, they have the same "entrance law":
$$ \frac{1-X^1_t}{t} \stackrel{t \rightarrow 0}{\longrightarrow} G ,$$
$$ \frac{1-X^2_t}{t} \stackrel{t \rightarrow 0}{\longrightarrow} G,$$
for  $$G \eqlaw \frac{\beta}{2 \gamma_{(\beta/2)-1}}.$$ 
From this convergence in law, one deduces  that  for all $\eta> 0$, there exists $\delta \in (0,t_0)$ small enough, depending on $\eta$, such that we can couple the distributions of $(X^1_t)_{t \geq \delta}$ and  $(X^2_t)_{t \geq \delta}$, in such a way that with probability larger than $1- \eta$, $|L^1_{\delta} - L^2_{\delta}| \leq \eta$, 
where 
$$L^j_t := \log \left( \frac{ 1- X^j_t}{t} \right)$$
and $(X^1_t)_{t \geq \delta}$ and  $(X^2_t)_{t \geq \delta}$ are driven by the same Brownian motion. From the SDE given by \eqref{edslog}, the drift term in $L^j_t$ is nonincreasing with respect to the value of $L^j_t$. 
 Moreover, since the SDE has strong solutions, the corresponding stochastic flows do not intersect, which implies that 
 the relative order of $L^1$ and $L^2$ never changes. We deduce that the process $(|L^1_t - L^2_t|)_{t \geq \delta}$ is a nonnegative supermartingale, and then for all $t \geq \delta$, 
$$\P\left[ |L^1_{t} - L^2_{t}| \geq \eta^{\half} \ | \ |L^1_{\delta} - L^2_{\delta}| \leq \eta \right]
 \leq
 \E\left[ \frac{|L^1_{t} - L^2_{t}|}{\eta^{\half}} \ | \ |L^1_{\delta} - L^2_{\delta}| \leq \eta \right]
 \leq \eta^{\half},$$
 which implies 
 $$\P [ |L^1_{t_0} - L^2_{t_0}| \geq \eta^{1/2} ] \leq \eta^{1/2} + \P [ |L^1_{\delta} - L^2_{\delta}| > \eta ]
 \leq \eta + \eta^{1/2}.$$
Hence, the characteristic functions of $L^1_{t_0} $ and $L^2_{t_0} $, taken at $\lambda \in \mathbb{R}$, differ 
by at most 
$$\E [ \min (2, \lambda |L^1_{t_0} - L^2_{t_0}|) ] \leq \lambda \eta^{1/2} + 2 \P [ |L^1_{t_0} - L^2_{t_0}| \geq \eta^{1/2} ] 
\leq \eta + (1+\lambda )\eta^{1/2}.$$
This bound does not depend on the coupling between $L^1_{t_0}$ and $L^2_{t_0}$ and is available for all $\eta > 0$, which implies that 
$L^1_{t_0}$ and $L^2_{t_0}$ have the same distribution.  Hence, $X^1_{t_0} \eqlaw X^2_{t_0}$.
\end{proof}

\section{Concluding the proof of the Main Theorem \ref{thm:main}}
\label{section:proofmain2}

Finally, we are ready to tackle the proofs of Lemma \ref{lemma:RenormalizedGMC_L1Limit} and Lemma \ref{lemma:OmegaRhoControl}, thus completing the proof of the Main Theorem.

\begin{proof}[Proof of Lemma \ref{lemma:RenormalizedGMC_L1Limit}]
If we multiply the quantity inside the expectation by the indicator of $\mathds{1}_{M_{\infty}^{-1} \leq A}$ for some $A > 0$, the convergence occurs since $GMC_r^{\gamma} (f)$ converges to  $GMC^{\gamma}(f)$ in $L^1$. 
Hence, the upper limit of the left-hand side is at most 
$$\|f\|_{\infty} \underset{r \rightarrow \infty}{\lim \sup} \ \E  \left[ \mathds{1}_{M_{\infty}^{-1} > A}  \frac{1}{M_{\infty}} GMC_r^{\gamma} (\partial \D) \right]
+ \|f\|_{\infty} \E  \left[ \mathds{1}_{M_{\infty}^{-1} > A}  \frac{1}{M_{\infty}} GMC^{\gamma} (\partial \D) \right] \ .$$
By Fatou's lemma:
$$ \E  \left[ \mathds{1}_{M_{\infty}^{-1} > A}  \frac{1}{M_{\infty}} GMC^{\gamma} (\partial \D) \right]
   \leq
   \underset{r \rightarrow \infty}{\lim \inf} \ 
   \E  \left[ \mathds{1}_{M_{\infty}^{-1} > A}  \frac{1}{M_{\infty}} GMC_r^{\gamma} (\partial \D) \right]
   \ ,
$$
which is trivially smaller than the $\limsup$, and therefore the upper limit is at most:
$$ 2  A^{-1}  \|f\|_{\infty} 
   \underset{r \rightarrow \infty}{\lim \sup} \,  \E  \left[ \frac{1}{M^2_{\infty}} GMC_r^{\gamma} (\partial \D) \right]
   \ .
$$
Since $A$ is arbitrarily large, it is enough to show that $\E  \left[ \frac{1}{M^2_{\infty}} GMC_r^{\gamma} (\partial \D) \right] $ is bounded, uniformly in $r <1$ sufficiently close to $1$. By rotational invariance, we get 
\begin{align*}
\E  \left[ \frac{1}{M^2_{\infty}} GMC_r^{\gamma} (\partial \D) \right]  
& = \E \left [ M_{\infty}^{-1}   \prod_{j=0}^{\infty} \frac{1 - |\alpha_j|^2}
                                          {|1 - \alpha_j Q_j(r)|^{2}}
                                          e^{ \frac{2}{\beta(j+1)}(1-r^{2j+2}) } \right]
\\ & =  \E^{\Q}  \left [ M_{\infty}^{-1}   e^{\rho_{r,0} + \omega_{r,0}} \right]
\\ & \leq  \left( \E^{\Q}  \left [ M_{\infty}^{-3} \right] \right)^{1/3}  \left(\E^{\Q}  \left [ e^{3 \rho_{r,0} }\right] \right)^{1/3}
   \left(\E^{\Q}  \left [ e^{3\omega_{r,0}} \right] \right)^{1/3}
\end{align*} 
Recall that 
$$\rho_{r,n} = \sum_{j=n}^{\infty}  \left( - \log \left( \frac{ 1- |\alpha_j Q_j(r)|^2}{1 - |\alpha_j|^2} \right)+ \frac{2}{\beta} \frac{1 - |Q_j(r)|^2}{j+1} \right)$$
and 
$$\omega_{r,n} = \frac{2}{\beta} \sum_{j=n}^{\infty} \frac{|Q_j(r)|^2 - r^{2j+2}}{j+1}.$$

Since the law of $(|\alpha_j|)_{j \geq 0}$ is the same under $\P$ and $\Q$, we have
\begin{align*}
\E^{\Q}  \left [ M_{\infty}^{-3} \right] 
& = \E^{\Q}  \left [ \prod_{j=0}^{\infty}  (1 - |\alpha_j|^2)^3  e^{\frac{6}{\beta (j+1)} }\right] 
= \E  \left [ \prod_{j=0}^{\infty}  (1 - |\alpha_j|^2)^3  e^{\frac{6}{\beta (j+1)}} \right] 
\\ & \leq \underset{m \rightarrow \infty}{\lim \inf} \prod_{j=0}^m \mathbb{E} [ (1 - |\alpha_j|^2)^3  e^{\frac{6}{\beta (j+1)}}]
\end{align*}
where the last inequality is due to Fatou's lemma and the independence of $(\alpha_j)_{j \geq 0}$.  
Now, 
$$
\mathbb{E} [ (1 - |\alpha_j|^2)^3  e^{\frac{6}{\beta (j+1)}}]
 = e^{3/\beta_j}  \, \beta_j \, \int_0^1 (1-x)^{\beta_j + 2} dx 
 =  e^{3/\beta_j} \, \frac{\beta_j}{\beta_j + 3} = 1 + \mathcal{O}(1/j^2), $$
which  shows that 
$$\E^{\Q}  \left [ M_{\infty}^{-3} \right] < \infty.$$
 We also know, 
from Proposition \ref{proposition:boundOmegaRho}, that $\rho_{r,0}$ and $\omega_{r,0}$ have exponential moments of all orders, bounded independently of $r$. 
\end{proof}

\begin{proof}[Proof of Lemma \ref{lemma:OmegaRhoControl}]
We have for $\eta >0 $, 
$$\E^{\Q} \left[ |e^{\rho_{r,n}} - 1| e^{\omega_{r,n}} \,| \Fc_n \right] 
\leq \E^{\Q} \left[ \mathds{1}_{|e^{\rho_{r,n}} - 1| \geq \eta} (1 + e^{\rho_{r,n}}) e^{ \omega_{r,n}}  \,| \Fc_n \right]  
+ \eta \E^{\Q} \left[e^{\omega_{r,n}} \,| \Fc_n \right] $$

In Proposition \ref{proposition:boundOmegaRho}, we have proven that the conditional  exponential moments of order $p$ of $\omega_{r,n}$ and $\rho_{r,n}$ 
are bounded by a quantity depending only on $p$ and $\beta$, independently of $r$, and that on the event 
$$\mathcal{A}_L := \{ (\alpha_0, \dots, \alpha_{n-1} ) \in L\},$$
$L$ being a given compact set of $\mathbb{D}^n$, 
$\Q [ |e^{\rho_{r,n}} - 1| \geq \eta | \Fc_n]$ is bounded by a quantity depending only on 
$\beta, \eta, n, L, r$ and tending to $0$ when $r$ goes to $1-$. 

We deduce, by applying H\"older inequality and letting $\eta \rightarrow 0$, that on $\mathcal{A}_L$, the conditional expectation $\E^{\Q} \left[ |e^{\rho_{r,n}} - 1| e^{\omega_{r,n}} \,| \Fc_n \right]$ is bounded by a deterministic quantity depending on $\beta, n, L, r$ and tending to zero when $r$ goes to $1-$. Since it is also uniformly bounded by a quantity 
depending only on $\beta$, we have 
$$ \E^{\Q}  \left[ \E^{\Q} \left[ |e^{\rho_{r,n}} - 1| e^{\omega_{r,n}} \,| \Fc_n \right]  \right]  \underset{r \rightarrow 1-}{\longrightarrow} 0 $$ 
by dominated convergence. 
It is then enough to show that 
$$ \E^{\Q} \left| \E^{\Q} \left[ e^{\omega_{r,n}}  \,| \Fc_n \right]- K_{\beta}  \right|
   \underset{r \rightarrow 1-}{\longrightarrow} 0 $$
for some $K_{\beta} \in \R$.

\medskip

Now, for all $u > 0$, 
$$ \E^\Q \left| 
         \E^{\Q} \left[ e^{ \omega_{r,n}}  - e^{ \min( \omega_{r,n} ,u)}  \,| \Fc_n \right] 
         \right|
   \leq e^{-u} \E^{\Q} \left[ e^{ 2\omega_{r,n}} \right] = \Oc (e^{-u})$$
since the exponential moments of $\omega_{r,n}$ are bounded uniformly in $r$ and $n$. It is then enough to show that for all $u > 0$, there exists some $K_{\beta, u} \in \mathbb{R}$ such that
$$
   \E^{\Q} \left| \E^{\Q} \left[ e^{\min(\omega_{r,n},u)}  \,| \Fc_n \right]- K_{\beta, u}  \right|
   \underset{r \rightarrow 1-}{\longrightarrow} 0 \ .$$
In this case, $K_{\beta,u}$ is bounded uniformly in $u$ because of the bound on the exponential moments 
of $\omega_{r,n}$ and one can take for $K_{\beta}$ a limit point of $K_{\beta,u}$ for $u \rightarrow \infty$. It is a fortiori  sufficient to prove that for some random variable $\omega$ with distribution depending only on $\beta$,  and for all continuous, bounded functions $G$ from $\mathbb{R}$ to $\mathbb{R}$, 
$$ 
   \E^{\Q} \left| \E^{\Q} \left[ G(\omega_{r,n})  \,| \Fc_n \right]- \E [G(\omega)] \right|
   \underset{r \rightarrow 1-}{\longrightarrow} 0 \ .$$

By dominated convergence, it is then enough to show that some determination of the conditional law of $\omega_{r,n}$ given $\mathcal{F}_n$, under $\Q$, almost surely converges to the distribution of a  random variable $\omega$ depending only on $\beta$. 

In order to avoid the issue of fixing these determinations of conditional laws (recall that conditional expectations are only defined almost surely), 
let us use the following trick. First, since the full event can be approximated by events of the form $\mathcal{A}_L$, it is enough to show, for all 
compact sets $L \in \D^n$, 
$$ 
   \E^{\Q} \left[ \mathds{1}_{\mathcal{A}_L} \left| \E^{\Q} \left[ G(\omega_{r,n})  \,| \Fc_n \right]- \E [G(\omega)] \right| \right]
   \underset{r \rightarrow 1-}{\longrightarrow} 0 \ .$$
Let $\mathcal{E}$ be the $\mathcal{F}_n$-measurable event, defined by
$$\mathcal{E} = \{ \E^{\Q} \left[ G(\omega_{r,n})  \,| \mathcal{F}_n \right]- \E [G(\omega)]  > 0 \}. $$
This event may depend on $r$. 
The left-hand side of the last convergence can be written as 
$$\E^{\Q} [  \E^{\Q} \left[ G(\omega_{r,n})  \,| \mathcal{F}_n \right]- \E [G(\omega)]  \; | \mathcal{A}_L, \mathcal{E} ] \; \mathbb{\Q} [ \mathcal{E} \cap \mathcal{A}_L] 
$$ $$+ \E^{\Q} [  \E^{\Q} \left[  - G(\omega_{r,n})  \,| \mathcal{F}_n \right] + \E [G(\omega)]  \; | \mathcal{A}_L, \mathcal{E}^c ] \; \mathbb{\Q} [ \mathcal{E}^c \cap \mathcal{A}_L] $$
and then, since $\mathcal{E}$ and $\mathcal{A}_L$ are in $\mathcal{F}_n$, 
it is equal to 
$$( \E^{\Q} \left[ G(\omega_{r,n})   \; | \mathcal{A}_L, \mathcal{E} \right]- \E [G(\omega)]  ) \; \mathbb{\Q} [ \mathcal{E} \cap \mathcal{A}_L] 
$$ $$+ ( \E^{\Q} \left[  - G(\omega_{r,n})   \; | \mathcal{A}_L, \mathcal{E}^c \right] + \E [G(\omega)]   ] ) \; \mathbb{\Q} [ \mathcal{E}^c \cap \mathcal{A}_L] $$
We deduce that it is sufficient to prove the following: for any event $\mathcal{G}$ of non-zero probability, possibly dependent on $r$, included in $\mathcal{A}_L$ and $\mathcal{F}_n$-measurable, 
the conditional law of $\omega_{r,n}$ given $\mathcal{G}$ tends to the distribution of a random variable $\omega$ depending only on $\beta$, when $r \rightarrow 1-$. 

\medskip

Recall that for $\varepsilon \in (0,1)$ and $A>1$, we have:
$$ \omega_{r,n} = \omega_{r,n}^{(0, \varepsilon)}
                + \omega_{r,n}^{(\varepsilon, A)}
                + \omega_{r,n}^{(A, \infty)} \ ,
                $$
where for $a<b$:
$$ \omega_{r,n}^{(a, b)}
 = \frac{2}{\beta} \sum_{j=a_{n,r}}^{b_{n,r}-1} \frac{|Q_j(r)|^2-r^{2j+2}}{j+1}\ .
$$
and $a_{n,r} = \max\left( n, \lfloor a/(1-r^2) \rfloor \right)$. From Proposition \ref{prop:convergencetosde}, we deduce that for all $\eta>0$:
$$ \limsup_{\varepsilon \rightarrow 0} \sup_{r \in (0,1)}
   \Q\left( \left| \omega_{r,n}^{(0, \varepsilon)}\right| \geq \eta | \mathcal{G} \right)
   = 0 \ ,
$$
$$ \limsup_{A \rightarrow \infty} \sup_{r \in (0,1)}
   \Q\left( \left| \omega_{r,n}^{(A, \infty)}\right| \geq \eta | \mathcal{G} \right)
   = 0. 
$$ 

From this bound, it is easy to deduce the convergence we are looking for, from the convergence, 
 for all positive integers $A$ and all $\varepsilon \in (0,1)$, of  the conditional distribution of $\omega_{r,n}^{(\varepsilon,A)}$, under $\Q$ and given $\mathcal{G}$, to a variable $\omega^{(\varepsilon,A)}$ depending only on $\beta$, $\varepsilon$ and $A$.
 Notice that we have convergence in law of $\omega^{(\varepsilon,A)}$ towards $\omega$, when $A \rightarrow \infty$ and $\varepsilon \rightarrow 0$.
 
Now, let us consider the process $X_t^{(r)}$ which is equal to $|Q_{n + (t/ \log (1/r^2))} (r)|^2$ when $n + (t/ \log (1/r^2))$ is an integer and which is linearly interpolated otherwise. In other words, for all $t \geq 0$, 
 $$ X_t^{(r) } = \sum_{j \geq n} w_j^r(t) |Q_j(r)|^2 \ ,$$
where $w_j^r(t) = w_0\left(- (j-n) + \frac{t}{\log (1/r^2)} \right)$ and 
$$ w_{0} (t) = 
   \left\{
   \begin{array}{cc}
   0     & \textrm{if } \left| t \right| > 1    \ , \\
   1-|t| & \textrm{if } \left| t \right| \leq 1 \ .
   \end{array}
   \right.
$$
That is to say that the graph of $w_j^r$ is given by a triangle of width $2 \log (1/r^2)$ and height $1$, centered at $t = (j-n) \log (1/r^2)$. 
 
We then have 
\begin{align*}
    \int_{\varepsilon}^A \frac{X_t^{(r)} }{t} dt
= & \sum_{j \geq n} |Q_j(r)|^2 \int_{\varepsilon}^A \frac{w_j^r(t)}{t} dt\\
= & \sum_{j \geq n} |Q_j(r)|^2
    \int_{-(j-n) + \varepsilon/\log (1/r^2)}^{-(j-n) + A/\log (1/r^2)} \frac{w_0(t)}{j-n+t} dt \ .
\end{align*}
Since $w_0$ is supported on $[-1, 1]$, each integral in the above formula vanishes if $-(j-n) + \varepsilon/\log (1/r^2) \geq 1$ or $-(j-n) + A/\log (1/r^2) \leq -1$. Therefore, we can restrict the summation index $j$ to 
\begin{align}
\label{eq:indexRestriction}
& \varepsilon/\log (1/r^2) - 1 < j-n < A/\log (1/r^2) + 1 \ .
\end{align}
For such indices, $j$ grows to infinity as $r \rightarrow 1^-$ and the contribution of each bounded number of terms to the series goes to zero as $r \rightarrow 1^-$. This remark allows us to take care of two boundary effects:

\begin{itemize}
 \item The integrals $\int_{-(j-n) + \varepsilon/\log (1/r^2)}^{-(j-n) + A/\log (1/r^2)} \frac{w_0(t)}{j-n+t} dt$ can be taken over $[-1, 1]$ since the integration interval does not cover $[-1, 1]$ only for a bounded number of indices $j$ such that \eqref{eq:indexRestriction} occurs. 
 \item Because $\varepsilon/\log (1/r^2) = \varepsilon_{n,r} + \Oc_{n, \varepsilon} (1)$ for $r \in (1/2,1)$ and the same for $A_{n,r}$, instead of restricting indices $j$ to \eqref{eq:indexRestriction}, we can restrict to $\varepsilon_{r,n} \leq j \leq A_{r,n}-1$.
\end{itemize}

Hence, if $o_r(1)$ denotes any quantity tending to zero when $r \rightarrow 1-$, $n$, $\varepsilon$ and $A$ being fixed: 
\begin{align*}
    \int_{\varepsilon}^A \frac{X_t^{(r)} }{t} dt
= & o_r(1)
    +
    \sum_{j=\varepsilon_{n,r}}^{A_{n,r}-1} |Q_j(r)|^2
    \int_{[-1, 1]} \frac{w_0(t)}{j-n+t} dt \\
= & o_r(1)
    +
    \sum_{j=\varepsilon_{n,r}}^{A_{n,r}-1} |Q_j(r)|^2
    \left( \int_{[-1, 1]} \frac{w_0(t)}{j+1} dt + \Oc_n  \left( \frac{1}{(j+1)^2} \right) \right)\\
= & o_r(1)
    +
    \sum_{j=\varepsilon_{n,r}}^{A_{n,r}-1} \frac{|Q_j(r)|^2}{j+1} \ .
\end{align*}

A comparison between Riemann sums and integrals easily gives 
$$ \int_{\varepsilon}^A \frac{e^{-t}}{t} dt
   = o_r(1)
  +  \sum_{j=\varepsilon_{r,n}}^{A_{r,n}-1} \frac{r^{2(j+1)}}{ j+1}, $$
and then, the combination of the two previous equations yields 
$$ \omega_{r,n}^{(\varepsilon,A)}
   = \frac{2}{\beta} \left(o_r(1) + \int_{\varepsilon}^A \frac{X_{t}^{(r)} - e^{-t}}{t} dt \right) \ .$$

\medskip

Now, from Proposition \ref{prop:convergencetosde}, $(X_{t}^{(r)})_{t \geq 0}$, conditionally on $\mathcal{G}$, under $\Q$, tends in law to a limiting stochastic process $(X_t)_{t \geq 0}$ whose distribution depends only on $\beta$, for the topology of uniform convergence on compact sets. Since the map 
$$Y \mapsto \frac{2}{ \beta} \int_{\varepsilon}^A \frac{Y_t - e^{-t}}{t} dt$$
is continuous for this topology, the continuous mapping theorem entails that 
$\omega_{r,n}^{(\varepsilon,A)}$ converges in distribution to 
 $$ \omega_{\varepsilon,A} := \frac{2}{\beta} \int_{\varepsilon}^A \frac{X_t - e^{-t}}{t} dt \ .$$
Since the integral
$$ \frac{2}{\beta} \int_0^{\infty} \frac{X_t - e^{-t}}{t} dt$$
is absolutely convergent by Proposition \ref{prop:convergencetosde}, it defines a random variable $\omega$ which is the limit of $\omega_{\varepsilon, A}$ when $\varepsilon$ goes to zero and $A$ goes to infinity. This completes the proof of the Main Theorem.
\end{proof}

\appendix 
\section{CBE as regularization of a Gaussian space}
\label{section:appendix_CBE}

Here we provide a short proof independent from \cite{JM15} that traces from the Circular Beta Ensemble become Gaussian as $n \rightarrow \infty$. This proof is absent from the literature and is in fact hidden in the book of Macdonald \cite{macdo}. 

Unlike \cite{DS94} or \cite{JM15}, this proof is not quantitative. It shows that the $C\beta E_n$ is the regularization of a Gaussian space by $n$ points at the level of symmetric functions.

\begin{lemma}[Gaussianity of traces]
Given a unitary matrix $U_n$ whose spectrum is sampled following the $C \beta E_n$, we have, for any positive integer $M$, the convergence in distribution:
$$ \left( \operatorname{tr}\left( U_n^k \right)  \right)_{k \in \{1,2, \dots, M\}} 
   \stackrel{n \rightarrow \infty}{\longrightarrow}
   \left( \sqrt{\frac{2 k}{\beta}} \Nc^\C_k\  \right)_{k \in \{1,2, \dots, M\}}   \ .
$$ 
\end{lemma}
\begin{proof}
Consider the following specialization of power sum polynomials, which form a basis of symmetric functions:
       $$ p_k := p_k(U_n) =\operatorname{tr} \left( U_n^k \right) $$
       and for a partition $\lambda$
       $$ p_\lambda := \prod_i p_{\lambda_i} \ .$$

Also, consider the following scalar product for functions in $n$ variables. Given $f, g: (\partial \D)^n \rightarrow \C$ symmetric, define:
       $$ \langle f, g \rangle_n := \E_{C\beta E_n}\left( f(y_1, \dots, y_n) \overline{g(y_1, \dots, y_n)} \right)\,$$
       where $\{y_1, \dots, y_n\}$ follows the $C\beta E_n$. 
       
The convergence in law of traces of the $C \beta E_n$ to Gaussians is equivalently reformulated as the convergence of the $\langle p_\lambda, p_\mu \rangle_n$ to the appropriate limit. This is readily obtained from the combination of the two following facts:
\begin{itemize}
 \item \cite[Chapter VI, \textsection 9, ``Another scalar product'', Theorem (9.9)]{macdo}: The scalar product $\langle \cdot, \cdot \rangle_n$ approximates the Macdonald scalar product in infinitely many variables
       $ \langle \cdot, \cdot \rangle_n \rightarrow \langle \cdot, \cdot \rangle \ ,$
       where
       $$ \langle p_\lambda, p_\mu \rangle = \delta_{\lambda, \mu} Cste(\lambda) \ .$$
 \item \cite[Chapter VI, \textsection 10, ``Jack symmetric functions'']{macdo}: The Macdonald scalar product has a Gaussian space lurking behind as
       $$ Cste(\lambda) \delta_{\lambda, \mu}
          = z_\lambda \left( \frac{2}{\beta} \right)^{\ell(\lambda)} \delta_{\lambda, \mu}
          = \E\left( \prod_{k} \left( \sqrt{\frac{2 k}{\beta}} \Nc_k^\C \right)^{m_k(\lambda)}
                               \left( \sqrt{\frac{2 k}{\beta}} \overline{\Nc_k^\C} \right)^{m_k(\mu)} \right)
       \ ,$$
       where $\ell(\lambda)$ is the length of the partition $\lambda$, $m_k(\lambda)$ the  multiplicity of $k$ in the partition $\lambda$ and 
       $$z_{\lambda} = \prod_k \left( m_k(\lambda)! \ k^{m_k(\lambda)} \right). $$
\end{itemize}

\end{proof}
\noindent {\bf Acknowledgments.} The authors would like to thank R. Rhodes and V. Vargas for the helpful discussions we had on the Gaussian Multiplicative Chaos and the Fyodorov-Bouchaud formula. 

R.C. is supported by the ANR LabEx CIMI (grant number ANR-11-LABX-0040) within the French State Programme ``Investissements d'Avenir'', by the ANR project ``POAS: Probability on Algebraic Structures'', grant number ANR-24-CE40-5511, and by the Institut Universitaire de France (IUF).
\bibliographystyle{alpha}
\bibliography{ChaosCBE}

\newcommand{\etalchar}[1]{$^{#1}$}
\begin{thebibliography}{MRV{\etalchar{+}}16}

\bibitem[APS19]{APS18}
Juhan Aru, Ellen Powell, and Avelio Sep\'ulveda.
\newblock Critical {L}iouville measure as a limit of subcritical measures.
\newblock {\em Electron. Commun. Probab.}, 24:16 pp., 2019.

\bibitem[B{\etalchar{+}}17]{B17}
Nathana{\"e}l Berestycki et~al.
\newblock An elementary approach to {G}aussian multiplicative chaos.
\newblock {\em Electronic Communications in Probability}, 22, 2017.

\bibitem[BFS07]{BFS07}
Jonathan Breuer, Peter~J Forrester, and Uzy Smilansky.
\newblock Random discrete {S}chr{\"o}dinger operators from random matrix theory.
\newblock {\em Journal of Physics A: Mathematical and Theoretical}, 40(5):F161, 2007.

\bibitem[BG06]{BG06}
Daniel Bump and Alex Gamburd.
\newblock On the averages of characteristic polynomials from classical groups.
\newblock {\em Communications in mathematical physics}, 265(1):227--274, 2006.

\bibitem[BHNY08]{BHNY}
Paul Bourgade, Christopher Hughes, Ashkan Nikeghbali, and Marc Yor.
\newblock The characteristic polynomial of a random unitary matrix: {A} probabilistic approach.
\newblock {\em Duke Math. J.}, 145(1):45--69, 2008.

\bibitem[BJRV13]{BJRV13}
Julien Barral, Xiong Jin, R{\'e}mi Rhodes, and Vincent Vargas.
\newblock Gaussian multiplicative chaos and {KPZ} duality.
\newblock {\em Communications in Mathematical Physics}, 323(2):451--485, 2013.

\bibitem[BS12]{BS12}
Andrei~N Borodin and Paavo Salminen.
\newblock {\em Handbook of Brownian motion-facts and formulae}.
\newblock Birkh{\"a}user, 2012.

\bibitem[CD16]{CLD}
Xiangyu Cao and Pierre~Le Doussal.
\newblock Joint min-max distribution and {E}dwards-{A}nderson's order parameter of the circular 1/f-noise model.
\newblock {\em {EPL} (Europhysics Letters)}, 114(4):40003, May 2016.

\bibitem[CMN18]{CMN18}
Reda Chhaibi, Thomas Madaule, and Joseph Najnudel.
\newblock On the maximum of the {$C\beta E$} field.
\newblock {\em Duke Math. J.}, 167(12):2243--2345, 09 2018.

\bibitem[CMV03]{CMV03}
M.~J. Cantero, L.~Moral, and L.~Vel\'azquez.
\newblock Five-diagonal matrices and zeros of orthogonal polynomials on the unit circle.
\newblock {\em Linear Algebra Appl.}, 362:29--56, 2003.

\bibitem[DG10]{DG10}
Charles Dunkl and Stephen Griffeth.
\newblock Generalized {J}ack polynomials and the representation theory of rational {C}herednik algebras.
\newblock {\em Selecta Mathematica}, 16(4):791--818, 2010.

\bibitem[DRS{\etalchar{+}}14]{DRSV14}
Bertrand Duplantier, R{\'e}mi Rhodes, Scott Sheffield, Vincent Vargas, et~al.
\newblock Critical {G}aussian multiplicative chaos: convergence of the derivative martingale.
\newblock {\em The Annals of Probability}, 42(5):1769--1808, 2014.

\bibitem[DS94]{DS94}
Persi Diaconis and Mehrdad Shahshahani.
\newblock On the eigenvalues of random matrices.
\newblock {\em J. Appl. Probab.}, 31A:49--62, 1994.
\newblock Studies in applied probability.

\bibitem[Dur19]{D19}
Rick Durrett.
\newblock {\em Probability: theory and examples}, volume~49.
\newblock Cambridge university press, 2019.

\bibitem[Fal04]{F04}
Kenneth Falconer.
\newblock {\em Fractal geometry: mathematical foundations and applications}.
\newblock John Wiley \& Sons, 2004.

\bibitem[FB08]{FB08}
Y.-V. Fyodorov and J.-P. Bouchaud.
\newblock Freezing and extreme value statistics in a {R}andom {E}nergy {M}odel with logarithmically correlated potential.
\newblock {\em Journal of Physics A: Mathematical and Theoretical}, 41(37), 2008.

\bibitem[GKRV21]{guillarmou2021segal}
Colin Guillarmou, Antti Kupiainen, R{\'e}mi Rhodes, and Vincent Vargas.
\newblock Segal's axioms and bootstrap for {L}iouville {T}heory.
\newblock {\em arXiv preprint arXiv:2112.14859}, 2021.

\bibitem[GKRV24]{guillarmou2020conformal}
Colin Guillarmou, Antti Kupiainen, R{\'e}mi Rhodes, and Vincent Vargas.
\newblock Conformal bootstrap in {L}iouville theory.
\newblock {\em Acta Mathematica}, 233(1):33--194, 2024.

\bibitem[GZ07]{GZ07}
Leonid Golinskii and Andrej Zlatos.
\newblock Coefficients of orthogonal polynomials on the unit circle and higher-order {S}zego theorems.
\newblock {\em Constructive Approximation}, 26(3):361--382, 2007.

\bibitem[JM15]{JM15}
Tiefeng Jiang and Sho Matsumoto.
\newblock Moments of traces of {C}ircular $\beta$ {E}nsembles.
\newblock {\em Ann. Probab.}, 43(6):3279--3336, 11 2015.

\bibitem[JS{\etalchar{+}}17]{JS17}
Janne Junnila, Eero Saksman, et~al.
\newblock Uniqueness of critical {G}aussian chaos.
\newblock {\em Electronic Journal of Probability}, 22, 2017.

\bibitem[KLS98]{KLS98}
Alexander Kiselev, Yoram Last, and Barry Simon.
\newblock Modified {P}r{\"u}fer and {EFGP} {T}ransforms and the {S}pectral {A}nalysis of {O}ne-{D}imensional {S}chr{\"o}dinger {O}perators.
\newblock {\em Communications in Mathematical Physics}, 194(1):1--45, May 1998.

\bibitem[KN04]{KN04}
Rowan Killip and Irina Nenciu.
\newblock Matrix models for circular ensembles.
\newblock {\em Int. Math. Res. Not.}, 1(50):2665--2701, 2004.

\bibitem[KRV20]{kupiainen2020integrability}
Antti Kupiainen, R{\'e}mi Rhodes, and Vincent Vargas.
\newblock Integrability of liouville theory: proof of the dozz formula.
\newblock {\em Annals of Mathematics}, 191(1):81--166, 2020.

\bibitem[Lam21]{L19}
Gaultier Lambert.
\newblock {Mesoscopic central limit theorem for the circular $\beta $-ensembles and applications}.
\newblock {\em Electronic Journal of Probability}, 26(none):1 -- 33, 2021.

\bibitem[Mac98]{macdo}
Ian~Grant Macdonald.
\newblock {\em Symmetric functions and Hall polynomials}.
\newblock Oxford university press, 1998.

\bibitem[MRV{\etalchar{+}}16]{MRV16}
Thomas Madaule, R{\'e}mi Rhodes, Vincent Vargas, et~al.
\newblock Glassy phase and freezing of log-correlated {G}aussian potentials.
\newblock {\em The Annals of Applied Probability}, 26(2):643--690, 2016.

\bibitem[NSW20]{NSW18}
Miika Nikula, Eero Saksman, and Christian Webb.
\newblock Multiplicative chaos and the characteristic polynomial of the {CUE}: the {L}1-phase.
\newblock {\em Transactions of the American Mathematical Society}, 373(6):3905--3965, 2020.

\bibitem[Pis16]{P16}
Gilles Pisier.
\newblock {\em Martingales in Banach spaces}, volume 155.
\newblock Cambridge University Press, 2016.

\bibitem[Pow18]{P18}
Ellen Powell.
\newblock Critical {G}aussian chaos: convergence and uniqueness in the derivative normalisation.
\newblock {\em Electronic Journal of Probability}, 23, 2018.

\bibitem[Rem20]{R17}
Guillaume Remy.
\newblock The {F}yodorov-{B}ouchaud formula and {L}iouville conformal field theory.
\newblock {\em Duke Math. J.}, 169(1):177--211, 2020.

\bibitem[RS80]{RS80}
Michael Reed and Barry Simon.
\newblock Functional analysis, vol. i, 1980.

\bibitem[RV14]{RV13}
R{\'e}mi Rhodes and Vincent Vargas.
\newblock {Gaussian multiplicative chaos and applications: A review}.
\newblock {\em Probability Surveys}, 11(none):315 -- 392, 2014.

\bibitem[RY99]{RevuzYor}
Daniel Revuz and Marc Yor.
\newblock {\em Continuous martingales and {B}rownian motion}.
\newblock Springer, Berlin, third edition, 1999.

\bibitem[Sim05a]{Sim05-1}
Barry Simon.
\newblock {\em Orthogonal polynomials on the unit circle. {P}art 1}, volume~54 of {\em American Mathematical Society Colloquium Publications}.
\newblock American Mathematical Society, Providence, RI, 2005.
\newblock Classical theory.

\bibitem[Sim05b]{Sim05-2}
Barry Simon.
\newblock {\em Orthogonal polynomials on the unit circle. {P}art 2}, volume~54 of {\em American Mathematical Society Colloquium Publications}.
\newblock American Mathematical Society, Providence, RI, 2005.
\newblock Spectral theory.

\bibitem[SV07]{SV07}
Daniel~W Stroock and SR~Srinivasa Varadhan.
\newblock {\em Multidimensional diffusion processes}.
\newblock Springer, 2007.

\bibitem[Vir14]{V14}
Balint Virag.
\newblock Operator limits of random matrices.
\newblock {\em Proceedings of the International Congress of Mathematicians, Volume 4, Seoul. ArXiv preprint arXiv:1804.06953}, pages 247--272, 2014.

\bibitem[Web15]{W15}
Christian Webb.
\newblock The characteristic polynomial of a random unitary matrix and {G}aussian multiplicative chaos - {T}he {L}2-phase.
\newblock {\em Electron. J. Probab.}, 20(104):1--21, 2015.

\end{thebibliography}

\end{document}